\documentclass[11pt, english, reqno]{amsart}
\usepackage{amsmath, amsfonts, amssymb, amscd, enumerate, color, graphics, mathtools, wasysym, empheq, stmaryrd, slashed, babel, mathtools, tikz, hyperref}
\usetikzlibrary{matrix,arrows,decorations.markings,positioning} 
\usepackage[normalem]{ulem}		
\usepackage{mathrsfs}
\usetikzlibrary{arrows,decorations.markings,positioning}
\tikzset{inner sep=0pt, node distance=5mm,
  root/.style={circle,draw,minimum size=5pt,thick},
  broot/.style={circle,draw,minimum size=5pt,thick,fill},
  xroot/.style={circle,draw,minimum size=5pt,thick,label=below:$\times$},
  doublearrow/.style={postaction={decorate},   decoration={markings,mark=at position .6 with {\arrow[line width=1.2pt]{>}}},double distance=1.6pt,thick},
  rdoublearrow/.style={postaction={decorate},   decoration={markings,mark=at position .4 with {\arrowreversed[line width=1.2pt]{>}}},double distance=1.6pt,thick},
  rtriplearrow/.style={postaction={decorate},   decoration={markings,mark=at position .4 with {\arrowreversed[line width=1.2pt]{>}}},double distance=2.5pt,thick},
	curvedline/.style={bend=right}
} 
\usepackage{here}
\usepackage{pdflscape}
\usepackage[margin=3cm]{geometry}
\usepackage[graphicx]{realboxes}
\usepackage[figuresright]{rotating}
\theoremstyle{plain}
\newtheorem{theo}{Theorem}
\theoremstyle{definition}
\newtheorem{example}[theo]{Example}
\newtheorem{definition}[theo]{Definition}
\theoremstyle{plain}
\newtheorem{lemma}[theo]{Lemma}
\newtheorem{theorem}[theo]{Theorem}
\newtheorem{corollary}[theo]{Corollary}
\newtheorem{proposition}[theo]{Proposition}
\newtheorem{conjecture}[theo]{Conjecture}
\theoremstyle{definition}

\newtheorem{remark}[theo]{Remark}

\newenvironment{colored}{\color{red}}{}
\newcommand{\bc}{\begin{colored}}
\newcommand{\ec}{\end{colored}}

\renewcommand{\inf}{\mathfrak{inf}}

\newcommand{\beq}{\begin{equation}}
\newcommand{\eeq}{\end{equation}}

\newcommand{\g}{\gamma}

\renewcommand{\l}{\lambda}

\newcommand{\gJ}{\mathfrak{J}}

\newcommand{\bC}{\mathbb{C}}

\newcommand{\bR}{\mathbb{R}}
\newcommand\R{{\mathbb R}}
\newcommand{\bZ}{\mathbb{Z}}

\newcommand{\fJ}{\mathfrak{J}}

\newcommand{\ga}{\mathfrak{a}}

\newcommand{\gc}{\mathfrak{c}}

\renewcommand{\gg}{\mathfrak{g}}
\newcommand{\stab}{\mathfrak{stab}}

\newcommand{\gk}{\mathfrak{k}}

\newcommand{\gp}{\mathfrak{p}}
\newcommand{\gq}{\mathfrak{q}}
\newcommand{\gr}{\operatorname{gr}}
\newcommand{\gs}{\mathfrak{s}}

\newcommand{\gu}{\mathfrak{u}}

\newcommand{\gsl}{\mathfrak{sl}}
\newcommand{\ggl}{\mathfrak{gl}}

\newcommand\GL{\mathrm{GL}}

\renewcommand\sp{\mathfrak{sp}}

\newcommand{\cC}{\mathcal{C}}
\newcommand{\cD}{\mathcal{D}}

\newcommand{\cF}{\mathcal{F}}

\newcommand{\cJ}{\mathcal{J}}
\newcommand{\cK}{\mathcal{K}}
\newcommand{\cL}{\mathcal{L}}
\newcommand{\clL}{\mathscr{L}}
\newcommand{\cM}{\mathcal{M}}
\newcommand{\cN}{\mathcal{N}}

\newcommand{\p}{\partial}

\renewcommand{\square}{\kern1pt\vbox
{\hrule height 0.6pt\hbox{\vrule width 0.6pt\hskip 3pt
\vbox{\vskip 6pt}\hskip 3pt\vrule width 0.6pt}\hrule height0.6pt}\kern1pt}

\DeclareMathOperator\ad{ad}

\renewcommand\Re{\operatorname{\mathfrak{Re}}}

\newcommand{\Hom}{{\operatorname{Hom}}}

\newcommand{\ev}{{\operatorname{ev}}}

\newcommand{\wt}{\widetilde}
\newcommand{\wh}{\widehat}

\newcommand{\be}{\begin{equation}}
\newcommand{\ee}{\end{equation}}

\def\<#1,#2>{\langle\,#1,\,#2\,\rangle}
\newcommand{\arr}{\begin{array}{rlll}}
\newcommand{\ea}{\end{array}}
\newcommand{\bea}{\begin{eqnarray}}
\newcommand{\eea}{\end{eqnarray}}
\newcommand{\bean}{\begin{eqnarray*}}
\newcommand{\eean}{\end{eqnarray*}}

\newcommand{\under}[1]{{\underline{#1}\,}}

\title[$3$-nondegenerate CR manifolds in dimension $7$ (I)]
{On $3$-nondegenerate CR manifolds in dimension $7$ (I):\\ the transitive case}
\address{Boris Kruglikov, Department of Mathematics and Statistics, UiT The Arctic University of Norway, Troms\o\,  9037, Norway}
\email{boris.kruglikov@uit.no}
\address{Andrea Santi, Dipartimento di Matematica,
Universit\'a degli Studi di Roma ``Tor Vergata'',
Via della ricerca scientifica 1, 00133 Roma, ITALY}
\email{asanti.math@gmail.com, santi@mat.uniroma2.it}
\thanks{}
\author{Boris Kruglikov, Andrea Santi}

\begin{document}

\begin{abstract} 
We investigate $3$-nondegenerate CR structures in the lowest possible dimension $7$, and one of our goals
is to prove Beloshapka's conjecture on the symmetry dimension bound for hypersurfaces
in $\mathbb C^4$. We claim that $8$ is the maximal symmetry dimension of
$3$-nondegenerate CR structures in dimension $7$, which is achieved on the homogeneous model. 
This part (I) is devoted to the homogeneous case: we prove that the model is locally the only 
homogeneous $3$-nondegenerate CR structure in dimension $7$.
\end{abstract}

\maketitle
\null \vspace*{-.50in}

\tableofcontents

\section{Introduction}
\setcounter{equation}{0}\setcounter{section}{1}

An {\it almost CR-structure} on a connected manifold $\cM$ is a subbundle $\cD\subset T\cM$ of the tangent bundle, 
called the CR-distribution, endowed with a field of complex structures
$\cJ\in\Gamma(\operatorname{End}(\cD))$. 
The complexified CR-distribution splits as the direct sum $\cD\otimes\mathbb C=\cD_{10}\oplus \cD_{01}$ of its holomorphic and antiholomorphic parts, where
 $$
\cD_{10}=\{X-i\cJ X\mid X\in \cD\},\;\;\cD_{01}=\{X+i\cJ X\mid X\in \cD\}.
 $$
The almost CR-structure $(\cM,\cD,\cJ)$ is called integrable, or a CR-structure, if the distribution $\cD_{10}$ (or equivalently the distribution
$\cD_{01}=\overline{\cD_{10}}$) is involutive.
In this paper we consider only {\it CR-hypersurfaces}, i.e., integrable CR manifolds 
of CR-codimension equal to $1$, and  denote by $n=\frac12(\dim \cM-1)=\operatorname{rank}_{\mathbb C}(\cD)$
the CR-dimension.

We are interested in the infinitesimal symmetry algebra $\gg$ of the CR structure $(\cM,\cD,\cJ)$. We will work in the analytic category, so 
all objects are real-analytic, and $\gg$
can consist of infinitesimal analytic transformations on the entire $\cM$ or defined on a fixed domain $\mathfrak U\subset\cM$, but it can also be the space of germs around a fixed point $x\in\cM$  of the infinitesimal analytic automorphisms of this structure. In what follows we can adapt any of these notions, and we interchangeably use either
the notation $\gg=\inf(\cM,\cD,\cJ)$ or $\gg=\inf(\cM,\cD,\cJ;x)$, if we wish to emphasize locality around $x\in\cM$. We also remark that $(\cM,\cD,\cJ)$ can be assumed {\it regular}, i.e., the rank of all involved bundles are constant.  In fact, they are constant on an open dense subset of $\cM$ due to upper-semicontinuity, and we may restrict to it by analyticity. For locally homogeneous CR manifolds, our results automatically
hold in the smooth situation, see \S\ref{sec:2.1}. We recall that $(\cM,\cD,\cJ)$ is locally homogeneous at $x\in \cM$ if there exists
a finite-dimensional Lie subalgebra $\gg'$ of $\gg=\inf(\cM,\cD,\cJ;x)$ such that the evaluation map $\ev_x:\gg'\rightarrow T_x\cM$ at $x\in \cM$ is
surjective. We say $(\cM,\cD,\cJ)$ is locally homogeneous if it is so at every point $x\in\cM$. Locally homogeneous CR manifolds will always be considered up to local CR-equivalence. If $(\cM,\cD,\cJ)$ admits a transitive action of a finite-dimensional Lie group by CR diffeomorphisms, then it is called globally homogeneous or, for brevity, homogeneous.

The {\it Levi form\/} of a CR-hypersurface is the tensor  
$\clL:\cD_{10}\otimes\cD_{01}\longrightarrow\nu_\cM^\mathbb C$,
 % (sometimes in $\Lambda^2\cD_{10}^*\otimes\nu_\cM^\mathbb C$ via conjugation)
where $\nu_\cM^\mathbb C=(T\cM/\cD)\otimes\mathbb C$ is the complexified normal bundle, 
given by the formula 
 $$
\clL(Z,\overline{Z'})=i[Z,\overline{Z'}]\!\!\!\mod\cD\otimes\mathbb C\;.
 $$ 
In the case the Levi form is nondegenerate, and identifying locally $\nu_\cM=T\cM/\cD$ with $\R$, 
this is a Hermitian form on the CR-distribution defined up to a real scalar multiple at each point.

As shown in classical works \cite{C,CM,Ta1,Ta2}, the dimension of the symmetry algebra
of a Levi-nondegenerate connected CR-hypersurface $\cM$ of CR-dimension $n$ 
does not exceed $n^2+4n+3$. If $\dim\inf(\cM,\cD,\cJ)$ attains this bound then $\cM$ is {\it spherical}, 
i.e., locally CR-equivalent to an open subset of the hyperquadric, which can be written as the tube
 \begin{equation*}
{\mathcal Q}^{2n+1}_{(p)}=\Bigl\{z=(z_0,\dots,z_n)\in\mathbb C^{n+1}:x=\Re(z)\text{ satisfies }
x_0=x_1^2+\dots+x_p^2-x_{p+1}^2-\dots-x_n^2\Bigr\},
 \end{equation*}
for some $0\le p\le n/2$. 

In the absence of Levi-nondegeneracy, finding the maximal dimension of the symmetry algebra is much harder. 
It is known \cite[\S12.5]{BER} that $\gg=\inf(\cM,\cD,\cJ)$ is finite-dimensional
provided that $\cM$ is {\it holomorphically nondegenerate}. 
Moreover, under the above regularity assumption on the constancy of ranks
of the involved bundles, this is equivalent to
$k$-nondegeneracy for some $1\le k\le n$, see \cite[\S11.1-11.3]{BER}, where $k$ coincides therein with the Levi number, see also \cite[Appendix]{KZ}. (The case $k=1$ corresponds to Levi-nondegeneracy.) We recall this notion in \S\ref{sec:2.1} in relation to the Freeman filtration \cite{Fr}, but presently notice that a
Levi-degenerate $(\cM,\cD,\cJ)$ has a Cauchy characteristic distribution $\operatorname{Ch}(\cD)\subset \cD$ that is independent of $\cJ$,
and that the $2$- and higher nondegeneracy conditions measure a failure of straightening of the distribution $\operatorname{Ch}(\cD)_{10}$ and its subfiltrands. From $3$-nondegeneracy on,
the Freeman sequence relies on the complex structure $\cJ$. 

Unless otherwise specified, we refer to a CR manifold $(\cM,\cD,\cJ)$ as a ``model'' in a class of CR manifolds (for instance, the CR-hypersurfaces of a certain fixed dimension and order of $k$-nondegeneracy, and, if required, with constant CR-symbol) if it is locally the most symmetric, i.e., at every point $x\in \cM$ the Lie
algebra $\gg=\inf(\cM,\cD,\cJ;x)$ of germs of infinitesimal CR-automorphisms has
the largest possible dimension in the class. Note that there are other usages of the terminology ``model'' in the literature, as in \cite[p. 305]{BER}, \cite{B1}, and \cite{San}, ours follows the conventions from Cartan geometry.

Regarding the maximal symmetry dimension in CR geometry,
the following conjecture is a variant of Beloshapka's conjecture, cf. \cite[p.~38]{B2}. 

 \begin{conjecture}\label{beloshapka}
For any real-analytic and connected holomorphically nondegenerate CR-hyper\-surface $(\cM,\cD,\cJ)$ of CR-dimension $n$ one has $\dim\gg\le n^2+4n+3$, with the maximal value $n^2+4n+3$ attained only if on a dense open set $\cM$ is spherical.
 \end{conjecture}

For $n=1$ the above conjecture is true since a 3-dimensional holomorphically nondegenerate CR-hypersurface always has points of Levi-nonde\-generacy. For $n=2$ the conjecture was established in \cite{IZ,MS}, where the proof relied on the reduction of 5-dimensional uniformly $2$-non\-degenerate CR-structures to absolute parallelisms. See also \cite{Eb1, MP}. 
A CR manifold in such class with symmetry algebra of maximal dimension is locally equivalent
to the tube over the future light cone
 \begin{equation}\label{5D}
{\mathcal C}^5=\Bigl\{z=(z_0,z_1,z_2)\in\mathbb C^3:x=\Re(z)\text{ satisfies }
x_0^2=x_1^2+x_2^2,\;x_0>0\Bigr\}.
 \end{equation}
This model is locally homogeneous with symmetry algebra $\gg=\mathfrak{so}(2,3)\cong \mathfrak{sp}_4(\mathbb R)$, however it is
not homogeneous. There are precisely two globally homogeneous models with automorphism
groups of maximal dimension, namely $G/H$ with $G=SO(2,3)^o$ (the connected component of the unity of $SO(2,3)$) and $G=Sp_4(\mathbb R)$, respectively, and $H\cong GL_1(\mathbb C)\ltimes \mathrm{Heis}(3)$, see \cite[Corollary 5.5]{IZ}.

Further results in this direction for $n=1$ (sphericity of $\cM$ near $x$ if $\dim\inf(\cM,\cD,\cJ;x)>5$) and $n=2$ 
(sphericity of $\cM$ near $x$ if $\dim\inf(\cM,\cD,\cJ;x)>11$) are contained in \cite{KS} and \cite{IK} respectively.
(For related results in the case $n>2$ see \cite{Kr}.)

Our goal is to prove Beloshapka's conjecture for $n=3$. In fact, this was recently done by V.\ Beloshapka
himself \cite{B2}, using the homological technique of Poincar\'e, but our approach is quite different
and it leads to finer results. In particular, we completely answer the questions on  
homogeneous $3$-nondegenerate 
CR manifolds in dimension $7$ posed by V.\ Beloshapka to the second author, during Vitushkin's seminar 
at Moscow State University in April 2021: determine whether there exist homogeneous $3$-nondegenerate 
CR manifolds in dimension $7$ besides the examples  
known at that time \cite{Moz, FK2, San} 
and settle the existence of 
3-nondegenerate
CR manifolds with a stabilizer subalgebra of large dimension.

The proof in \cite{B2} goes as follows. If there exists a point $x\in \cM$ of Levi-nondegeneracy on $\cM$, 
then $\dim\inf(\cM,\cD,\cJ;x)\leq24$. If $\cM$ is uniformly Levi-degenerate
but 2-nondegenerate then $\dim\inf(\cM,\cD,\cJ;x)\leq17$ and, finally, if $\cM$ is 
3-nondegenerate then $\dim\inf(\cM,\cD,\cJ;x)\leq20$.

In fact, for 2-nondegenerate $(\cM,\cD,\cJ)$ satisfying a constancy assumption on CR-symbols the situation can be improved due to the results of \cite{PZ,SZ}:
a combination of those implies $\dim\inf(\cM,\cD,\cJ;x)\leq16$ and this bound is sharp.
This includes the classification of CR-symbols, indeed there are more $2$-nondegenerate models than those in dimension $5$,
i.e., \eqref{5D}. 

We are going to significantly improve the dimension bound for the $3$-nondegenerate CR structures in dimension $7$ from \cite{B2}.
Namely, the first of our main results is: 
\begin{theorem}\label{T1}
Let $(M ,\cD,\cJ)$ be a 3-nondegenerate 7-dimensional real-analytic connected CR-hypersurface. Assume 
the symmetry algebra $\gg$ acts with an open orbit.\footnote{This assumption is not essential and can be removed in light of our follow-up paper \cite{KS2} devoted to the intransitive case. However, since this paper is devoted to the locally transitive case, we include it here.}
Then $\dim\gg\leq8$ and 
this bound is sharp.
 \end{theorem}

Our proof uses various techniques from differential geometry, Lie theory, different notions of Tanaka prolongations, etc.
In particular we treat in a completely different manner the cases where the symmetry algebra $\gg$ acts locally transitively on $\cM$
(there exists an open orbit) and intransitively (there exist local invariants for the action). 
Due to this reason, we split the presentation of our results in two separate works, cf. \cite{KS2}. 
This part (I) is dedicated to the locally transitive case. By this we mean that $\cM$ has an open subset $\mathfrak U$, 
where the structure is locally homogeneous -- since restriction to an open subset does not reduce the symmetry
dimension, we lose no generality in assuming $\cM$ itself to be locally homogeneous. 

In \cite{San} several results about homogeneous CR-hypersufaces were obtained, in particular, this concerned 
3-nondegenerate CR structures in dimension $7$. 
We will show that the abstract model in the sense of  \cite{San} derived in loc.cit.\ is also a tube, namely, it can be realized as follows
 \begin{align}
\label{7D}
{\mathcal R}^7=\Bigl\{z=&\,(z_0,z_1,z_2,z_3)
\in\mathbb C^4:x=\Re(z)\text{ satisfies }\notag\\
&x_0=r^3 u - 3r^2 s t,\ x_1=r^2 s u + (r^3\,–\,2 r s^2) t,\ x_2=r\,s^2 u - (s^3 - 2 r^2 s)\, t,\\
&x_3=s^3 u + 3 r s^2 t,\;\text{for real}\;r,s,t,u\;\,\text{s.t.}\;\,(r,s)\neq(0,0),\, t\neq0\notag
\Bigr\}
\end{align}
In other words, ${\mathcal R}^7=\Sigma^3\times i\R^4\subset\bC^4$, where $\Sigma^3$
is described by equations \eqref{7D} on $x\in\mathbb R^4$. Note that this description is not a parametrization, 
as 4 parameters are involved for a 3-dimensional surface. Actually all fixed $u\neq0$ yield the same 
open dense subset in $\Sigma^3\subset\R^4_x$, complementary to the codimension 1 subset given by $u=0$.
Alternatively, $\Sigma^3$ is covered by two charts, parametrized by
$x_0=r^3,\ x_1=r^2(s+t),\ x_2=r\,s\,(s+2t),\ x_3=s^2(s+3t)$, with $r,t\neq0$, and
$x_0=r^2(r-3t)$, $x_1=r\,s\,(r-2t)$, $x_2=s^2(r-t)$, $x_3=s^3$, with $s,t\neq0$.
One can prove that the CR hypersurface \eqref{7D} coincides with the example from the end of \cite[\S 5.1]{FK2}.
%x_0^2x_3^2-6x_0x_1x_2x_3+4x_0x_2^3+4x_1^3x_3-3x_1^2x_2^2= 0

The symmetry algebra of ${\mathcal R}^7$ is generated by the affine symmetry algebra
of $\Sigma^3\subset\R^4_x$ and the translations along the subspace $\R^4_y$, namely it is $\mathfrak{gl}_2(\R)\ltimes\R^4$, where $\R^4=S^3\R^2$.

We can compute the automorphism group as well. Moreover, we will develop further the 
%locally homogeneous 
technique from \cite{San} and exploit Lie theory to arrive at our second main result, 
the classification of $3$-nondegenerate locally homogeneous CR manifolds in dimension 7:

 \begin{theorem}\label{T2}
There exists a unique locally homogeneous $3$-nondegenerate CR structure in dimension $7$, 
and it is  locally isomorphic to the model ${\mathcal R}^7$. Furthermore:
 \begin{enumerate}
\item 
This model is globally homogenous, i.e.,
the automorphism group $G$ acts transitively, and 
$G\cong\operatorname{GL}_2(\mathbb R)\ltimes S^3 \mathbb R^2$;
\item
Every connected globally homogenous $3$-nondegenerate CR manifold in dimension $7$ with  automorphism group of maximal dimension is isomorphic 
to a finite or countable covering of ${\mathcal R}^7$. The automorphism group is the universal cover
 $$\widetilde{G}\cong\widetilde{\operatorname{GL}_2(\mathbb R)}\ltimes S^3 \mathbb R^2$$ of $G$
or a quotient by a subgroup $m\bZ$ of the discrete normal subgroup $\bZ=\pi_1(G)$ of $\widetilde G$.
\end{enumerate}
\end{theorem}
 
We note that $G$ is not connected (but it is Zariski connected, and a posteriori we see that symmetries are algebraic) 
while ${\mathcal R}^7$ is connected, and the action of $G$ on ${\mathcal R}^7$ is effective. The classification of globally homogeneous $3$-nondegenerate CR manifolds in dimension $7$ is then finished in \cite{KS2}, see part (1) of Theorem 3 therein.

\smallskip

There is a growing interest in uniformly Levi-degenerate CR structures, in particular, many recent
papers marked progress for 2-nondegenerate ones \cite{Gr,PZ,SZ}. In \cite{Lo} a question was raised on classifying Levi-degenerate
locally homogeneous CR structures in dimension 7, as the next step after the celebrated 
article \cite{FK2}. Theorem \ref{T2} finishes this classification problem in the 3-nondegenerate case.
%there is one homogeneous 3-nondegenerate simply connected geometry
In the 2-nondegenerate case, there are nine locally homogeneous models (w.r.t. the so-called abstract reduced modified symbols), see \cite{Syk}, but there are uncountably many locally homogeneous CR structures and so far the classification is incomplete.

\vskip0.15cm\par\noindent
{\it Structure of the paper.} 
In \S\ref{sec:2}  we recall the notions of the Freeman sequence, CR algebra and universal CR algebra, 
and establish some relations of the Freeman filtration with another natural filtration 
on the Lie algebra of infinitesimal CR symmetries.
 % of a CR manifold of hypersurface type that is holomorphically nondegenerate. 
The following \S\ref{sec:3} deals with the interplay of the filtrations introduced in \S\ref{sec:2} and the associated graded Lie algebras. It is the most general and technical part of the paper, and its results might be of independent interest. It finishes with a subsection on $k$-nondegenerate homogeneous models in the sense of \cite{San}. 
In \S\ref{sec:4}-\ref{sec:3.3} we specialize to the $7$-dimensional case and, with a 
careful analysis of all the possibilities for the CR algebra, establish the main results, 
via Theorem \ref{thm:dichotomy-notreally} and global topological considerations. Then in
\S\ref{sec:5} we further describe the maximal symmetric model ${\mathcal R}^7$,
provide a characterization of it in terms of the rational normal curve of degree $3$, 
and relate it to the geometry of 4th order scalar ODEs. 
We also consider some generalizations.

\vskip0.15cm\par\noindent
{\it Notations.} 
For a real vector space $V$ we set $V_\times=\left\{v\in V\,|\,v\neq0\right\}$ and frequently $\wh V=V\otimes \bC$; for a bundle $\mathcal K$ on $\cM$, we denote the space of its sections by $\underline{\mathcal K}$. We decompose any section $X$ of $\widehat\cD=\cD\otimes\bC$ into the sum $X=X_{10}+X_{01}$ of its holomorphic $X_{10}\in\underline\cD_{10}$ and antiholomorphic $X_{01}\in\underline\cD_{01}$ components. 
In \S\ref{sec:4}, we will make contact with the notation in \cite{San} and the symbols $M,N,L$, etc., 
will denote basis elements of Lie algebras as in loc.cit. 

\vskip0.15cm\par\noindent
{\it Acknowledgments.} 
The research leading to these results has received funding from the Norwegian Financial Mechanism 2014-2021 (project registration number 2019/34/H/ST1/00636) and the Troms\o{} Research Foundation (project ``Pure Mathematics in Norway''). The first author was also supported by the UiT Aurora project MASCOT. The second author is supported by a tenure-track RTDB position and 
acknowledges the MIUR Excellence Department Project MatMod@TOV
awarded to the Department of Mathematics, University of Rome Tor Vergata, CUP E83C23000330006.  This article/publication was also supported by the ``National Group for
Algebraic and Geometric Structures, and their Applications'' GNSAGA-INdAM (Italy) and it is
based upon work from COST Action CaLISTA CA21109 supported by COST
(European Cooperation in Science and Technology), 
{\scriptsize\url{https://www.cost.eu}}.
\vskip0.1cm\par\noindent
\section{Freeman filtration and universal CR algebras}\hfill\par\label{sec:2}
\setcounter{section}{2}\setcounter{equation}{0}

\subsection{The filtration of Freeman}\label{sec:2.1}

Let $(\cM, \cD, \cJ)$ be a $2n+1$-dimensional CR manifold of hypersurface type.
The associated {\it  Freeman sequence} \cite{Fr}
is the decreasing filtration of sheaves of complex vector fields 
$\under\cF^{-1}\supset \under\cF^{0}\supset \under\cF^{1} \supset \cdots\supset \under\cF^{p-1}\supset\under\cF^{ p}\supset\under \cF^{ p+1}
\supset\cdots$
%\begin{equation*}
%\label{eq:Freeman-filtration}
%\end{equation*}
iteratively defined by $\under\cF^{-1}=\underline{\widehat\cD}$ and
\begin{equation*}
\begin{array}{l}
\under \cF^{p}=\under \cF^{p}_{10}\oplus \under \cF^{p}_{01}
\;,\quad \text{where} 
\quad  \under{\cF}^{p}_{01}= \overline{\under \cF^{p}_{10}}\quad\text{and}
\\[5pt]
\!\!\under{\cF}^{p}_{10}=\left\{X\in \under{\cF}^{p-1}_{10} \mid [X, \under{\cD}_{01}] = 0^{\vphantom{P^c}}
\!\!\!\!\mod  \under{\cF}^{p-1}_{10} \oplus \under{\cD}_{01} \right\}\,.
\end{array}
\end{equation*}
Each term $\under\cF^p$ of the sequence is a $\cC^{\infty}(\cM)$-module and, for $p\geq 0$,
also a Lie subalgebra of the Lie algebra $\underline{\widehat{T\cM}}$. Clearly $\under\cF^{-1}$ consists of the complexified sections of $\cD$ whereas  $\under{\cF}^{p}_{10}$
for $p\geq 0$ coincides with the left kernel of the higher order Levi form
\begin{equation}
\label{eq:h.o.l.f.}
\begin{aligned}
\clL_{p+1}:&\;\; \under{\cF}^{p-1}_{10}\otimes_{\cC^{\infty}(\cM)} \under{\cD}_{01}\longrightarrow (\under{\cF}^{p-2}_{10}\oplus\under{\cD}_{01})/ (\under{\cF}^{p-1}_{10}\oplus\under{\cD}_{01})\\
&\;(X,Y)\longrightarrow [X,Y]\!\!\!\mod \under{\cF}^{p-1}_{10}\oplus\under{\cD}_{01}\,.
\end{aligned}
\end{equation}
For this definition to work in the case $p=0$, we have to understand $\under{\cF}^{p-2}_{10}\oplus\under{\cD}_{01}$ just as $\underline{\widehat{T\cM}}$.
We note that \eqref{eq:h.o.l.f.} is a tensorial map.

\begin{remark}\hfill
\begin{enumerate}
\item Sometimes it is convenient to extend the domain of definition of \eqref{eq:h.o.l.f.}
allowing for two antiholomorphic entries, that is, to the space of sections $\big(\under{\cF}^{p-1}_{10}\oplus \under{\cD}_{01}\big)\otimes_{\cC^{\infty}(\cM)} \under{\cD}_{01}$. This is the trivial extension, so the resulting map will be denoted by the same symbol.
	\item It is an easy induction to see that $[\under{\cF}^{p-1}_{10}\oplus \under{\cD}_{01},\under{\cF}^0_{01}]\subset \under{\cF}_{10}^{p-1}\oplus \under{\cD}_{01}$
for all $p\geq 0$, so we have in fact a map
$\clL_{p+1}:\big(\under{\cF}^{p-1}_{10}\oplus \under{\cD}_{01}\big)\otimes_{\cC^{\infty}(\cM)} (\under{\cD}_{01}/\under{\cF}_{01}^0)\longrightarrow\under{\cF}^{p-2}_{10}/ \under{\cF}^{p-1}_{10}$.
\end{enumerate}
\end{remark}
From now on, we shall assume that $(\cM, \cD, \cJ)$ is {\it regular}, i.e., the vector fields in $\under{\cF}^{p}_{10}$ are the sections of an associated distribution $\cF^{p}_{10}$, and consider the subbundles 
$\cF^p=\cF^p_{10}\oplus\cF^{p}_{01}$ of $\widehat\cD$. 
The bundle $\cF^p$ is stable under conjugation, so it is the complexification of the underlying
real subbundle $\Re(\cF^p)=\Re(\cF^p_{10}\oplus\cF^{p}_{01})$
 of $\cD$. We will also work in the real-analytic category.

The focus of this first part (I) is on homogeneous CR manifolds and our results automatically
holds in the smooth situation as well, due to the following general facts: If $\cM$ is a
smooth CR manifold that is locally homogeneous under a finite-dimensional Lie algebra $\gg$
of smooth infinitesimal CR automorphisms, then it is regular (because the fibers at different points of a given characteristic bundle correspond each other via local CR automorphisms) and there is a real-analytic atlas on
$\cM$ such that all vector fields in $\gg$ become real-analytic (because any real Lie group admits a unique compatible real-analytic structure, see for instance \cite{Pon}).
%The CR structure is then analytic as well (it is moved from a point with local analytic transformations)
\begin{definition}\cite{BER, Fr}  The CR manifold $(\cM, \cD, \cJ)$ is {\it  $k$-nondegenerate at $x\in \cM$} if   $\cF^p|_x \neq 0$ for all $-1 \leq$ $p$ $\leq k-2$  and $\cF^{k - 1}|_x= 0$.
\end{definition}

\subsection{Locally homogeneous CR manifolds and CR algebras}\label{sec:2.2}

Let $(\cM, \cD, \cJ)$ be a CR manifold of hypersurface type, or a germ of it at a fixed point $x\in \cM$, which is locally homogeneous $k$-nondegenerate,  and $\gg$ the associated Lie algebra of infinitesimal CR automorphisms. 
Transitivity of the action amounts to
 $T_x\cM=\left\{\ev_x(\xi)\mid
\xi\in \gg\right\}$, where $\ev_x:\gg\rightarrow T_x\cM$ is the evaluation map at $x\in \cM$.
(Sometimes we will use the simpler notation $\xi|_x$ instead of $\ev_x(\xi)$. Furthermore, since any bundle $\mathcal K=\cup_{x\in\cM} \mathcal K|_x$ is the union of its fibers over $\cM$, we may write shorter expressions like ``$\ev_x(\xi)\in\cK$'' instead of the logically equivalent ones ``$\ev_x(\xi)\in\cK|_x$'').

We recall that set
$$
\gq=\Big\{\xi\in\wh\gg\mid\ev_x(\xi)\in \cD_{10}\Big\}
$$
and note that
\begin{itemize}
\item[(i)] $\gg$ is a real Lie algebra;
\item[(ii)] $\gq$ is a complex subalgebra of $\wh{\gg}$, thanks to the integrability condition of CR manifolds;
\item[(iii)] the quotient $\gg/\mathfrak{stab}$ is finite-dimensional, where $$\mathfrak{stab}=\gg\cap\gq=\Re(\gq\cap\overline\gq)=\left\{\xi\in\gg\,|\!\!\!\!\!\!\!\!\!\!\!\phantom{C^{C^C}}\ev_x(\xi)=0\right\}\ $$
is the stabilizer subalgebra at $x$.
\end{itemize}
The pair $(\gg,\gq)$ is called an abstract CR algebra, in the terminology of \cite{MeNa3}.

Conversely any abstract CR algebra determines a unique germ of a locally homogeneous CR manifold $(\cM,\cD,\cJ; x)$ with $T_x\cM\cong \gg/\mathfrak{stab}$ and $\cD_{10}|_x\cong\gq/(\gq\cap\bar\gq)$, see \cite{Fel, MeNa3} for more details. The Freeman bundles are locally homogeneous bundles with fiber $\cF^{p}_{10}|_{x}\cong \gq^{p}/(\gq\cap\overline\gq)$, where $\gq^{-1}=\gq\supset \gq^0\supset \cdots\supset \gq^{p-1}\supset \gq^p\supset \gq^{p+1}\supset\cdots
\supset \gq\cap\overline\gq$ is the nested sequence of complex subalgebras 
of $\gq$ iteratively defined by (see \cite{Fel}):
$$
\begin{aligned}
\notag\gq^{p}&=\left\{\xi\in \wh{\gg}\,\mid\!\!\!\!\!\!\!\!\!\!\!\phantom{C^{C^C}} \ev_x(\xi)\in \cF^{p}_{10}\right\}\\
&=\left\{\xi\in\gq^{p-1}\,\mid\!\!\!\!\!\!\!\!\!\!\!\phantom{C^{C^C}}[\xi,\overline\gq]\subset \gq^{p-1}+\overline\gq\right\}\,.
\end{aligned}
$$
\begin{remark}
\hfill
\begin{enumerate}
\item
Parallel to the fact that $\under\cF^p$ is a Lie subalgebra of $\underline{\widehat{T\cM}}$ for all $p\geq 0$, we have that $\Re(\gq^p+\overline\gq^p)=\Big\{\xi\in\gg\mid\ev_x(\xi)\in\Re(\cF^p)\Big\}$ is a Lie subalgebra of $\gg$ for all $p\geq 0$. 
\item
We stress that
the filtration $\{\gq^p\}$ is not respected by the Lie brackets in general. For instance $[\gq^1,\gq^1]$ is not included in $\gq^2$ for the CR algebra $(\gg,\gq)$ of Example \ref{ex:4nondeg} later on. 
\end{enumerate}
\end{remark}
One has $k$-nondegeneracy when $\gq^{k-2}\neq\gq\cap\overline\gq$ and $\gq^{k-1}=\gq\cap\overline\gq$, as it can be readily seen from the following simple but useful lemma.
\begin{lemma}
\label{h.o.l.-with-symmetries}
If $\xi\in\gq^{p-1}+\overline\gq$ and $\eta\in\overline\gq$, then 
$\clL_{p+1}(\xi|_x,\eta|_x)=-[\xi,\eta]|_x\!\!\mod (\cF^{p-1}_{10}\oplus\cD_{01})|_x$.
\end{lemma}
\begin{proof}
Let $X\in \under{\cF}^{p-1}_{10}$ and $Y,Z\in \under{\cD}_{01}$ such that $X|_x=(\xi|_x)_{10}$, $Y|_x=(\xi|_x)_{01}$, $Z|_x=\eta|_x$, and compute
$
[\xi-X-Y,\eta-Z]=[\xi,\eta]-[\xi,Z]-[X+Y,\eta]+[X+Y,Z]
$.
Now $[\xi-X-Y,\eta-Z]|_x=0$ since both vector fields vanish at $x$, while $[\xi,Z]\in\under{\cD}_{01}$ and $[X+Y,\eta]\in \under{\cF}^{p-1}_{10}\oplus \under{\cD}_{01}$ since $\xi$ and $\eta$ are CR symmetries.
\end{proof}

\subsection{A filtration on the Lie algebra of infinitesimal CR automorphisms}
\label{subsec:2.3}
Let  $\gg$ be as in \S\ref{sec:2.2}. Following \cite{San}, we now introduce a novel filtration on $\gg$, cf. also \cite{MMN}. It is given by
$$\cdots\supset \gg^{q-1}\supset\gg^{q}\supset\gg^{q+1}\supset\cdots\supset\gg^{-1}\supset \gg^{0}\supset \gg^{1} \supset \cdots\supset \gg^{p-1}\supset\gg^{p}\supset \gg^{p+1}
\supset\cdots$$
with 
\begin{equation}
\label{eq:filtration}
\begin{aligned}
\gg^{-1}&=\left\{\xi\in \gg \mid \ev_x(\xi)\in \cD \right\}\\
\gg^q&=\gg^{q+1}+[\gg^{-1},\gg^{q+1}]\\
\gg^p&=\left\{\xi\in \gg^{p-1} \mid [\xi,\gg^{-1}]\subset \gg^{p-1}\right\}
\end{aligned}
\end{equation}
for all $q\leq -2$, $p\geq 0$. We emphasize that \eqref{eq:filtration}  does {\it not} coincide with the traditional filtration on $\gg$ introduced by Weisfeiler and independently by Morimoto and Tanaka to study transitive Lie algebras of vector fields \cite{Weis, MorTan,Ta1}. In general, their terms of non-negative degree do not coincide with ours, since $\gg^0=\mathfrak{stab}$ for them, while $\gg^0=\left\{\xi\in \gg^{-1} \mid [\xi,\gg^{-1}]\subset \gg^{-1}\right\}$ for us, which is generally bigger.
An analogous filtration $\cdots\supset\wh\gg^{p-1}\supset\wh\gg^p\supset\wh\gg^{p+1}\supset\cdots$ is introduced on $\wh\gg$,
and each term $\wh\gg^p$ coincides in fact with the complexification of $\gg^p$. In particular, it is stable under conjugation. 

We note that 
these filtrations do not use the complex structure $\cJ$ anywhere but only $\cD$, and refer to them
as {\it contact filtrations} of the CR algebra. We also emphasize that transitivity and $k$-nondegeneracy are crucial for the following result to hold. 
\begin{proposition}
\label{prop:basic-properties-filtration}
The filtration \eqref{eq:filtration} satisfies the following basic properties
\begin{itemize}
\item[(i)] $[\gg^p,\gg^q]\subset \gg^{p+q}$ for all $p,q\in\bZ$,	
\item[(ii)] $\gg^p=\left\{\xi\in \gg^{0} \mid [\xi,\gg^{-1}]\subset \gg^{p-1}\right\}=\left\{\xi\in \gg^{-1} \mid [\xi,\gg^{-1}]\subset \gg^{p-1}\right\}$ for all $p\geq 0$,
\item[(iii)] $\gg^{-2}=\gg$,
\item[(iv)] $\mathfrak{stab}\subset\gg^0$,
\end{itemize}
and analogous ones hold for the complexified filtration
on $\wh\gg$. Moreover
\begin{itemize}
\item[(v)] $\wh\gg^p=\gq^p+\overline\gq^p$ for $p=-1,0$,
\item[(vi)] $\wh\gg^p\subset\gq^p+\overline\gq^p$ for all $p\geq 1$.
\end{itemize}
\end{proposition}
\begin{proof}
Properties (i)-(ii) follow by a straightforward induction, which we omit. Now, by $k$-nondegeneracy, the Levi form 
$\clL$ is not identically zero and, by transitivity and Lemma \ref{h.o.l.-with-symmetries}, there exist $\xi\in\gq$, 
$\eta\in\overline\gq$ such that $[\xi,\eta]\notin\wh\gg^{-1}$. Since $(\cM, \cD, \cJ; x)$ is of hypersurface type, 
$\wh\gg^{-2}=\wh\gg$, and (iii) holds. Let $\xi\in\stab$. For any $\eta\in\gg^{-1}$ and $X\in\under\cD$ with 
$X|_x=\eta|_x$, we have $[\xi,\eta]|_x=[\xi,X]|_x$, which is in $\cD|_x$ since $\xi$ is a CR symmetry. 
Therefore $\mathfrak{stab}\subset\gg^0$. 

Transitivity gives (v) for $p=-1$. We prove that $\wh\gg^p\subset\gq^p+\overline\gq^p$ for all $p\geq -1$ by induction. Let $\xi\in\wh\gg^p$, $p\geq 0$, and first note that 
$\wh\gg^p\subset\wh\gg^{p-1}\subset
\gq^{p-1}+\overline\gq^{p-1}\subset \gq^{p-1}+\overline\gq$ by the induction hypothesis; 
in particular $\xi|_x=(\xi|_x)_{10}+(\xi|_x)_{01}\in \cF^{p-1}_{10}\oplus \cF^{p-1}_{01}$. 
%and we may choose $X\in \under{\cF}^{p-1}_{10}$ and $Y\in \under{\cF}^{p-1}_{01}$ such that
%$\xi|_x=X|_x+Y|_x$. 
Lemma \ref{h.o.l.-with-symmetries} then says that $[\xi,\eta]|_x\!\mod (\cF^{p-1}_{10}\oplus\cD_{01}  )|_x=
-\clL_{p+1}(\xi|_x,\eta|_x)$ for all $\eta\in\overline\gq$, but
%$\equiv-[X,Z]|_x\!\mod (\cF^{p-1}_{10}\oplus\cD_{01}  )|_x$.
%$Z\in \under{\cD}_{01}$ with $\eta|_x=Z|_x$, we compute
%\vskip0.5cm\par\noindent
%$$
%[\xi-X-Y,\eta-Z]=[\xi,\eta]+[X,Z]-\underbrace{[\xi,Z]}_{\in \under{\cD}_{01}}-\underbrace{[X,\eta]}_{\in \under{\cF}^{p-1}_{10}}-\underbrace{[Y,\eta]}_{\in \under{\cF}^{p-1}_{01}}+\underbrace{[Y,Z]}_{\in\under{\cD}_{01}}\;,
%$$
%\vskip0.2cm\par\noindent
%and evaluating at $x$ yields $[\xi,\eta]|_x\equiv-[X,Z]|_x\!\mod (\cF^{p-1}_{10}\oplus\cD_{01}  )|_x$. 
$[\xi,\eta]\in\wh\gg^{p-1}\subset \gq^{p-1}+\overline\gq^{p-1}$ since $\xi\in\wh\gg^p$, so
$\clL_{p+1}(\xi|_x,\eta|_x)=0$.
%$[X,Z]|_x\in (\cF^{p-1}_{10}\oplus\cD_{01})|_x$
By transitivity, this is equivalent to $(\xi|_x)_{10}\in \cF^{p}_{10}$.
%and therefore $X|_x\in \cF^{p}_{10}$ by transitivity. 
One similarly establishes that 
$(\xi|_x)_{01}\in \cF^{p}_{01}$.
%$Y|_x\in \cF^{p}_{01}$. 
Let now $\xi'\in\gq^p$ and $\xi''\in\overline\gq^p$ such that $\xi'|_x=(\xi|_x)_{10}$ and $\xi''|_x=(\xi|_x)_{01}$. Then 
$\xi-(\xi'+\xi'')$ is an element of the complexified stabilizer subalgebra $\wh\stab=\gq\cap\overline\gq=\gq^p\cap\overline\gq^p$ and the inclusion
$\wh\gg^p\subset\gq^p+\overline\gq^p$ is settled. 

To finish the proof of (v)-(vi), it remains to show that $\gq^0+\overline\gq^0\subset\wh\gg^0$.
Let $\xi\in \gq^0$ and $\eta\in\wh\gg^{-1}$, which we may write as $\eta=\eta'+\eta''$ 
for some $\eta'\in\gq$ and $\eta''\in\overline\gq$. Then
$[\xi,\eta]|_x\!\!\mod \wh\cD|_x=[\xi,\eta'']|_x\!\!\mod \wh\cD|_x=-\clL(\xi|_x,\eta''|_x)=0$,
since $\gq$ is a subalgebra and where we used Lemma \ref{h.o.l.-with-symmetries} with $p=0$. 
%and $X\in\under{\cF}^{0}_{10}$, $Y\in\under{\wh\cD}$ such that $\xi|_x=X|_x$, $\eta|_x=Y|_x$. Then
%\vskip0.5cm\par\noindent
%$$
%[\xi-X,\eta-Y]=[\xi,\eta]-\underbrace{[\xi,Y]}_{\in \under{\wh\cD}}-\underbrace{[X,\eta]}_{\in \under{\cF}^{0}_{10}}+\underbrace{[X,Y]}_{\in \under{\wh\cD}}
%$$
%\vskip0.1cm\par\noindent
%and evaluating at $x$ says $[\xi,\eta]|_x\in \wh\cD$. 
Since $\gq^0\subset \wh\gg^{-1}$, this readily says that $\gq^0\subset\wh\gg^0$, and by conjugation $\gq^0+\overline\gq^0\subset\wh\gg^0$. 
\end{proof}

\begin{remark}\hfill\label{remark:later-on}
\begin{enumerate}
\item Property (v) of Proposition \ref{prop:basic-properties-filtration} does not hold for $p\geq 1$ in general. In fact, if $\wh\gg^p=\gq^p+\overline\gq^p$, then $\wh\gg^p$ includes the stabilizer subalgebra $\wh\stab$, which is typically not the case for $p\geq 1$. 
\item Property (vi) of Proposition \ref{prop:basic-properties-filtration} implies that
\begin{equation}
\label{eq:eval-p}
\wh\gg^p+\wh\stab\subset\gq^p+\overline\gq^p
\end{equation}
for all $p\geq 1$. Requiring the opposite inclusion is equivalent to ask that the evaluation map 
$\ev_x:\wh\gg^p\rightarrow \cF^{p}|_x$ is surjective, but this is in fact not always true. A counterexample in dimension $9$ can be found in Example \ref{ex:4nondeg}.

As a consequence of our classification result in \S\ref{sec:3}-\S\ref{sec:4}, the property $\wh\gg^p+\wh\stab=\gq^p+\overline\gq^p$ holds for the {\it full}  infinitesimal symmetry algebra of all homolorphically nondegenerate locally homogeneous CR manifolds of hypersurface type up to dimension $7$. 
\end{enumerate}
\end{remark}

 \begin{example}\label{ex:4nondeg}
We consider the homogeneous CR manifold $\cM^{k,c}:= \Gamma(a) + iV$ as in \cite[\S 5]{FK2} for $k=4$, $c=1$. It is the tube over the $4$-dimensional 
group orbit $\Gamma(a)\subset V$, where $V=S^4\mathbb R^2$ with the action induced by that of $\Gamma=\GL_2^+(\mathbb R)$ on $\mathbb R^2=\langle u_1,u_2\rangle$, and $a=u_2^4+u_1u_2^3+u_1^2u_2^2$. The CR manifold
$\cM^{4,1}$ is $9$-dimensional, $4$-nondegenerate, and homogeneous for the action of the complex affine group  $G=\Gamma\ltimes iV$. The action is almost effective and simply transitive. 
We consider $x=(a,0)\in \cM^{4,1}$ as base point and let $(\gg,\gq)$ be the associated CR algebra.

Abstractly $\gg=\ggl_2(\mathbb R)\ltimes S^4\mathbb R^2$, where the action  of 
$\ggl_2(\mathbb R)$ on $S^4\mathbb R^2$ is by derivations, that is
\begin{equation*}
\begin{pmatrix}\alpha & \beta\\ \gamma & \delta\end{pmatrix}\mapsto
\begin{pmatrix} 
4\alpha & \beta & 0 & 0 &0 \\
4\gamma & 3\alpha+\delta & 2\beta & 0 & 0\\
0 & 3\gamma & 2\alpha+2\delta & 3\beta & 0\\
0 & 0 & 2\gamma & \alpha+3\delta & 4\beta \\
0 & 0 & 0 & \gamma & 4\delta
\end{pmatrix}
\end{equation*}
w.r.t. the natural basis $\{u_1^4,u_1^3 u_2,u_1^2 u_2^2, u_1 u_2^3,u_2^4\}$ of $S^4\mathbb R^2$ and for any matrix $\begin{psmallmatrix}\alpha & \beta\\ \gamma & \delta \end{psmallmatrix}\in \ggl_2(\mathbb R)$.
The complexification $\wh\gg=\ggl_2(\mathbb C)\ltimes S^4\mathbb C^2$ of $\gg=\ggl_2(\mathbb R)\ltimes S^4\mathbb R^2$ comes with the subalgebra $\mathfrak q$ generated by
\begin{align*}
A&=\begin{pmatrix} 1& 0 \\ 0 & 0\end{pmatrix}-i\big(u_2^3u_1+2u_2^2u_1^2\big)\;,\\
B&=\begin{pmatrix} 0& 1 \\ 0 & 0\end{pmatrix}-i\big(4u_2^3u_1+3u_2^2u_1^2+2u_2u_1^3\big)\;,\\
C&=\begin{pmatrix} 0& 0 \\ 1 & 0\end{pmatrix}-i\big(u_2^4+2u_2^3u_1\big)\;,\\
D&=\begin{pmatrix} 0& 0 \\ 0 & 1\end{pmatrix}-i\big(4u_2^4+3u_2^3u_1+2u_2^2u_1^2\big)\;,
\end{align*}
and with the standard antilinear involution that corresponds to its real form $\gg$.
The coefficients of $-i$ in the above formulae are given by the action of the generators of $\mathfrak{gl}_2(\mathbb R)$ on $a\in S^4\mathbb R^2$, hence the subalgebra $\gq$ is isomorphic to $\ggl_2(\mathbb C)$. 
Its Lie brackets are collected in the following table: 
\par
 \bigskip
\centerline{\rotatebox{0}{\footnotesize
$\begin{array}{|c|c|c|c|c|}\hline %\hline 
{}^{\phantom{A^A}}[-,-]^{\phantom{A^A}} &
A & B& C& D \\
\hline\hline
A & 0 & B & -C & 0
\\
\hline
B & \star & 0 & A-D & B  \\
\hline
C & \star & \star  & 0 & -C \\
\hline
D & \star & \star  & \star & 0 \\
\hline
\end{array}$
}}
\bigskip
\centerline{\small\it %Lie brackets between infinitesimal CR symmetries.
}
\bigskip
\par\noindent
where ``$\star$'' refers to Lie brackets that can be inferred from the others using skew-symmetry. 
We note that
\begin{equation*}
\begin{aligned}
\gq+\overline\gq&=\ggl_2(\mathbb C)\ltimes\langle u_2^4,u_2^3u_1,u_2^2u_1^2,u_2u_1^3\rangle\;,\\
\wh\gg/(\gq+\overline\gq)&\cong\langle u_1^4\rangle\;.
\end{aligned}
\end{equation*}
In particular $\wh\stab=\gq\cap\overline\gq=0$, coherent with the fact that the stabilizer of $x$ in $G$ is  
 $$
\mathbb Z_2\times\mathbb Z_2=\Bigl\{\pm\begin{pmatrix}1 & 0 \\ 0 & 1\end{pmatrix},
\pm\begin{pmatrix} 1 & 0 \\ 1 & -1\end{pmatrix}\Bigr\}\;,
 $$
hence discrete.

The Lie brackets between elements of $\gq$ and $\overline\gq$ are more complicated. They are given by
\begin{align*}
[A,\overline A]&= \tfrac16\overline D - \tfrac 23\overline C + \tfrac{11}{6}\overline A - \tfrac{1}{6}D -\tfrac{11}{6} A + \tfrac{2}{3}C\;,\\
[A, \overline B]&= \tfrac{13}{12}\overline D - \tfrac{13}{3}\overline C + 3\overline B - \tfrac{31}{12}\overline A - \tfrac{13}{12}D + \tfrac{13}{3}C + \tfrac{31}{12}A - 2B\;,\\
[A, \overline C]&= -\tfrac{1}{3}\overline D + \tfrac{4}{3}\overline C + \tfrac{1}{3}\overline A +\tfrac{1}{3} D - \tfrac{1}{3}A - \tfrac{7}{3}C\;,\\
[A, \overline D]& = -\tfrac{1}{6}\overline D + \tfrac23\overline C + \tfrac{13}{6}\overline A + \tfrac16 D - \tfrac{13}{6}A - \tfrac{2}{3}C\;,\\
[B, \overline C]& = -\tfrac{1}{6}\overline D + \tfrac{2}{3}\overline C + \tfrac{19}{6}\overline A - \tfrac{5}{6}D - \tfrac{13}{6}A - \tfrac{2}{3}C\;,\\
[B, \overline D]& = -\tfrac{13}{12}\overline D + \tfrac{13}{3}\overline C + 2\overline B + \tfrac{31}{12}\overline A + \tfrac{13}{12} D - \tfrac{13}{3} C - \tfrac{31}{12} A - B\;,\\
[C, \overline C]& = \tfrac{2}{3}\overline D - \tfrac{2}{3}\overline C - \tfrac{2}{3}\overline A - \tfrac{2}{3} D + \tfrac{2}{3}A + \tfrac{2}{3}C\;,\\
[C, \overline D]& = \tfrac13\overline D + \tfrac{5}{3}\overline C - \tfrac{1}{3}\overline A - \tfrac{1}{3}D + \tfrac{1}{3}A - \tfrac{8}{3}C\;,\\
[D, \overline D]& = \tfrac{25}{6}\overline D - \tfrac{2}{3}\overline C - \tfrac{13}{6}\overline A - \tfrac{25}{6}D + \tfrac{13}{6}A + \tfrac{2}{3}C\;,
\end{align*}
the fact that $[B, \overline B]$ is not an element of $\gq+\overline\gq$ (we omit its explicit expression), and by brackets that can be inferred from the others using skew-symmetry and conjugation.

The terms of the Freeman sequence are given by
\begin{align*}
\gq^{-1}=\gq\supset\gq^0=\langle A,C,D\rangle\supset \gq^1=\langle A-D,C\rangle\supset \gq^2=\langle A-D+C\rangle\supset\gq^3=0\;,
\end{align*}
thus decreasing by one dimension at each step, as expected. We then see that
\begin{align*}
\gq^1+\overline \gq^1&=\Bigl\langle\begin{pmatrix} 1 & 0 \\ 1 & -1 \end{pmatrix}, \begin{pmatrix} 0 & 0 \\ 1 & 0 \end{pmatrix}, u_2^4, u_2^3u_1\Bigr\rangle\;,\\
\gq^2+\overline \gq^2&=\Bigl\langle\begin{pmatrix} 1 & 0 \\ 1 & -1\end{pmatrix}, u_2^4\Bigr\rangle\;.
\end{align*} 

Now $\wh\gg^{-1}=\gq+\overline\gq$, 
$\wh\gg^0=\gq^0+\overline\gq^0=\Bigl\langle \begin{pmatrix} 1 & 0 \\ 0 & 0 \end{pmatrix}, \begin{pmatrix} 0 & 0 \\ 1 & 0 \end{pmatrix}, \begin{pmatrix} 0 & 0 \\ 0 & 1 \end{pmatrix}, u_2^4, u_2^3u_1, u_2^2u_1^2\Bigr\rangle$ 
and $\wh\gg^3=0$ due to Proposition \ref{prop:basic-properties-filtration}. However a direct computation shows that
\begin{align*}
\wh\gg^{1}&=\Bigl\langle\begin{pmatrix} 0 & 0 \\ 1 & 0 \end{pmatrix}, u_2^4, u_2^3u_1\Bigr\rangle\;,\\
\wh\gg^2&=\langle u_2^4\rangle\;,
\end{align*}
and \eqref{eq:eval-p} is a {\it proper} inclusion for $p=1,2$. 
 \end{example}
\smallskip

For any given Lie algebra $\gg$ with a compatible decreasing filtration $\gg^p\subset \gg^q$, $p>q$, it is possible to consider the associated graded Lie algebra $\gr(\gg)=\bigoplus_{q\in\mathbb Z}\gg_q$, where 
$\gg_q=\gg^q/\gg^{q+1}$. By fixing complementary subspaces in the chain of filtrands, which is always possible, $\gg$ can be identified with $\gr(\gg)$ as a  {\it filtered vector space}, i.e., $\gg=\gr(\gg)$ as a vector space and $\gg^q=\oplus_{p\geq q}\gg_p$. Furthermore $\gg$ is a filtered deformation of $\gr(\gg)$, in the sense that the Lie algebra structure of $\gg$ is a variation of that of $\gr(\gg)$ inducing the trivial change at the graded level, see \cite{CK}. More explicitly, the structure equations of $\gg$ are a 
variation of those $[e_i,e_j]=\sum_k c_{ij}^k e_k$ of $\gr(\gg)$ (here $e_i\in \gg_p, e_j\in \gg_q, e_k\in \gg_{p+q}$ are basis elements) by contributions of degree $>p+q$ in the r.h.s. 
If there exists a redefinition of the complementary subspaces in the chain of filtrands of $\gg$
so that such perturbations can be eliminated, then $\gg$ and $\gr(\gg)$ are isomorphic as {\it filtered Lie algebras}. In this case, the filtered deformation 
$\gg$ is called {\it trivial} and we write $\gg\cong \gr(\gg)$.

\begin{remark}
\label{rem:11}
Of course the above point of view can be reversed, taking graded Lie algebras as the starting point and introducing their filtered deformations as variations of their structure equations.
For example, the Heisenberg Lie algebra $\mathfrak{heis}(3)=\gg_1\oplus\gg_2$ with bases $e'_1,e''_1$ of $\gg_1$, $e_2$ of $\gg_2$, and structure equations $[e'_1,e''_1]=e_2$ is a graded Lie algebra that is filtration rigid: there is simply no place for any non-trivial filtered deformation. However, if we change the grading to $\mathfrak{heis}(3)=\gg_{-1}\oplus\gg_{-2}$ with bases $e'_{-1},e''_{-1}$ of $\gg_{-1}$, $e_{-2}$ of $\gg_{-2}$, structure equations $[e'_{-1},e''_{-1}]=e_{-2}$, then 
%non-Abelian $3$-dimensional Lie algebras can be obtained as filtered deformations. For instance, 
the additional relations $[e'_{-1},e_{-2}]=t e'_{-1}$ and  $[e''_{-1},e_{-2}]=-t e''_{-1}$ define the simple Lie algebra $\mathfrak{sl}_2(\mathbb R)$ for any fixed $t\neq 0$. 

Note that the obtained Lie algebra $\mathfrak{sl}_2(\mathbb R)$ becomes in turn graded declaring $e'_{-1}$, $e_{-2}$, $e''_{-1}$ to have degrees $-1$, $0$, and $1$, respectively. This graded Lie algebra is filtration rigid.
\end{remark}
\smallskip

We will introduce in \S\ref{sec:2.5} sufficient conditions for the surjectivity of the map $\ev_x:\wh\gg^p\rightarrow \cF^{p}|_x$ in terms of certain kinds of trivial filtered deformations. This result will be crucial for our main classification in dimension $7$ and, before establishing it, we need to recall the notion of universal CR algebra as introduced in \cite{San}.

\subsection{Universal CR algebra}\label{sec:2.3}

Let $\gc=\bigoplus_{p\geq -2}\gc_{p}$ be the infinite-dimensional contact algebra, i.e, the maximal transitive prolongation, in the sense of Tanaka, of the Heisenberg Lie algebra $\gc_-=\gc_{-2}\oplus\gc_{-1}\cong \mathbb R\oplus\mathbb R^{2n}$ in dimension $n$ \cite{MorTan}. 
There are natural identifications $$\gc_p\cong S^{p+2}(\gc_{-1})\oplus S^{p}(\gc_{-1})\oplus\cdots\;,$$ 
where $\gk_{p}=S^{p+2}(\gc_{-1})$ is the subspace of $\gc_p$ consisting of the elements acting trivially on $\gc_{-2}$. In particular, the zero-degree part is isomorphic to the linear conformal Lie algebra $\gc_0=\gk_0\oplus\bR E$, with $\gk_0\cong \sp(\gc_{-1})$ and $E$ the grading element. Analogous observations are true for $\wh\gc=\gc\otimes\bC$.

Fix  a complex structure $\mathfrak{J}$ on $\gc_{-1}$ such that $\mathfrak{J}\in\gk_0$. 
Clearly $\wh{\gc}_{-1}$  decomposes into the direct sum $\wh{\gc}_{-1}=\gc_{-1(10)}\oplus\gc_{-1(01)}$ of its holomorphic and antihomolorphic parts  and the action of $\mathfrak J$ extends to each symmetric power of $\gc_{-1}$ via the adjoint action. 

For any fixed $p\geq -2$, we denote the  
$\ad(\mathfrak J)$-eigenspace of eigenvalue $i\ell$ in $S^k(\wh\gc_{-1})\subset \wh\gc_p$  by
$$\gc_{(p,\ell)}^{k}=S^{\tfrac{k+\ell}{2}}(\gc_{-1(10)})\otimes S^{\tfrac{k-\ell}{2}}(\gc_{-1(01)})$$
\vskip0.5cm\par\noindent
so that we decompose
$$\displaystyle\wh\gc_p=\bigoplus_{\tiny\begin{gathered}k=p+2,p,p-2,\ldots\\ |\ell|=k,k-2,k-4,\ldots\end{gathered}}\gc_{(p,\ell)}^{k}\;\;\;\;\;\;\text{and}\;\;\;\;\;\;
\displaystyle\wh\gc=\bigoplus_{\tiny\begin{gathered}p=-2,-1,0,\ldots \\ k=p+2,p,p-2,\ldots\\ |\ell|=k,k-2,k-4,\ldots\end{gathered}}\gc_{(p,\ell)}^{k}\;.$$
\vskip0.2cm\par\noindent
The lower indices $(p,\ell)$ in $\gc_{(p,\ell)}^{k}$ define a $\mathbb Z$-bigrading of the contact algebra $\wh\gc$ that is compatible with the Lie algebra structure. 
The upper index $k$ indicates the symmetric power $S^k(\wh\gc_{-1})\subset \wh\gc_p$.

More is true, however, as it follows directly from the explicit expressions of the Lie brackets of $\wh\gc$ as obtained in \cite[Prop. 3.2]{San}:
\begin{proposition}
\label{prop:quasi-grading}
The inclusion $[\gc_{(p_1,\ell_1)}^{k_1},\gc_{(p_2,\ell_2)}^{k_2}]\subset \gc_{(p_1+p_2,\ell_1+\ell_2)}^{k_1+k_2}\oplus \gc_{(p_1+p_2,\ell_1+\ell_2)}^{k_1+k_2-2}$ holds for all indices for which the expression makes sense. (The first component on the r.h.s. is absent if $k_1=p_1+2$, $k_2=p_2+2$,
while the second one is absent if $k_1+k_2\leq 1$.)
\end{proposition}
The first component of the bracket is proportional to the full symmetrization operation
$S^{k_1}(\wh\gc_{-1})\otimes S^{k_2}(\wh\gc_{-1})\to S^{k_1+k_2}(\wh\gc_{-1})$ but 
the coefficient of proportionality can sometimes vanish (cf. the coefficients ``$\tfrac{p}{2}$'' in equation (3.5) and ``$\alpha(p,i;q,j)$'' in equation 
(3.6) of \cite{San}).
On the other hand, the second component is always a {\it non-zero} multiple of the full symmetrization of the partial contraction
$S^{k_1}(\wh\gc_{-1})\otimes S^{k_2}(\wh\gc_{-1})\to S^{k_1-1}(\wh\gc_{-1})\otimes S^{k_2-1}(\wh\gc_{-1})\to S^{k_1+k_2-2}(\wh\gc_{-1})$ (cf. ``$1$'' in equation (3.5) and ``$\beta(i;j)$'' in equation 
(3.6) of \cite{San}). We refer the reader to the original source for more details and only remark here that $\gc=\gk\oplus\mathfrak{z}$ decomposes into the direct sum of two graded subalgebras $\gk=\bigoplus_{p\geq -2}\gk_{p}$ and $\mathfrak{z}=\bigoplus_{p\geq -1}\mathfrak{z}_{p}$, where $\mathfrak{z}_p\cong S^{p}(\gc_{-1})\oplus S^{p-2}(\gc_{-1})\oplus\cdots$ is the unique $\gk_0$-submodule that is complementary to $\gk_p$ inside $\gc_p$.

\begin{definition}\cite[\S 3]{San}
\label{def:13}
The {\it universal CR algebra} is the pair $(\gc,\gu)$, where 
$\mathfrak{u}=\bigoplus_{p\geq -1}\gu_{p}$ is the $\bZ$-graded subspace of $\wh\gc$ with graded component 
$$\gu_p=\gc_{(p,p+2)}^{p+2}\bigoplus_{\tiny\begin{gathered}|\ell|=p,p-2,p-4,\ldots\end{gathered}}\gc_{(p,\ell)}^{p+2}
\bigoplus_{\tiny\begin{gathered}k=p,p-2,\ldots\\ |\ell|=k,k-2,k-4,\ldots\end{gathered}}\gc_{(p,\ell)}^{k}$$
given by the direct sum of all $\ad(\gJ)$-eigenspaces in $\wh\gc_p$ except the eigenspace of minimal eigenvalue \!$-i(p+2)$, namely, $\gc_{(p,-p-2)}^{p+2}\cong S^{p+2}(\gc_{-1(01)})$. 
\end{definition}

The name ``universal CR algebra'' stems from the fact that $\gu$ is a complex subalgebra of $\wh\gc$. Moreover a pointwise invariant of $k$-nondegenerate CR manifolds of hypersurface type
that has been first introduced in \cite{San} under the name of ``core'' (a generalization of the usual CR-symbol of Levi-nondegenerate CR manifolds) always has a natural injection into $\overline\gu$ \cite[\S 3.1]{San}. In our locally homogeneous context, this reads as the injection $\imath:\overline\gq^p/\overline\gq^{p+1}\longrightarrow\gc_{(p,-p-2)}^{p+2}$ . However, as Example \ref{ex:4nondeg} demonstrates, the image of $\imath$ does not necessarily land in the graded Lie algebra $\gr(\gg)$ associated to the contact filtration, making it less useful for our purposes. A more transparent version of this result is then given in Proposition \ref{prop:crucial} later on.

We note that $\overline\gu_p$ is the sum of all $\ad(\gJ)$-eigenspaces in $\wh\gc_p$ except that $\gc_{(p,p+2)}^{p+2}\cong S^{p+2}(\gc_{-1(10)})$ of maximal eigenvalue, and that the intersection $\gu\cap\overline	\gu$ can be regarded as the formal analogue of the usual complexified stabilizer subalgebra of a finite-dimensional CR algebra.

\section{Infinitesimal CR automorphisms of Levi degenerate CR hypersurfaces}
\setcounter{section}{3}\setcounter{equation}{0}\label{sec:3}

\subsection{From filtrations to gradings}\label{sec:2.4}

Let $(\cM, \cD, \cJ)$ be a CR manifold of hypersurface type, or a germ of it at a fixed point $x\in \cM$, which is locally homogeneous and $k$-nondegenerate,  and let $\gg$ be the associated Lie algebra of infinitesimal CR automorphisms, endowed with the compatible filtration \eqref{eq:filtration}. We let
$$\gr(\gg)=\bigoplus_{p\geq -2}\gg_p\;,\qquad\gg_p=\gg^p/\gg^{p+1}\;,$$ be the associated graded Lie algebra. The analogous construction clearly holds for $\wh\gg=\gg\otimes\bC$.
By Lemma \ref{h.o.l.-with-symmetries} and $(iii)$ and $(v)$ of Proposition \ref{prop:basic-properties-filtration}, we have:
\begin{lemma}
\label{lem:14}
The negatively-graded part $\gg_-=\gg_{-2}\oplus\gg_{-1}\cong \big(T_x\cM/\cD|_x\big)\oplus \big(\cD|_x/\Re(\cF^0)|_x\big)$ of $\gr(\gg)$ is isomorphic to the Heisenberg algebra $\gc_-$. If $v\in\gg_p$, $p\geq 0$, satisfies $[v,\gg_{-1}]=0$ then $v=0$. 
\end{lemma}
We emphasize that the above result does in fact use both transitivity and $k$-nondegeneracy, and that this is the case for all the remaining results of \S\ref{sec:3}, unless otherwise explicitly stated. We also recall (see (iii) of Proposition \ref{prop:basic-properties-filtration} in \S\ref{subsec:2.3}) that $\gr(\gg)$ starts with $\gg_{-2}$, due to transitivity, $k$-nondegeneracy and the fact that $(M,\cD,\cJ)$ is of hypersurface type. Furthermore $\mathfrak{stab}=\gg^0$ in the Levi-nondegenerate case (for example, by Lemma \ref{lem:14} with $\cF^0=0$), while the stabilizer is only properly contained in $\gg^0$ in the general $k$-nondegenerate case with $k\geq 2$.

In fact,  the main motivation of \S\ref{sec:3} is two-fold. One objective is to investigate the interplay of the stabilizer with the non-negative part of $\gg$, or, more generally, of the Freeman filtration with the contact filtration, 
thus constraining the non-negative graded components of $\gr(\gg)$. The other is to understand under which conditions the complex structure, or equivalently $\gq$, is fully determined by $\gg$.
\vskip0.3cm

It follows that $\gr(\gg)$ is a graded subalgebra of the contact algebra $\gc$. Moreover $\cJ|_x$ induces a complex structure $\mathfrak{J}$ on $\gc_{-1}$ that satisfies
$[\fJ v,\fJ w]=[v,w]$ for all $v,w\in\gc_{-1}$; this complex structure can be used to construct the universal CR algebra $(\gc,\gu)$ as in \S \ref{sec:2.3}. 

The Lie algebra $\gr(\gg)$ has complexified graded components
$$
\wh\gg_{p}\subset\wh\gc_{p}=\gu_{p}\oplus \gc_{(p,-p-2)}^{p+2}\;,
$$
for all $p\geq -2$. Our aim is to constrain them.
\begin{proposition}
\label{prop:crucial}
 Let $\xi\in\wh\gg^{p}$ for some $p\geq -1$, with the equivalence class $v=\llbracket\xi\rrbracket\in\wh\gg_p$. 
Then $\xi\in \gq^p+\overline\gq^p$ and
\begin{itemize}
	\item[$(i)$] $\xi\in\gq^p+\overline\gq^{p+1}$ if and only if $v\in \gu_p$;
	\item[$(ii)$] $\xi\in\gq^{p+1}+\overline\gq^{p+1}$ if and only if $v\in \gu_p\cap\overline\gu_p$.
\end{itemize}

\end{proposition}
\begin{proof}
We recall that $\wh\gg^p\subset\gq^p+\overline\gq^p$ always due to \eqref{eq:eval-p}
and set to prove claim $(i)$ by induction.
\vskip0.1cm\par\noindent
\underline{\it Case $p=-1$}. 
First of all $\xi\in\wh\gg^{-1}= \gq+\overline\gq$, so we may write $\xi=\xi'+\xi''$ with 
$\xi'\in\gq$ and $\xi''\in\overline\gq$.
Let then $\eta\in \gq$ with equivalence class $w=\llbracket\eta\rrbracket\in\gu_{-1}=\gc_{(-1,1)}^{1}$   %\wh\gg_{-1}$
 and decompose $v=v'+v''$ according to $
\wh\gg_{-1}=\gu_{-1}\oplus\overline\gu_{-1}=\gc_{(-1,1)}^{1}\oplus \gc_{(-1,-1)}^{1}$.
By construction
$$[\xi'',\eta]\!\!\mod\wh\gg^{-1}=[\xi,\eta]\!\!\mod\wh\gg^{-1}=[v,w]=[v'',w]\;,$$
which vanishes for all $\eta$ if and only if $\xi''\in\overline\gq^0$ or, equivalently, 
$v''=0$. The latter condition just means that $v\in \gu_{-1}$.
\vskip0.1cm\par\noindent
\underline{\it Case $p\geq 0$}. Let $\eta$ and $w$ be as above and note that $\epsilon=[\xi,\eta]\in \wh\gg^{p-1}$. We decompose
$\epsilon=\epsilon'+\epsilon''$ according to $\wh\gg^{p-1}\subset\gq^{p-1}+\overline\gq^{p-1}$
and let $u=\llbracket\epsilon\rrbracket$ be the class of $\epsilon$ in $\wh\gg_{p-1}$.

If $v\in\gu_p$, then $u=[v,w]\in [\gu_{p},\gu_{-1}]\subset\gu_{p-1}$ trivially projects to $\gc_{(p-1,-p-1)}^{p+1}$ and 
$\epsilon\in\gq^{p-1}+\overline\gq^{p}$ by the induction hypothesis.
We write $\xi=\xi'+\xi''$ with 
$\xi'\in\gq^{p}$ and $\xi''\in\overline\gq^{p}$ by Proposition \ref{prop:basic-properties-filtration}, so that 
$[\xi'',\eta]\!\!\mod\gq+\overline\gq^{p}=\epsilon\!\!\mod\gq+\overline\gq^{p}=0$
and $\xi''\in\overline\gq^{p+1}$. This proves one direction.

Conversely, assume $\xi=\xi'+\xi''$ with $\xi'\in\gq^{p}$ and $\xi''\in\overline\gq^{p+1}$. If
$\eta_k\in\gq$ with equivalence class $w_k=\llbracket\eta_k\rrbracket\in\wh\gg_{-1}$, for $k=1,\ldots,p+2$, then
$$[\ldots[[\xi,\eta_1],\eta_2],\ldots,\eta_{p+2}]\in\gq+\overline\gq$$
so its equivalence class $[\ldots[[v,w_1],w_2],\ldots,w_{p+2}]$ in $\wh\gg_{-2}$ vanishes. Since all $w_k\in \gu_{-1}=\gc^{1}_{(-1,1)}$
and $v=v'+v''$ with $v'\in \gu_{p}$ and $v''\in\gc_{(p,-p-2)}^{p+2}$, we finally see that
$$
0=[\cdots[[v,w_1],w_2],\cdots,w_{p+2}]=[\cdots[[v'',w_1],w_2],\cdots,w_{p+2}]
$$
as element of $\wh\gg_{-2}$. By Proposition \ref{prop:quasi-grading}, each bracket  $[\gc^{k}_{(k-2,-k)},\gc^{1}_{(-1,1)}]\subset \gc^{k-1}_{(k-3,-k+1)}$ is a non-zero multiple of full symmetrization of contraction, showing that $v''=0$. 
\vskip0.1cm\par
Claim $(i)$ has been proved. Claim $(ii)$ follows then from $(i)$ and its conjugate statement.
\end{proof}
%By conjugation, one readily sees that the collection of elements of $\wh\gg^p$ with equivalence class 
%projecting trivially to $\gc_{(p,p+2)}^{p+2}\oplus \gc_{(p,-p-2)}^{p+2}$ is $\wh\gg^p\cap(\gq^{p+1}+\overline\gq^{p+1})$. 
\begin{corollary}
\label{cor:10}
Let $\xi\in\gq^q\setminus \gq^{q+1}$ for $q\geq -1$, then $\xi\notin\wh\gg^{q+1}$. If $p\leq q$ is the maximal integer such that $\xi\in\wh\gg^{p}$, then the equivalence class $v=\llbracket\xi\rrbracket\in\wh\gg_p$ is non-trivial and satisfies: 
$$
\begin{cases}
v\in\gu_p\setminus \big(\gu_p\cap\overline\gu_p\big)&\;\;\text{if}\;\;p=q\;,\\
v\in\gu_p\cap \overline\gu_p&\;\;\text{if}\;\;p<q\;.
\end{cases}
$$
\end{corollary}
\begin{proof}
If $\xi$ were an element of $\wh\gg^{q+1}$, then $\xi\in\gq^{q+1}+\overline\gq^{q+1}$ thanks to (v)-(vi) of Proposition \ref{prop:basic-properties-filtration}. Since $\xi\in\gq$, this implies $\xi\in\gq^{q+1}$, which is a contradiction. The first claim has been proved.

The second claim follows then immediately from Proposition \ref{prop:crucial}.
\end{proof}
These last two results don't make use of $k$-nondegeneracy. The next subsection deals with the graded case and the stronger results therein rely on $k$-nondegeneracy.

\subsection{Trivial filtered deformations}\label{sec:2.5}

In this section, we assume that the filtered deformation $\gg$ of 
$\gr(\gg)$ is trivial, namely, 
%there exists a choice of complementary subspaces in the chain of filtrands of $\gg$ so that 
that $\gg\cong \gr(\gg)$ as filtered Lie algebras with filtrands $\gg^p=\bigoplus_{j\geq p}\gg_{j}$. 
(We recalled these notions and notations just before Remark \ref{rem:11}; see also \cite{CK} for more details.) 
We emphasize that while the embedding of $\gq$ in $\wh\gg$ and the contact filtration on $\gg$ are canonical, the identification $\gg\cong \gr(\gg)$ is not. We fix one identification once and for all and decompose any element $\xi\in\gg$ as $\xi=\sum_{j\geq -2}\xi_{j}$, with graded components
$\xi_j\in\gg_j$.
\begin{remark}
\label{rem:17}
 The following notions and results also involve the complex structure, i.e., they depend on the whole CR algebra $(\gg,\gq)$. They are motivated by the fact that, while $\gg\cong\gr(\gg)$ is a trivial filtered deformation, thus a subalgebra of the contact algebra $\gc$, the condition $\gq\subset \gu$ is not automatic.
(In the terminology of \cite[\S 4]{MeNa3}, the inclusion $\gg\cong\gr(\gg)\subset\gc$ may not give rise to a morphism of CR algebras.) We note that the condition $\gq\subset \gu$ is equivalent to $\gq\subset \widehat\gg\cap\gu$. In particular, it is not guaranteed that
$(\gg,\gq)$  is a CR subalgebra of the universal CR algebra, as this requires the stronger condition $\gq=\widehat\gg\cap\gu$, see \cite[\S 1]{MeNa3}.
\end{remark}
%\begin{definition}
%An element $\eta\in\gq$ is called {\it nice} if it does not belong to $\gq^0$.
%\end{definition}
We first recall that the component $\eta_{-1}$ of any $\eta=\sum_{j\geq -1}\eta_{j}\in\gq\setminus\gq^0$
is a non-zero element of $\gu_{-1}$, thanks to Corollary \ref{cor:10} and the fact that $\gu_{-1}\cap\overline\gu_{-1}=0$. In general, one cannot expect a similar statement for the higher graded components of $\eta$, as explained in Remark \ref{rem:17} and Example \ref{ex:19} below.

\begin{definition}
\label{def:aligned}
A trivial filtered deformation $\gg$ is called:
\vskip0.1cm\par\noindent
\begin{enumerate}
	\item {\it Semi-aligned} (w.r.t. $\gq$) if there exist elements $\eta_{(i)}\in\mathfrak q\cap\mathfrak u$, $1\leq i\leq d=\dim(\gu_{-1})$, decomposed as
\begin{equation}
\label{eq:components-nicebasis}
\eta_{(i)}=\sum_{j\geq -1}\eta_{(i)j}\;,\;\;\eta_{(i)j}\in\gu_j\;,
\end{equation}
such that the set $\{\eta_{(i),-1}:1\leq i\leq d\}$ is a basis of $\gu_{-1}$;
\vskip0.2cm\par\noindent
%\begin{itemize}
%	\item[$(i)$] the set $\{\eta_{(i)-1}:1\leq i\leq D\}$ is a basis of $\gu_{-1}$,
%	\item[$(ii)$] all components $\eta_{(i)j}\in\gu_j$, in other words, $\eta_{(i)}\in\gu$ for all $1\leq i\leq D$;
%\end{itemize}
%\vskip0.3cm\par\noindent
\item {\it Aligned} (w.r.t. $\gq$) if $\gu_{-1}\subset\gq$.
\end{enumerate}
\end{definition}

We note that an aligned trivial filtered deformation is always semi-aligned, by simply taking a basis of $\gu_{-1}$.

\begin{example}
\label{ex:19}
The graded Lie algebras $\gg$ considered in \cite[Examples 4.4, 4.5]{San} are aligned. An example of a trivial filtered deformation which admits a graded Lie algebra structure that
is not semi-aligned is given by Example \ref{ex:4nondeg}.
In this case, $\gg=\gg_{-2}\oplus\cdots\oplus\gg_{+2}$ as a graded Lie algebra with components
 \begin{align*}
\gg_{-2}&=\langle u_1^4\rangle\;,\quad
\gg_{-1}=\Bigl\langle\begin{pmatrix} 0& 1 \\ 0 & 0\end{pmatrix},u_2u_1^3\Bigr\rangle\;,\\
\gg_{0}&=\Bigl\langle\begin{pmatrix} 1& 0 \\ 0 & 1\end{pmatrix}, \begin{pmatrix} 1& 0 \\ 0 & -1\end{pmatrix}, u_2^2u_1^2\Bigl\rangle\;,\\
\gg_{1}&=\Bigl\langle\begin{pmatrix} 0& 0 \\ 1 & 0\end{pmatrix}, u_2^3u_1\Bigl\rangle\;,\quad
\gg_{2}=\langle u_2^4\rangle\;,
 \end{align*}
and $\gu_{-1}$ is $1$-dimensional and generated by $\begin{pmatrix} 0& 1 \\ 0 & 0\end{pmatrix}-2iu_2u_1^3$. Now, the element
$$B=\begin{pmatrix} 0& 1 \\ 0 & 0\end{pmatrix}-2iu_2u_1^3-3iu_2^2u_1^2
-4iu_2^3u_1\in\widehat\gg_{-1}\oplus\widehat\gg_0\oplus\widehat\gg_1$$ 
is in $\gq\setminus\gq^0$, and its degree $-1$ component is a basis of $\gu_{-1}$. However, it is not difficult to check that its degree $0$ component $-3iu_2^2u_1^2\notin \gu_0$, so $B$ is {\it not} an element of $\gu$ and, ultimately, it cannot be used to satisfy the semi-aligned condition in Definition \ref{def:aligned}. Modifying $B$ with the addition of elements of $\gq^0$ does not resolve the issue, as the leading part of any element of $\gq^0$ is in $\gu_{\geq 0}=\gu_{0}\oplus\gu_1\oplus\gu_2\oplus\cdots$ due to Corollary \ref{cor:10}, and hence the modified degree $0$ component is not in $\gu_0$ in any case. This shows that the graded Lie algebra structure is not semi-aligned.
\end{example}

For (semi-)aligned trivial filtered deformations, Corollary \ref{cor:10} can be extended to the next graded components.

 \begin{lemma}
\label{lemma:semi-aligned}
Assume $\gg$ is a semi-aligned trivial filtered deformation and let $\xi=\sum_{j\geq -1}\xi_{j}\in\wh\gg$. If $\xi\in\gq$ then $\xi_j\in\gu_j$ for all $j\geq -1$. In other words, the inclusion $\gq\subset\gu$ holds.
 \end{lemma}

 \begin{proof}
By $k$-nondegeneracy, the dimension of $\gg$ is finite, hence there exists $\hslash\geq -1$ such that $\wh\gg_p=0$ for all $p\geq \hslash+1$. We write 
$\xi=\sum_{\mu\leq j\leq \nu}\xi_{j}\in\gq$ with $-1\leq \mu\leq \nu\leq\hslash$, $\xi_\mu\neq 0$, $\xi_\nu\neq 0$, and assume, by contradiction, that $\xi_q\notin\gu_q$ for some minimal integer $\mu\leq q\leq\nu$. 

We decompose $\xi=\xi_{< q}+\xi_q+\xi_{>q}$, where $\xi_{<q}\in\gu$, and consider the iterated $(q+2)$-brackets 
$
\epsilon=[\ldots[[\xi,\eta],\eta],\ldots,\eta]
$,
where $\eta$ denotes, at each step, any of the elements $\eta_{(1)},\ldots,\eta_{(d)}$ as in part (1) of Definition \ref{def:aligned} (repetitions are allowed). First of all $\epsilon\in\gq\subset\wh\gg^{-1}$, as all $\eta$'s are in $\gq$. We then compute
\begin{align*}
\epsilon&=[\ldots[[\xi_{<q},\eta],\eta],\ldots,\eta]+[\ldots[[\xi_q,\eta],\eta],\ldots,\eta]+[\ldots[[\xi_{>q},\eta],\eta],\ldots,\eta]\\
&\equiv[\ldots[[\xi_q,\eta],\eta],\ldots,\eta]+[\ldots[[\xi_{>q},\eta],\eta],\ldots,\eta]\!\!\mod\gu\\
&\equiv[\ldots[[\xi_q,\eta_{-1}],\eta_{-1}],\ldots,\eta_{-1}]\!\!\mod\wh\gg^{-1}\;,
\end{align*}
which is an element of $\mathbb Z$-degree $-2$ and therefore has to vanish. Using Proposition \ref{prop:quasi-grading} as at the end of the proof of Proposition \ref{prop:crucial}, we infer $\xi_q\in\gu_q$, a contradiction.
 \end{proof}

 \begin{proposition}
\label{lem:trivialfil1}
Assume $\gg$ is an aligned trivial filtered deformation and let $\xi=\sum_{j\geq -1}\xi_{j}\in\wh\gg$. Then $\xi\in\gq$ if and only if $\xi_j\in\gu_j$ for all $j\geq -1$. 
%has a trivial projection to $\gc_{(p,-p-2)}^{p+2}$.
 \end{proposition}

 \begin{proof}
One direction follows from Lemma 
\ref{lemma:semi-aligned}, we now set to prove the converse direction: we have that $\xi\in\gq$, provided all $\xi_j\in\gu_j$.
We write $\xi=\sum_{\mu\leq j\leq \nu}\xi_{j}$ as in the proof of Lemma 
\ref{lemma:semi-aligned}
and note that $\xi\in\wh\gg^{\mu}\subset\gq^{\mu}+\overline\gq^{\mu}$ due to Proposition \ref{prop:basic-properties-filtration}. 
We work by induction on $\nu\geq -1$ and we will use that $\gu_{-1}=\gq\cap\wh\gg_{-1}$, since $\gg$ is aligned.
%If $\xi\in\gq$, then $\xi\in\wh\gg^\mu\cap\gq=\gq^{\mu}$ and $\xi_\mu\in\gu_\mu$ 
%by (i) of Proposition \ref{prop:crucial}.
%If $\mu=\nu$, we are done, otherwise $\xi_{\mu}\in\wh\gg_{\mu}\subset\wh\gg$ and therefore $\xi_{\mu}\in\gq$ by the induction hypothesis. Since $\xi-\xi_{\mu}\in\wh\gg^{\mu+1}\cap\gq$, we may iterate the argument and each $\xi_p\in\gu_p$, proving
%one implication. 

If $\nu=-1$, then $\xi=\xi_{-1}\in\gu_{-1}\subset \gq$; this is the base of our induction. We assume the claim holds for all $\nu\leq N$ for some given $N\geq -1$ and set to prove the claim for $\nu=N+1$.
Let $\xi=\sum_{\mu\leq j\leq N+1}\xi_{j}$ with all $\xi_j\in\gu_{j}$
and note that the difference $\xi-\xi_{N+1}$ is in $\gq$ by the induction hypothesis. Moreover $\xi_{N+1}\in\gq^{N+1}+\overline\gq^{N+2}$ by (i) of Proposition \ref{prop:crucial}. Take an arbitrary number of elements $\eta_{(1)},\ldots,\eta_{(n)}$ in $\gq$, each of which decomposes as
$$
\eta_{(i)}=\sum_{-1\leq j\leq\hslash}\eta_{(i)j}\;,
$$
with $\eta_{(i)j}\in\gu_j$ by Lemma \ref{lemma:semi-aligned}. We consider the iterated bracket 
$[\ldots[[\xi_{N+1},\eta_{(1)}],\eta_{(2)}],\ldots,\eta_{(n)}]$, which is an element of $\wh\gg\cap\gu$. Its graded components of degree $\leq N$ are in $\gq$ by the induction hypothesis, while those of degree $\geq N+1$ are in $\gq^{N+1}+\overline\gq^{N+2}$
by Proposition \ref{prop:crucial}. In summary 
$$[\cdots[[\xi_{N+1},\eta_{(1)}],\eta_{(2)}],\cdots,\eta_{(n)}]\in\gq+\overline\gq^{N+2}$$ 
and it is straightforward to check that this is equivalent to say that $\xi_{N+1}\in\gq^{N+1}+\overline\gq^{N+2+n}$. By $k$-nondegeneracy $\overline\gq^{N+2+n}=\gq\cap\overline\gq$ for $n$ sufficiently large, so $\xi_{N+1}\in\gq$ and $\xi\in\gq$ too.
\end{proof}
This result can be reformulated in a more suggestive way: if the filtered algebra structure is aligned, then the complex structure is completely determined:
\begin{corollary}
\label{cor:stab!}
If $\gg$ is an aligned trivial filtered deformation, then $\gq=\wh\gg\cap\gu$ and
$$\wh\gg_p\cap(\gq^{p+1}+\overline\gq^{p+1})=\wh\gg_p\cap\wh\stab$$ for all $p\geq -1$.
\end{corollary}
\begin{proof}
The first claim is just a reformulation of Proposition \ref{lem:trivialfil1}.
For the second claim, we start with the obvious inclusion
$\wh\gg_p\cap\wh\stab\subset \wh\gg_p\cap(\gq^{p+1}+\overline\gq^{p+1})$.
On the other hand, if $\xi\in\wh\gg_p\cap(\gq^{p+1}+\overline\gq^{p+1})$, then $\xi\in\gu_p\cap\overline\gu_p$ 
%has a trivial projection to
%$\gc_{(p,p+2)}^{p+2}\oplus\gc_{(p,-p-2)}^{p+2}$ by Proposition \ref{prop:crucial}
by $(ii)$ of Proposition \ref{prop:crucial}
and therefore $\xi\in\wh\stab=\gq\cap\overline\gq$ by Proposition \ref{lem:trivialfil1}. This proves the opposite inclusion.
\end{proof}
Summing up, we are now ready to prove the following. (See also Remark \ref{remark:later-on}.)
\begin{proposition}
\label{prop:stab-and-freeman}
If $\gg$ is an aligned trivial filtered deformation, then the complexified stabilizer subalgebra $\wh\stab=\wh\gg\cap\gu\cap\overline\gu$ is $\mathbb Z$-graded and 
$$\gq^p+\overline\gq^p=\wh\gg^p+\wh\stab$$ for all $p\geq -1$.
\end{proposition}
\begin{proof}
The first claim follows readily by $\gq=\wh\gg\cap\gu$ of Corollary \ref{cor:stab!} and the fact that $\wh\stab=\gq\cap\overline\gq$.

We already saw that $\wh\gg^p+\wh\stab\subset\gq^p+\overline\gq^p$ and we now establish the opposite inclusion. Let $\xi\in\gq^p$, which we decompose into $\xi=\sum_{\mu\leq j\leq \nu}\xi_{j}$ with $-1\leq \mu\leq \nu\leq\hslash$, $\xi_\mu\neq 0$, and $\xi_\nu\neq 0$. By Proposition \ref{lem:trivialfil1} and Proposition \ref{prop:crucial},
%\ref{prop:basic-properties-filtration}, 
each component $\xi_j\in\gq^j$.

If $\mu\geq p$, there is nothing to prove, since $\xi\in\wh\gg^p$ automatically.
If instead $\mu<p$, then $\xi_\mu$ has to be in $\gq^{\mu+1}$ since $\xi\in\gq^p$, hence $\xi_\mu\in\wh\stab$ by Corollary \ref{cor:stab!}. Iterating the argument we see that $\xi_i\in\wh\stab$ for all $\mu\leq j<p$, whence
$\xi\in\wh\stab+\wh\gg^p$. We have shown that 
$\gq^p\subset\wh\gg^p+\wh\stab$ and the
desired inclusion follows from conjugation.
\end{proof}
Taking the intersection with $\gq=\wh\gg\cap\gu$ immediately yields the following nice characterization of the $q$-th term of the Freeman sequence of $(\gg,\gq)$, which is in agreement with \cite[eq. (4.1) in the proof of Thm. 4.2]{San}:
\begin{corollary}
\label{cor:charact}
If $\gg$ is an aligned trivial filtered deformation, then
$$\gq^q=\bigoplus_{0\leq p\leq q-1}\wh\stab_{p}\oplus\bigoplus_{p\geq q} \wh\gg_p\cap\gu_p\;,$$ 
for all $q\geq 0$.
\end{corollary}

\subsection{$k$-nondegenerate homogeneous models}

We conclude this section with the first main result of this paper. To do so, we first recall the definition of homogeneous model in the sense of \cite[Def. 4.1]{San}. (The definition we give below is slightly reformulated so to make no reference to the concept of core, a generalization for $k$-nondegenerate CR manifolds of the usual notion of CR-symbol for Levi-nondegenerate CR manifolds.  The concept also appeared  for the case $k=2$  and with a different name in \cite{PZ}. Since there exists only one core for $3$-nondegenerate $7$-dimensional
CR manifolds, see \cite[\S 6]{San}, this will be enough for our purposes.)
\begin{definition}
\label{ukip}
A {\it $k$-nondegenerate model}
is the datum of a $\bZ$-graded Lie subalgebra 
$$
%\beq
%\label{hommod}
\gg=\bigoplus_{p\in\bZ}\gg_{p}
$$ 
of the infinite-dimensional contact algebra $\gc$ satisfying the following properties:
\begin{itemize}
\item[(i)] $\gg_-=\gc_-$;
\item[(ii)] the grading element $E$ is in $\gg_{0}$;
\item[(iii)] $\wh\gg_p=(\wh\gg_p\cap\gu_p)+(\wh\gg_p\cap\bar\gu_p)$ for all $p\geq 0$;
\item[(iv)] $\wh\gg_p$ projects to the $\ad(\mathfrak{J})$-eigenspace
 %$\mathfrak{M}^{p(10)}$
of maximum eigenvalue $\gc_{(p,p+2)}^{p+2}$ non-trivially for all $p\leq k-2$ and trivially for all $p\geq k-1$.
\end {itemize}
\end{definition}
We refer the interested reader to the original source \cite[\S 4]{San} for more details, in particular to Theorem 4.2 and Examples 4.4 and 4.5 therein. We now show that the graded Lie algebras $\gg$ considered in \S\ref{sec:2.5} (possibly supplemented by the grading element if this is originally absent) are in fact models in the sense of Definition \ref{ukip}.

\begin{theorem}
\label{thm:25}
Let $(\cM, \cD, \cJ)$ be a locally homogeneous and $k$-nondegenerate CR manifold of hypersurface type and assume that the Lie algebra $\gg=\inf(\cM,\cD,\cJ)$ of its infinitesimal CR automorphisms  is a trivial filtered deformation w.r.t. the contact filtration \eqref{eq:filtration} and that the identification $\gg\cong \gr(\gg)$ can be chosen aligned w.r.t. $\gq$ in the sense of Definition \ref{def:aligned}.
Then: 
%\begin{equation}
%\label{eq:to-recall}
 \begin{enumerate}
	\item The corresponding CR algebra is given by $(\gg,\gq)=(\gg,\wh\gg\cap\gu)$
%\end{equation} 
and it is a $\mathbb Z$-graded CR subalgebra of the universal CR algebra $(\gc,\gu)$, 
	\item $\gg+\mathbb R E=\bigoplus_{p\geq -2}\gg_p+\mathbb R E$ is a $k$-nondegenerate model. 
 \end{enumerate}
\end{theorem}

\begin{proof}
The first claim has been already proved in Corollary \ref{cor:stab!}, so
it only remains to establish properties $(iii)$ and $(iv)$ of Definition \ref{ukip} for the Lie algebra $\gg+\mathbb R E$. 

Now $\wh\gg_p\subset \gq^p+\overline\gq^p$ by Proposition \ref{prop:basic-properties-filtration}, so any $\xi\in\wh\gg_{p}$ decomposes into $\xi=\xi'+\xi''$, with $\xi'\in\gq^p$, $\xi''\in\overline\gq^p$. Writing
$
\xi'=\sum_{j\geq -1}\xi'_{j}$, $\xi''=\sum_{j\geq -1}\xi''_{j}$, we
 note that $\xi'_p\in\wh\gg_p\cap\gu_p$ and $\xi''_p\in\wh\gg_p\cap\overline\gu_p$ by Proposition \ref{lem:trivialfil1}. This shows $\wh\gg_p\subset(\wh\gg_p\cap\gu_p)+(\wh\gg_p\cap\overline\gu_p)$, and the converse inclusion is obvious.

If $q\leq k-2$, then there exists some $\xi\in\gq^{q}\setminus\gq^{q+1}$. Writing $\xi=\sum_{j\geq -1}\xi_{j}$, then $\xi_j\in\wh\stab$ for all $j\leq q-1$ and  $\xi_q\in\wh\gg_q\cap\gu_q\setminus\wh\stab_q$
by Corollary \ref{cor:charact}.
%so we may assume w.l.og. that $\xi=\sum_{i\geq p}\xi_{i}$. 
Now $\xi_q\in\gq^q$ and if it were to project trivially to $\gc_{(q,q+2)}^{q+2}$ then it would be in $\overline\gq^q$, hence in $\wh\stab$, 
%and 
%$$\xi=\xi_p+\sum_{i\geq p+1}\xi_{i}$$ would be in $\wh\stab+\wh\gg^{p+1}\subset\gq^{p+1}+\overline\gq^{p+1}$ and therefore in $\overline\gq^{p+1}$, 
which is a contradiction. Hence $\xi_q$ is the desired element projecting non-trivially to $\gc_{(q,q+2)}^{q+2}$. 

Let $q\geq k-1$ and assume that there exists $\xi\in\wh\gg_q$ with a non-trivial projection to $\gc_{(q,q+2)}^{q+2}$. By property $(iii)$, we may decompose $\xi=\xi'+\xi''$, where $\xi'\in\wh\gg_q\cap\gu_q$ and 
$\xi''\in\wh\gg_q\cap\overline\gu_q$. Hence $\xi'\in\gq^{q}$ with a non-trivial projection to $\gc_{(q,q+2)}^{q+2}$, so $\xi'\notin\wh\stab$ and $\gq^{q}\neq \gq\cap\overline\gq$, a contradiction.
\end{proof}

\vskip0.1cm\par\noindent

\section{Homogeneous $7$-dimensional $3$-nondegenerate CR manifolds}\hfill\par
\label{sec:4}\setcounter{section}{4}\setcounter{equation}{0}

Here we classify the locally homogeneous $7$-dimensional $3$-nondegenerate CR manifolds of hypersurface type up to local CR equivalence. We will consider global CR equivalence in \S\ref{sec:3.3}.

\subsection{The maximal symmetric homogeneous space: abstract model}\label{sec:3.1}

The model we now describe is realized as a real Lie subalgebra of the
complex contact algebra $\wh\gc=\bigoplus\wh\gc_p$ with negatively-graded part 
$\wh\gc_{-2}=\bC e_{-2}$, $\wh\gc_{-1}=\left\langle z,\overline z\right\rangle$, where 
$e_{-2}$ is a real generator of $\gc_{-2}$ and $z$ a basis of $\gc_{(-1,1)}^1$. We note that
each term of the universal CR algebra
$$\gc_{(p,\ell)}^{k}\cong\bC z^{\tfrac{k+\ell}{2}}\overline z^{\tfrac{k-\ell}{2}}$$ 
is one-dimensional, where we dropped the symbol $\odot$ in the expression of symmetric products. 
 In our conventions the grading element $E= -2\in\gc_{(0,0)}^{0}$ and 
$\mathfrak{J}=
2z\bar z\in\gc_{(0,0)}^{2}$. Furthermore $\gc_{(1,1)}^{1}\cong \mathbb C z$ and $\gc_{(1,-1)}^{1}\cong\mathbb C \overline z$, and we will always make clear from the context if $z,\overline z$ have to be regarded as elements in degree $+1$ instead of $-1$.

Theorem 6.1 of \cite{San} states 
that there exists a $7$-dimensional $3$-nondegenerate homogeneous model $\gg$, unique up to isomorphism:
it is the $8$-dimensional $\bZ$-graded Lie subalgebra
\begin{equation*}
%\label{modellino}
\gg=\bigoplus_{p\in\bZ}\gg_{p}
\end{equation*} 
of the real contact algebra $\gc$ with components
$$
\gg_p=\begin{cases}
0\;\;\;\;\;\;\;\;\;\;\;\;\;\;\;\;\;\;\;\;\;\;\;\;\;\;\;\;\;\;\;\;\;\;\;\;\;\;\;\;\;\;\;\;\;\;\;\;\;\;\;&\text{for all}\;\;p>1\;,\\
\Re\left\langle N, \overline N\right\rangle\,\,\,\;\;\;\;\;\;&\text{for}\;\;p=1\;,\\
\Re\left\langle E, M, \overline M\right\rangle\,\;\;\;\;\;&\text{for}\;\;p=0\;,\\
\gc_{p}\,\;\;\;\;\;\;\;\;\;&\text{for}\;\;p=-2, -1\;,\\
%0\;\;\,\;\;\;\;\;\;\;\;&\text{for all}\;\;p>1\;,\\
\end{cases}
$$
where $M=z^2+z\overline z$ and  $N=z^3+2z^2\bar z+z\bar z^2-3iz-3i\bar z$. The non-trivial Lie brackets are given by the obvious action of the grading element and the following relations (together with their conjugates):
\begin{equation}
\label{eq:brackets-model}
\begin{aligned}
[z,\overline z]&=-\frac{i}{2}e_{-2}\;,\quad [M,z]=\frac{i}{2}z\;,\quad [M,\overline z]=-iz-\frac{i}{2}\overline z\;,\\
[M,\overline M]&=-i(M+\overline M)\;,\quad
[N,e_{-2}]=-3i(z+\overline z)\;,\quad[N,z]=-\frac{i}{2}M-\frac{3}{4}E\;,\\
[N,\overline z]&=-\frac{3}{2}iM-2i\overline M+\frac{3}{4}E\;,\quad
[M,N]=-\frac{i}{2}N\;,\quad [\overline M,N]=\frac{3}{2}iN+i\overline N\;.
\end{aligned}
\end{equation}

The associated terms of the Freeman sequence are
\begin{align*}
\gq^{-1}&=\gq=\left\langle z, E, M, N  \right\rangle\,,\quad
\gq^{0}=\left\langle E, M, N \right\rangle\,,\\
\gq^{1}&=\left\langle E, N  \right\rangle\,,\qquad
\gq^2=\gq\cap\overline\gq=\left\langle E\right\rangle\;,
\end{align*}
and the contact filtration \eqref{eq:filtration} coincides with the natural filtration associated to the grading.

The Lie algebra $\gg_0$ is the Borel subalgebra of $\ggl(\gg_{-1})$ stabilizing the line $z+\overline z$ in $\gg_{-1}$. 
As abstract Lie algebra $\gg\cong\ggl_2(\bR)\ltimes S^3\bR^2$, with $5$-dimensional radical $\mathfrak{rad}(\gg)=\mathfrak{z}(\ggl_2(\mathbb R))\ltimes S^3\mathbb R^2$ and Levi factor $\gsl_2(\mathbb R)$ as follows:
\begin{align*}
\mathfrak{rad}(\gg)&=\langle e_{-2}, z+\overline z, M+\overline M, E-i(M-\overline M), N+\overline N\rangle\;,\\
\gsl_2(\mathbb R)&=\langle \widetilde E=-\tfrac i2(M-\overline M)-\tfrac32 E, X=-\tfrac{i}{\sqrt 2}(z-\overline z), Y=-\tfrac{i}{\sqrt 2}(N-\overline N)\rangle\;.
\end{align*}
Here the element $E-i(M-\overline M)\cong\operatorname{diag}(-2/3,-2/3)$ generates the center $\mathfrak{z}(\ggl_2(\mathbb R))$ of $\ggl_2(\mathbb R)$ 
%and it acts with eigenvalues $-2$ on $S^3\mathbb R^2$.
whereas $E\cong\operatorname{diag}(-2/3,1/3)$ the stabilizer subalgebra.

We summarize the Lie algebra structure and grading corresponding to the contact filtration 
in the following root diagram of $\gg$ (as usual, nontrivial brackets of root vectors correspond to
nontrivial sums of roots), where we also indicated generators of graded components and circled the 
stabilizer subalgebra:

 
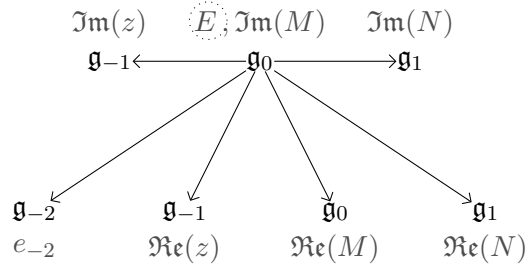
\begin{center}
\begin{tikzpicture}
\node (A) at (-2,2) {$\gg_{-1}$};
\node (B) at (0,2) {$\gg_0$};
\node (C) at (2,2) {$\gg_1$}; 
\node (a) at (-3,0) {$\gg_{-2}$};
\node (b) at (-1,0) {$\gg_{-1}$};
\node (c) at (1,0) {$\gg_0$};
\node (d) at (3,0) {$\gg_1$};
\draw[dotted] (-0.69,2.54) circle (0.25);
\path[->,font=\scriptsize,>=angle 90]
(B) edge (A)
(B) edge (C)
(B) edge (a)
(B) edge (b)
(B) edge (c)
(B) edge (d);
\node[black!70] (A1) at (-2,2.5) {$\mathfrak{Im}(z)$};
\node[black!70] (B1) at (0,2.5) {$E$\,,\,$\mathfrak{Im}(M)$};
\node[black!70] (C1) at (2,2.5) {$\mathfrak{Im}(N)$};
\node[black!70] (a1) at (-3,-0.5) {$e_{-2}$};
\node[black!70] (b1) at (-1,-0.5) {$\Re(z)$};
\node[black!70] (c1) at (1,-0.5) {$\Re(M)$};
\node[black!70] (d1) at (3,-0.5) {$\Re(N)$};
\end{tikzpicture}
 \end{center}

For an explicit coordinate embedding in $\bC^4$ of this model, see \S \ref{sec:5}. It is a refinement of that found at the end of \cite[\S 5.1]{FK2} -- our proposed geometric interpretation in terms of the rational normal curve of degree $3$ is also amenable to generalizations, see again \S \ref{sec:5}.

\subsection{The maximal symmetric homogeneous space in disguise}
\label{sec:section-added}

In this auxiliary section, we present some locally homogeneous $7$-dimensional $3$-nondegenerate CR manifolds relevant for the proof of our local classification result in \S\ref{sec:3.2}. They are given in terms of their CR algebras $(\gs,\gp)$, where, in all cases, $\gs$ is a subalgebra of the $8$-dimensional Lie algebra $\gg$ of \S\ref{sec:3.1} but $\gp$ is not readily related to $\gq$. A posteriori, it turns out that all these examples are (locally) geometrically equivalent to the maximally symmetric model. 
\begin{example}
\label{ex:1-parameter}
The locally homogeneous model $(\gg,\gq)$ of \S\ref{sec:3.1} admits a $1$-parameter family of deformations $(\gs_t,\gp_t)$, $t\in\mathbb C$, defined as follows: $\gs_t=\gg$ for all $t\in\mathbb C$, whereas $\gp_t$ is generated by
\begin{align*}
\eta&=z+t \overline M-t^2\overline N\;,\\
\epsilon&=M+t \overline N\;,\\
\xi&=N\;,\\
E_o&=E-\tfrac 23 i\overline t N+\tfrac 23 i t\overline N\;.
\end{align*}
It is straightforward to see that $\gp_t$ is a complex subalgebra of $\wh\gs_t=\wh\gg$ and that the associated terms of the Freeman sequence are
\begin{align*}
\gp_t^{-1}&=\gp_t=\left\langle \eta, \epsilon, \xi, E_o  \right\rangle\,,\quad
\gp_t^{0}=\left\langle \epsilon, \xi, E_o \right\rangle\,,\\
\gp_t^{1}&=\left\langle \xi, E_o  \right\rangle\,,\qquad
\gp_t^2=\gp_t\cap\overline\gp_t=\left\langle E_o\right\rangle\;.
\end{align*}
However, this is just the maximally symmetric homogeneous space $(\gg,\gq)$ in disguise. In fact, the CR algebra 
$(\gs_t,\gp_t)=e^{\ad_X}\cdot\,(\gg,\gq)$ for $X=\tfrac 23 i\big(\overline t N-t\overline N\big)\in\gg_1$ and the $1$-parameter family of deformations consists of CR algebras that are all isomorphic.
\end{example}
\begin{example}
\label{ex:1-parameter-bis}
Consider the $7$-dimensional graded subalgebra $\gs$ of $\gg$ with components
$$
\gs_p=\begin{cases}
0\;\;\;\;\;\;\;\;\;\;\;\;\;\;\;\;\;\;\;\;\;\;\;\;\;\;\;\;\;\;\;\;\;\;\;\;\;\;\;\;\;\;\;\;\;\;\;\;\;\;\;&\text{for all}\;\;p>1\;,\\
\Re\left\langle N, \overline N\right\rangle\,\,\,\;\;\;\;\;\;&\text{for}\;\;p=1\;,\\
\Re\left\langle L,\overline L\right\rangle\,\;\;\;\;\;&\text{for}\;\;p=0\;,\\
\gc_{p}\,\;\;\;\;\;\;\;\;\;&\text{for}\;\;p=-2, -1\;,\\
%0\;\;\,\;\;\;\;\;\;\;\;&\text{for all}\;\;p>1\;,\\
\end{cases}
$$
where $L=2i M+3E$. It is isomorphic to $\gsl_2(\bR)\ltimes S^3\bR^2$ and its brackets are given in \eqref{eq:brackets-smaller} later on. If we endow it with the complex subalgebra $\gp=\gq\cap\wh\gs=\left\langle z,L,N \right\rangle$ of $\wh\gs$, we just get a simply transitive CR subalgebra $(\gs,\gp)$ of the maximally symmetric homogeneous space $(\gg,\gq)$.

We define the $1$-parameter family of deformations $(\gs_t,\gp_t)$, $t\in\mathbb C$, as follows: $\gs_t=\gs$ for all parameters, and $\gp_t$ is generated by
\begin{align*}
\eta&=z+t \overline L+8t^2\overline N\;,\\
\epsilon&=L+8t \overline N\;,\\
\xi&=N\;.
\end{align*}
The latter is a complex subalgebra of $\wh\gs_t=\wh\gs$ and the associated terms of the Freeman sequence are
\begin{align*}
\gp_t^{-1}&=\gp_t=\left\langle \eta, \epsilon, \xi  \right\rangle\,,\quad
\gp_t^{0}=\left\langle \epsilon, \xi \right\rangle\,,\\
\gp_t^{1}&=\left\langle \xi  \right\rangle\,,\qquad
\gp_t^2=\gp_t\cap\overline\gp_t=0\;.
\end{align*}
In this case
$(\gs_t,\gp_t)=e^{\ad_X}\cdot\,(\gs,\gp)$ for $X=-\tfrac 43\big(\overline t N+t\overline N\big)\in\gs_1$ and the $1$-parameter family of deformations consists of isomorphic CR algebras.
Again, this is just the maximally symmetric homogeneous model in disguise.
\end{example}
\begin{example}
\label{ex:1-parameter-tris}
Finally we consider the $7$-dimensional graded subalgebra $\mathfrak{s}$ of $\gg$ with components
$$
\gs_p=\begin{cases}
0\;\;\;\;\;\;\;\;\;\;\;\;\;\;\;\;\;\;\;\;\;\;\;\;\;\;\;\;\;\;\;\;\;\;\;\;\;\;\;\;\;\;\;\;\;\;\;\;\;\;\;&\text{for all}\;\;p>1\;,\\
\Re\left\langle \Xi:=N+\overline N\right\rangle\,\,\,\;\;\;\;\;\;&\text{for}\;\;p=1\;,\\
\Re\left\langle E, M,\overline M\right\rangle\,\;\;\;\;\;&\text{for}\;\;p=0\;,\\
\gc_{p}\,\;\;\;\;\;\;\;\;\;&\text{for}\;\;p=-2, -1\;,\\
%0\;\;\,\;\;\;\;\;\;\;\;&\text{for all}\;\;p>1\;,\\
\end{cases}
$$
together with the complex subalgebra $\gp_t$ of $\wh\gs$ generated by
\begin{align*}
\eta&=z+t\overline M-t^2\Xi\;,\\
%z+\big(\overline z^2+z\overline z\big)-\Xi\;,\\
\epsilon&=M+t\Xi\;,\\
%\big(z^2+z\overline z\big)+\Xi\;,\\
\xi&=E+\tfrac23 it\Xi\;.
\end{align*}
Again we set $\gs_t=\gs$ for all $t\in\mathbb C$ and we obtain the $1$-parameter family of CR algebras $(\gs_t,\gp_t)$. 
We also let $\gp=\gp|_{t=0}=\,\langle z,M,E\rangle $ and note that $\gp_t=e^{\ad_{X_t}}\cdot\,\gp$ for $X_t=-\tfrac 23it\Xi\in\wh\gs_1$.

If $t\in i\mathbb R$ is purely imaginary, then $X_t\in\gs_1$ is real and the CR algebra $(\gs_t,\gp_t)=e^{\ad_{X_t}}\cdot\,(\gs,\gp)$ is isomorphic to $(\gs,\gp)$. The complexified stabilizer $\gp_t\cap\overline\gp_t=\,\langle \xi\rangle$ is non-trivial, in fact we get a $6$-dimensional locally homogeneous CR manifold of CR-dimension $2$ and CR-codimension $2$. 
(One may show that this can be realized as an hypersurface inside the maximally symmetric homogenoeus model, but we won't need this fact.) 

If $t\notin i\mathbb R$, we write $X_t=X_{\Re(t)}+X_{i\mathfrak{Im}(t)}$, so
$\gp_t=e^{\ad_{X_{i\mathfrak{Im}(t)}}}\cdot e^{\ad_{X_{\Re(t)}}}\cdot\,\gp=
e^{\ad_{X_{i\mathfrak{Im}(t)}}}\cdot\, \gp_{\Re(t)}$, 
and
\begin{align*}
(\gs_t,\gp_t)&=e^{\ad_{X_{i\mathfrak{Im}(t)}}}\cdot\,(\gs,\gp_{\Re(t)})\\
&\cong (\gs,\gp_{\Re(t)})
\end{align*}
as CR algebras. We may then restrict to the case where the parameter $t$ is real and non-zero. We have a simply transitive action on a $7$-dimensional CR manifold of hypersurface type, with associated terms of the 
Freeman sequence
\begin{align*}
\gp_t^{-1}&=\gp_t=\left\langle \eta, \epsilon, \xi  \right\rangle\,,\quad
\gp_t^{0}=\left\langle \epsilon, \xi \right\rangle\,,\\
\gp_t^{1}&=\left\langle \xi  \right\rangle\,,\qquad
\gp_t^2=\gp_t\cap\overline\gp_t=0\;.
\end{align*}
It is therefore $3$-nondegenerate. However, as we will now explain, this is again geometrically equivalent to the maximally symmetric homogeneous model.

First of all, $\gs_t$ can be embedded in $\gg$ via the morphism of real Lie algebras $\varphi:\gs_t\to\gg$ given by the following formulae
\begin{equation}
\label{eq:immersion}
\begin{aligned}
e_{-2}&\mapsto e_{-2}- 4t (z+\overline z) + \tfrac{16}{3}t^2 (M+\overline M) - \tfrac{64}{27}t^3\Xi\;,\\
z+\overline z&\mapsto z+\overline z - \tfrac{8}{3}t (M+\overline M) + \tfrac{16}{9}t^2 \Xi\;,\\
z-\overline z&\mapsto z-\overline z + 2it E  - \tfrac{2}{3}t (M-\overline M) - \tfrac{8}{9}t^2 (N-\overline N)\;,\\
E&\mapsto E+ \tfrac{2}{3}it (N-\overline N)\;,\\
M+\overline M&\mapsto M+\overline M - \tfrac{4}{3}t \Xi\;,\\
M-\overline M&\mapsto M-\overline M + \tfrac{2}{3}t (N-\overline N)\;,\\
\Xi&\mapsto \Xi\;,
%(N-\overline N)&\mapsto (N-\overline N)\;.
\end{aligned}
\end{equation}
and their complex-linear extensions.
Since $\varphi$ is compatible with the contact filtrations of $\gs_t$ and $\gg$, and it induces the identity map at the graded level (observe the leading terms on the r.h.s. of \eqref{eq:immersion}), then it is injective. The verification that \eqref{eq:immersion} is a morphism of Lie algebras is a direct computation, which also shows that its complex-linear extension sends $\gp_t$ into $\gq$. (Alternatively, these computations were also performed in Maple and they can be found
in a supplement accompanying the arXiv posting of this article.) In summary, $\varphi:(\gs_t,\gp_t)\to(\gg,\gq)$ is an injective morphism of CR algebras.

It follows that $\varphi(\gp_t)\subset \varphi(\widehat\gs_t)\cap\gq$, where we denoted the complex-linear extension of $\varphi$ with the same symbol. The subalgebra $\varphi(\gp_t)$ is $3$-dimensional because $\varphi$ is injective, while $\varphi(\widehat\gs_t)\cap\gq$ is contained in $\gq$ and therefore is at most $4$-dimensional. If $\dim\big(\varphi(\widehat\gs_t)\cap\gq\big)=4$, then 
\begin{align*}
\varphi(\widehat\gs_t)\cap\gq=\gq&\Longrightarrow \gq\subset \varphi(\widehat\gs_t)\\
&\Longrightarrow \gq+\overline\gq\subset \varphi(\widehat\gs_t)
\end{align*}
because $\varphi(\widehat\gs_t)$ is closed by conjugation. However, the element $N-\overline N\in \gq+\overline\gq$ is not in $\varphi(\widehat\gs_t)$, showing that the assumption $\dim\big(\varphi(\widehat\gs_t)\cap\gq\big)=4$ is impossible. Thus $\varphi(\widehat\gs_t)\cap\gq$ is $3$-dimensional and it coincides with $\varphi(\gp_t)$. In other words,
$(\gs_t,\gp_t)$ is isomorphic to a CR subalgebra of $(\gg,\gq)$, so it is geometrically equivalent to the maximally symmetric model.
\end{example}

\subsection{Proof of the main results: local theory}\label{sec:3.2}

Our aim here is to prove the local homogeneous claim of Theorem \ref{T1}, namely the % first 
part of Theorem \ref{T2} concerning infinitesimal symmetries.

Let $\gg=\inf(\cM,\cD,\cJ)$ be the Lie algebra of infinitesimal CR automorphisms of a 
locally homogeneous $7$-dimensional $3$-nondegenerate CR manifold $(\cM,\cD,\cJ)$ of hypersurface type, with the contact filtration \eqref{eq:filtration} and the sequence of Freeman subalgebras $\gq\supset \gq^0\supset\gq^{1}\supset\gq\cap\overline\gq$ of the complexification $\wh\gg=\gg\otimes\bC$. 
By the results of \S\ref{sec:2}, we have the inclusion 
$\wh\gg^1\subset \gq^1+\overline\gq^1\subset\wh\gg^0$, and we also note that $\dim \gq^1/(\gq\cap\overline{\gq})=1$ and $\dim\cF^1|_x=2$ by obvious reasons. The evaluation map 
\begin{equation}
\ev_x:\wh\gg^1\rightarrow \cF^{1}|_x\cong \big(\gq^1+\overline\gq^1\big)/\gq\cap\overline\gq
\label{eq:ev-1}
\end{equation}
at a fixed point $x\in \cM$ can therefore be the zero map, have rank $1$, or be surjective.
\vskip0.1cm\par\noindent
 
 \begin{theorem}
\label{thm:dichotomy-notreally}
Every locally homogeneous $7$-dimensional $3$-nondegenerate CR manifold of hypersurface type $(\cM,\cD,\cJ)$ is locally CR diffeomorphic to the homogeneous model of \S\ref{sec:3.1}.
 \end{theorem}

\vskip0.1cm\par\noindent
The proof of this result splits in three parts. In the first part, we show that 
$(\cM,\cD,\cJ)$ is either locally CR diffeomorphic to the homogeneous
model or its symmetry algebra $\gg=\inf(\cM,\cD,\cJ)$ is a filtered deformation of $\gsl_2(\bR)\ltimes S^3\bR^2$. 
More precisely, the proof of the first part relies on the rank of the map \eqref{eq:ev-1}: we show that the case of rank $0$ is impossible, the case of rank $1$ leads to the Example \ref{ex:1-parameter-tris}, and, finally, the case of rank $2$ leads either to the homogeneous model of \S\ref{sec:3.1}, to the Example \ref{ex:1-parameter}, or to the conclusion that $\gg$ is a filtered deformation of $\gsl_2(\bR)\ltimes S^3\bR^2$. In the second part, we prove filtration rigidity, i.e., no such non-trivial deformations are possible. 
In the third and last part, we consider the complex structure on such trivial deformations and show that they are simply-transitive subalgebras of the symmetry algebra of the maximally symmetric homogeneous model (in particular, those of the Example \ref{ex:1-parameter-bis}). 

\vskip0.2cm\par\noindent
{\bf First part of the proof.}
\vskip0.2cm\par\noindent
\underline{\it The map \eqref{eq:ev-1} has rank $0$.} 
We know that the evaluation map $\ev_x:\wh\gg^0\rightarrow \cF^0|_x\cong \big(\gq^0+\overline\gq^0\big)/\gq\cap\overline\gq$ is surjective, thanks to $(v)$ of Proposition \ref{prop:basic-properties-filtration}. If \eqref{eq:ev-1} is the trivial map, we may quotient by $\widehat\gg^1$ and the map $\ev_x:\wh\gg^0/\wh\gg^1\rightarrow \cF^0|_x\cong \big(\gq^0+\overline\gq^0\big)/\gq\cap\overline\gq$ is still surjective, so $\wh\gg_0\cong \wh\gg^0/\wh\gg^1$ has dimension at least $4$. Since
$\wh\gc_0\cong\ggl_2(\bC)$ and $\wh\gg_0\subset \wh\gc_0$, we see that $\wh\gg_0=\wh\gc_0$, $\gg_0=\gc_0$, hence $\gr(\gg)$ contains the grading element $E$. By a classical result of Singer--Sternberg and Kac (see, e.g., \cite[Corollary 2.2]{CK}), the Lie algebra $\gg\cong \gr(\gg)$ is a trivial filtered deformation. 
Moreover, since $\gg$ is finite-dimensional, a classical result \cite[Proposition 3.2]{MorTan} of Morimoto and Tanaka implies that either $\gg=\gg_{-2}\oplus\gg_{-1}\oplus\gg_{0}$ or $\gg$ is isomorphic to the projective contact algebra $\mathfrak{sp}_4(\mathbb R)$. In the first case $\gg$ is simply transitive, whereas $\gg=\gg_{-2}\oplus\cdots\oplus\gg_{+2}$ for the projective contact algebra with
$\mathfrak{stab}=\gg^1=\gg_1\oplus\gg_2$ (since  \eqref{eq:ev-1} is the trivial map and by dimensional reasons). We now deal with the two cases simultaneously, with the understanding that the components of positive degree are absent in the simply transitive case.

By $3$-nondegeneracy assumption we can choose (unique up to a triangular transformation) elements $\eta\in\gq\setminus\gq^0$, $\epsilon\in\gq^0\setminus\gq^1$, $\xi\in\gq^1\setminus
\wh\stab$, where $\wh\stab=\gq\cap\overline\gq$ is the complexified stabilizer as usual, 
and write them as  $\eta=\sum_{-1\leq p\leq 2}\eta_p$, 
$\epsilon=\sum_{0\leq p\leq 2}\epsilon_p$, and $\xi=\sum_{0\leq p\leq 2}\xi_p$. By appropriately substracting elements from $$\wh\stab=\wh\gg^1=\wh\gg_1\oplus\wh\gg_2\;,$$ we may assume w.l.o.g. that 
$\eta=\eta_{-1}+\eta_0$, $\epsilon=\epsilon_0$, and $\xi=\xi_0$. 
Corollary \ref{cor:10} then tells us that
\begin{equation}
\eta=z+\eta_0\;,\quad\epsilon=z^2+\gamma z\overline z+\delta E\;,\quad \xi=\rho z\overline z+\tau E\;,
\end{equation}
by possibly rescaling $\eta$ and $\epsilon$. Here $\gamma,\delta,\rho,\tau\in\mathbb C$, with at least one of $\rho$ and $\tau$ non-zero. Since $\overline\xi=\overline\rho z\overline z+\overline\tau E\in\overline\gq^1\setminus
\wh\stab$ is not parallel to $\xi$, we see $\mathfrak{Im}(\rho\overline\tau)\neq 0$ and both $\rho$ and $\tau$ are non-zero. We may assume that $\eta_0\in\langle\overline z^2,E\rangle$ by substracting appropriate multiples of $\epsilon$ and $\xi$ to $\eta$, and that $\gamma=0$ by substracting an appropriate multiple of $\xi$ to $\epsilon$. We normalize $\xi$ with $\rho=1$. In summary, we arrive at
\begin{equation}
\eta=z+\alpha\overline z^2+\beta E\;,\quad\epsilon=z^2+\delta E\;,\quad \xi=z\overline z+\tau E\;,
\end{equation}
where $\alpha,\beta,\delta,\tau\in\mathbb C$, and $\mathfrak{Im}(\tau)\neq 0$. 

We now exploit that $\gq$ is a subalgebra of $\wh\gg$, by the integrability condition of CR manifolds. The bracket $[\eta,\epsilon]=\delta[z,E]+\alpha[\overline z^2,z^2]=\delta z+2i\alpha z\overline z$
is in $\gq$ and it has graded components of degree $-1$ and $0$. It can then be written as a linear combination of $\eta,\epsilon$ and $\xi$, and this readily implies the conditions
\begin{equation}
\label{eq:4.5}
\begin{aligned}
\alpha\delta&=0\;,\\
\beta\delta+2i\alpha\tau&=0\;.
\end{aligned}
\end{equation}
Similarly $[\eta,\xi]=[z,z\overline z]+\tau[z,E]+\alpha[\overline z^2,z\overline z]=\big(\tau-\tfrac{i}{2}\big)z+i\alpha\overline z^2$ and $[\epsilon,\xi]=[z^2,z\overline z]=-i z^2$ are elements in $\gq$, which says
\begin{equation}
\label{eq:4.6}
\begin{aligned}
\big(\tau-\tfrac{i}{2}\big)\alpha&=i\alpha\;,\\
\big(\tau-\tfrac{i}{2}\big)\beta&=0\;,\\
\delta&=0\;.
\end{aligned}
\end{equation}
Since $\tau\neq 0$, the system \eqref{eq:4.5}-\eqref{eq:4.6} is equivalent to $\delta=\alpha=\big(\tau-\tfrac{i}{2}\big)\beta=0$. In particular, we see that $\eta=\eta_{-1}+\eta_{0}=z+\beta E$, where $\eta_0=\beta E\in\gu_0$. This is the semi-aligned constraint, which in this case contradicts $3$-nondegeneracy: the iterated bracket
$
[[[\xi,\overline\eta],\overline\eta],\overline\eta]=-\big(\tau+\tfrac{i}{2}\big)\overline\beta^2 \overline z
$ belongs to $\wh\gg^{-1}=\gq+\overline\gq$, so $\xi\in\gq^2=\gq\cap\overline\gq=\wh\stab$.
In summary, this case cannot happen.
\vskip0.2cm\par\noindent
\underline{\it The map \eqref{eq:ev-1} has rank $1$.} 
There exists an element $\Xi\in\wh\gg^1$, with non-zero value $\ev_x(\Xi)\in\cF^{1}|_x$, such that $\ev_x(\Xi)$ and $\ev_x(\overline\Xi)$ are linearly dependent. In particular $\Xi\in\gq^1+\overline\gq^1$ but it is not in $\gq^1$ or in $\overline\gq^1$. %and $\ev_x(\Xi)=(1,e^{i\theta})$ for some $\theta\in\bR$, w.r.t. a basis of $\cF^{1}|_x=\cF^{1}_{10}|_x\oplus \cF^{1}_{01}|_x$ consisting of a pair of conjugate elements. 
We may then write it as the sum $\Xi=\Xi_{10}+\Xi_{01}$, $\Xi_{10}\in\gq^1\setminus \wh\stab$ and $\Xi_{01}\in\overline\gq^1\setminus\wh\stab$, with
\begin{equation}
\label{eq:the-value}
\begin{aligned}
\ev_x\bigl(\overline{\mathstrut\Xi_{01}}\bigr)&=\lambda\cdot\ev_x(\Xi_{10})\;,\\
\ev_x\bigl(\overline{\mathstrut\Xi_{10}}\bigr)&=\lambda\cdot\ev_x(\Xi_{01})\;.
\end{aligned}
\end{equation}
It is easy to see that $\lambda=e^{i\theta}$ for some $\theta\in[0,2\pi)$. We also note that $\Xi_{10}\in\wh\gg^0\setminus\wh\gg^1$ (otherwise the map \eqref{eq:ev-1} would be surjective), and similarly $\Xi_{01}\in\wh\gg^0\setminus\wh\gg^1$. %Moreover $\overline \Xi_{10}-\lambda\Xi_{10}, \overline \Xi_{01}-\lambda\Xi_{01}\in\wh\stab$.

The class $\llbracket \Xi\rrbracket\in\wh\gg_1$ is not in $\gu_1$ or $\overline\gu_1$ by Proposition \ref{prop:crucial} applied with $\gq^2=\gq\cap\overline\gq=\wh\stab$, in particular this class is non-trivial.
Finally, we emphasize that 
\begin{equation}
\label{eq:stabilizer-first}
\begin{aligned}
\wh\gg^2&\subset \gq^2+\overline\gq^2=\gq\cap\overline\gq=\wh\stab\;,\\
\wh\gg^1&\subset \langle\Xi\rangle\oplus\;\wh\stab\;,
\end{aligned} 
\end{equation}
by $3$-nondegeneracy and, respectively, since
the map \eqref{eq:ev-1} has rank $1$.

%then we are in a ``resonance'' case: there exists $\eta\in\wh\gg^1$, $\ev_x(\eta)\neq 0$, such that $\ev_x(\eta)$ and $\ev_x(\overline\eta)$ are proportional. In particular $\eta$ is neither in $\gq^1$ nor in $\overline\gq^1$, and $\ev_x(\eta)=(1,e^{i\theta})$ for some $\theta\in\bR$, w.r.t. a basis of $\cF^{1}|_x=\cF^{1}_{10}|_x\oplus \cF^{1}_{01}|_x$ consisting of a pair of conjugate elements. 
We consider any $\epsilon\in\gq^0\setminus\gq^1$ and $\xi\in\gq^1\setminus\wh\stab$. We note that $\epsilon,\xi\in\wh\gg^0$, and that $\epsilon\notin\wh\gg^1$ due to Proposition \ref{prop:basic-properties-filtration} and $\xi\notin\wh\gg^1$ too (otherwise \eqref{eq:ev-1} would be surjective).
We claim $3\leq\dim\wh\gg_0\leq 4$, with the equivalence classes $\llbracket\xi\rrbracket$, $\llbracket\epsilon\rrbracket$, $\llbracket\overline\epsilon\rrbracket$ in $\wh\gg_0\cong \wh\gg^0/\wh\gg^1$ linearly independent. Let us assume that
$
\lambda_1\llbracket\xi\rrbracket+\lambda_2\llbracket\epsilon\rrbracket+\lambda_3\llbracket\overline\epsilon\rrbracket=0
$, i.e., $\lambda_1\xi+\lambda_2\epsilon+\lambda_3\overline\epsilon\in\wh\gg^1$
and let us apply \eqref{eq:ev-1} so to get the value $\lambda_1\ev_x(\xi)+\lambda_2\ev_x(\epsilon)+\lambda_3\ev_x(\overline\epsilon)\in \cF^{1}|_x$. Since $\ev_x(\xi)\in\cF^{1}_{10}|_x$, we have that
$\lambda_2\ev_x(\epsilon)\in\cF^{1}_{10}|_x$ and $\lambda_3\ev_x(\overline\epsilon)\in 
\cF^{1}_{01}|_x$. It then follows from $\epsilon\notin\gq^1$ that $\lambda_2=\lambda_3=0$
and we finally infer that $\lambda_1=0$ as well, since $\lambda_1\llbracket\xi\rrbracket=0$ and $\xi\notin\wh\gg^1$. This proves $3\leq \dim\wh\gg_0\leq 4$.
\vskip0.2cm\par

If $\dim\wh\gg_0=4$, then $\gg\cong\gr(\gg)$ is a trivial filtered deformation, with 
$\wh\gg_{p}=\wh\gc_p$ for $p=-2,-1,0$. Moreover $\wh\gg_1\neq 0$, since the class $\llbracket \Xi\rrbracket$ is non-trivial, so $\gg=\gg_{-2}\oplus\cdots\oplus\gg_{+2}$ is isomorphic to the projective contact algebra by \cite[Proposition 3.2]{MorTan}. However $\gg_1=\mathfrak{z}_1\cong S^1(\gc_{-1})$ is the unique $\gk_0$-submodule that is complementary to $\gk_1$ inside $\gc_1$ (see the discussion before Definition \ref{def:13} and \cite[Proposition 3.2]{MorTan}), so that $\llbracket \Xi\rrbracket\in\wh\gg_1=\wh{\mathfrak{z}}_1\subset\gu_{1}\cap\overline\gu_{1}$. This is a contradiction.

%Proposition \ref{prop:crucial} 
Therefore the classes
$\llbracket\xi\rrbracket, \llbracket\epsilon\rrbracket, \llbracket\overline\epsilon\rrbracket$ in $\wh\gg_0$ are linearly independent, generate $\wh\gg_0$, and $\dim\wh\gg_0$ is exactly $3$.
In this case, Corollary \ref{cor:10}
tells us that
\begin{equation}
\label{eq:xiandepsilon-symbol}
\begin{aligned}
\llbracket\xi\rrbracket&=\alpha z\overline z+\beta E\;,\\
\llbracket\epsilon\rrbracket&=z^2+\gamma z\overline z+\delta E\;,
\end{aligned}
\end{equation}
by possibly rescaling $\epsilon$. Here $\alpha,\beta,\gamma,\delta\in\mathbb C$, with at least $\alpha$ or $\beta$ different from zero. If $\alpha\neq0$, then
$
[\llbracket\xi\rrbracket,\llbracket\epsilon\rrbracket]=[\alpha z\overline z,z^2]=i\alpha z^2
$ and $[z^2,\overline z^2]=-2i z\overline z$, so that $\wh\gg_0=\langle z^2,\overline z^2,z\overline z\rangle\cong\gsl_2(\mathbb C)$. Again \cite[Proposition 3.2]{MorTan} tells us that $\gr(\wh\gg)=\wh\gg_{-2}\oplus\wh\gg_{-1}\oplus\wh\gg_{0}$. %However 
%the class $\llbracket\eta\rrbracket$ of $\eta$ in $\wh\gg_1$ is non-zero, 
%
Thus $\dim(\wh\gg)=\dim(\gr\wh\gg)=6$, which is a contradiction.

We have thus shown that $\alpha=0$, $\wh\gg_0=\langle E,z^2+\gamma z\overline z,\overline z^2+\overline\gamma z\overline z\rangle$ and $\gg$ is a trivial filtered deformation. 
Now
$
[z^2+\gamma z\overline z,\overline z^2+\overline\gamma z\overline z]=-2iz\overline z-i\overline\gamma z^2-i\gamma\overline z^2\in\wh\gg_0
$,
so that $2i(\|\gamma\|^2-1)z\overline z\in\wh\gg_0$. Since $\dim\wh\gg_0=3$, we have that $\|\gamma\|^2=1$, $\gamma=e^{i\vartheta}$ for some $\vartheta\in[0,2\pi)$, and $\widehat\gg_0$ equals the Borel subalgebra of $\wh\gc_0\cong\ggl_{2}(\bC)$ given by
$$
\mathfrak{b}_{\vartheta}=\left\langle E, z^2+e^{i\vartheta}z\bar z, \bar z^2+e^{-i\vartheta}z\bar z\right\rangle\;.
$$ 
In this case \cite[Proposition 3.2]{MorTan} does not apply and we need much finer arguments, which we split in different claims.
We preliminary note that, by rescaling $z$ to $e^{-i\vartheta/2} z$, we may assume w.l.o.g. that $\vartheta=0$ and that 
\begin{equation*}
\wh\gg_0=\left\langle E, M, \overline M\right\rangle
\end{equation*}
%z^2+z\bar z, \bar z^2+z\bar z
 is the Borel subalgebra stabilizing the line of $z+\overline z$ in $\wh\gg_{-1}$.
\vskip0.2cm\par
{\it First claim: the stabilizer subalgebra is graded in positive degrees.}
\vskip0.2cm\par
By appropriately subtracting elements from $\wh\gg^2\subset\wh\stab$, we may take $\Xi_{10}$ and $\Xi_{01}$ in $\wh\gg_0\oplus\wh\gg_1$. Hence
$\Xi=\Xi_{10}+\Xi_{01}\in\wh\gg_1$, and
$$\wh\gg_1=\langle\Xi\rangle\oplus\,\big(\wh\stab\cap\wh\gg_1\big)\;.$$
By simple dimensional reasons, we arrive at $\wh\stab=\big(\wh\stab\cap\wh\gg_1\big)\oplus\wh\gg^2$.
In particular the stabilizer subalgebra is graded. (Since the grading element is not in the stabilizer subalgebra, this claim is {\it not} immediate.) 
\vskip0.2cm\par
{\it Second claim: the a priori estimate $1\leq\dim\wh\gg_1\leq 3$.}
\vskip0.2cm\par

We now make use of some results from \cite[Step 2, proof Thm. 6.1]{San}. 
Clearly, the space $\wh\gg_1$ is a non-trivial $\wh\gg_0$-submodule of
the first prolongation 
\begin{equation}
\label{eq:first-prol}
\widetilde\gg_1=\left\{X\in\wh\gc_1\mid[X,\wh\gc_{-1}]\subset \wh\gg_0\phantom{C^{C^C}}\!\!\!\!\!\!\!\!\!\!\!\right\}
\end{equation}
of the Borel subalgebra $\wh\gg_0$. The space \eqref{eq:first-prol} is $4$-dimensional, more precisely, it is generated by the elements
$\displaystyle
\left\{N,\overline N, V, W\!\!\!\!\!\!\!\!\!\!\!\!\!\phantom{C^{C^C}}\right\}
$,
where
\begin{align*}
V&=z^2\bar z+z\bar z^2+\frac{i}{2}z-\frac{i}{2}\bar z\;,\qquad\qquad\\
W&=z+\bar z\;,\\
N&=z^3+2z^2\bar z+z\bar z^2-3i z -3i\bar z\;.
\end{align*}
We note that $V=\overline V$, $W=\overline W$.
The element $N$ has already appeared in \S\ref{sec:3.1}, and it is characterized by the following property: it is the unique non-trivial element in $\widetilde\gg_1\cap(\gu_1\setminus\overline\gu_1)$ which commutes with its conjugate. 
The adjoint action of $\wh\gg_0$ on $\wt\gg_1$ is given by the obvious action of the grading element and the following formulae, together with their conjugates:
\begin{equation}
\label{eq:brackets}
\begin{aligned}
[M,N]&=-\tfrac 12 iN\;,\\
[\overline M,N]&=\tfrac 32 i N+i \overline N\;,\\
[M,V]&=-iN+\tfrac i2 V+\tfrac{5}{2}W\;,\\
[M,W]&=-\tfrac i2 W\;.
\end{aligned}
\end{equation}

By (ii) of Proposition \ref{prop:crucial} and $3$-nondegeneracy, we get the following chain of inclusions
\begin{align*}
\wh\stab\cap\wh\gg_1&\subset \big(\gu_1\cap\overline\gu_1\big)\cap\wh\gg_1\\
&\subset \big(\gu_1\cap\overline\gu_1\big)\cap\wt\gg_1 =\langle V,W\rangle
\end{align*}
and the a priori estimate $1\leq \dim\wh\gg_1\leq 3$. 
\vskip0.2cm\par
{\it Third claim: the element $\Xi=N+\overline N$.}
\vskip0.2cm\par

Since $\Xi=\llbracket \Xi\rrbracket$ is not in $\gu_1$ or in $\overline\gu_1$, we may write it as
$
\Xi=a N+b V+c W+d\overline N
$,
for some $a,b,c,d\in\mathbb C$ with $a$, $d$ non-zero. As already advertised below \eqref{eq:the-value}, we have $\overline\Xi-e^{i\theta}\Xi\in\wh\stab$
, and this translates in the condition $\overline d=e^{i\theta} a$. Normalizing $\Xi$ so that $a=1$, we finally arrive at
\begin{equation*}
\Xi=N+b V+c W+e^{-i\theta}\overline N\;.
\end{equation*}
%It follows that
%$
%\Xi\equiv z^3+(2+b+e^{-i\theta})z^2\overline z+(1+b+2e^{-i\theta})z\overline z^2+e^{-i\theta}\overline z^3\mod S^1(\wh\gc_{-1})
%$,
A direct computation then gives
\begin{align*}
[M,\Xi]&\equiv -i\big(\tfrac 12+b+e^{-i\theta}\big)z^3-\tfrac 32 i e^{-i\theta}\overline z^3\mod\langle z^2\overline z, z\overline z^2,z,\overline z\rangle\;,\\
[\overline M, \Xi]&\equiv \tfrac 32 i z^3+i\big(1+b+\tfrac 12 e^{-i\theta}\big)\overline z^3\;\;\;\;\;\;\,\,\mod\langle z^2\overline z, z\overline z^2,z,\overline z\rangle\;,
\end{align*}
and both terms have to be proportional to $\Xi\!\mod\langle z^2\overline z, z\overline z^2,z,\overline z\rangle=z^3+e^{-i\theta}\overline z^3$, since the map \eqref{eq:ev-1} has rank $1$. This leads to $b=1-e^{-i\theta}$ and $b=e^{-i\theta}-1$, that is, $b=0$ and $e^{-i\theta}=1$. Thus 
$\Xi=N+cW+\overline N$ and we then observe that
$
[M,\Xi]+\tfrac 32 i\Xi=i cW\in\wh\gg_1$. If $c=0$, then $\Xi=N+\overline N$. If $c\neq 0$, then $W\in\wh\gg_1$, so that $W\in \wh\stab\cap\wh\gg_1$ thanks to Proposition \ref{prop:crucial} applied with $\gq^2=\wh\stab$, and once more we may assume that $\Xi=N+\overline N$. 
\vskip0.2cm\par
{\it Fourth claim: the refined estimate $1\leq\dim\wh\gg_1\leq 2$.}
\vskip0.2cm\par

We established the chain of inclusions
\begin{equation}
\label{eq:inclusion-preliminary}
\langle\Xi\rangle\,\subset\, \wh\gg_1=\,\langle\Xi\rangle\oplus\,\big(\wh\stab\cap\wh\gg_1\big)\subset\, \langle\Xi, V,W\rangle\;,
\end{equation}
where $\Xi=N+\overline N$. If an element of the form $V+\delta W\in\wh\gg_1$, we may bracket it with $M$ and, using \eqref{eq:brackets}, get an element that contradicts the last inclusion in \eqref{eq:inclusion-preliminary}. In summary 
$$\wh\stab\cap\wh\gg_1\subset\,\langle W\rangle\quad\text{and}\quad
\wh\gg_1\subset\, \langle\Xi, W\rangle\;.
$$
Using \eqref{eq:brackets}, we see that the lines of $\Xi$ and $W$ are both $\wh\gg_0$-stable.
\vskip0.2cm\par\noindent
\vskip0.2cm\par
{\it Fifth and last claim: the spaces $\wh\gg_1=\langle\Xi\rangle$ and $\wh\gg^2=0$.}
\vskip0.2cm\par
By the discussion below \eqref{eq:xiandepsilon-symbol} on the equivalence classes of $\epsilon\in\gq^0\setminus\gq^1$ and $\xi\in\gq^1\setminus\wh\stab$, and by subtracting appropriate elements from $\wh\stab$, we may write
\begin{equation}
\begin{aligned}
\epsilon=\epsilon_0+\epsilon_1&=\big(M+\delta E\big)+\tau \Xi\;,\\
\xi=\xi_0+\xi_1&=E+\rho\Xi\;,
\end{aligned}
\end{equation}
by possibly rescaling $\xi$. Here $\delta,\tau,\rho\in\mathbb C$ and, by subtracting a multiple of $\xi$ to $\epsilon$, we set $\delta=0$. Finally, we may choose $\eta\in\gq\setminus\gq^0$ of the form 
\begin{equation}
\eta=\eta_{-1}+\eta_0+\eta_1=z+\mu\overline M+\nu\Xi\;,
\end{equation}
for some $\mu,\nu\in\mathbb C$. Since $\gq=\,\langle\eta,\epsilon,\xi\rangle\oplus\,\big(\wh\stab\cap\wh\gg_1\big)$ is a subalgebra and the components of degree $1$ of the brackets of elements in $\langle\eta,\epsilon,\xi\rangle$ are always parallel to $\Xi$, we see that in fact $\langle\eta,\epsilon,\xi\rangle$ is a subalgebra. 
 % A tad long but straightforward computation says 
A straightforward computation shows that this condition is equivalent to the following system of equations 
\begin{equation}
\label{eq:brackets-symmetries-q}
\begin{aligned}
\mu&=\tau\;,\\
 % 2i\nu&=-\tfrac i2\mu\big(\mu+3\tau\big)\;,\\
 % \mu&=-i\tfrac 32 \rho\;,\\
2\nu&=3i\rho\mu\;,\\
\tau&=-i\tfrac32 \rho\;.
\end{aligned}
\end{equation}
Its solution is given by $\tau=\mu$, $\nu=-\mu^2$, $\rho=\tfrac 23 i\mu$.

We now claim that $\wh\stab=0$. First of all, using the explicit expressions of the Lie brackets of $\wh\gg_{1}$ with $\wh\gg_{-1}$ in \cite[Step 2, proof Thm. 6.1]{San}, we see that 
\begin{align*}
[W,\eta]&=[W,z]+\mu[W,\overline M]-\mu^2[W,\Xi]
\equiv[W,z]%+\mu[W,\overline M]
\!\!\mod\wh\stab\\
&=\tfrac12 M-\tfrac 14 iE\\
&=\tfrac12 \epsilon+\tfrac 14 i\xi-\tfrac 12 i\overline\xi
\end{align*}
does not belong to $\gq$. This contradicts $W\in\wh\gg_1$, from which $\wh\gg_1=\langle\Xi\rangle$. Since its prolongation in degree $\geq 2$ is trivial by \cite[Step 3, proof Thm. 6.1]{San}, we finally see that $\wh\gg^2=0$ and $\wh\stab=0$. 

In summary, we have arrived to the CR algebras of Example \ref{ex:1-parameter-tris}. Those corresponding to $7$-dimensional manifolds are geometrically equivalent to the maximally homogeneous model.

%Now $\Xi_{10}\equiv\upsilon \xi\mod\wh\stab\cap\wh\gg_1$ and $\Xi_{01}\equiv\chi\overline\xi\mod\wh\stab\cap\wh\gg_1$, for some non-zero $\upsilon,\chi\in\mathbb C$. Actually $\Xi=\Xi_{10}+\Xi_{01}$, whence $\chi=-\upsilon$, 
%\begin{align*}
%\Xi&\equiv\upsilon\rho\Xi-\upsilon\overline\rho\overline\Xi\mod\wh\stab\cap\wh\gg_1\\
%&\equiv \upsilon\big(\rho-\overline\rho e^{i\theta})\Xi\mod\wh\stab\cap\wh\gg_1\;,
%\end{align*}
%and $\upsilon\big(\rho-\overline\rho e^{i\theta}\big)=1$. In particular $\rho\neq 0$.

\vskip0.2cm\par\noindent
\underline{\it The map \eqref{eq:ev-1} has rank $2$.} 
We are left to study the case where \eqref{eq:ev-1} is surjective. In this case, there exist $\eta\in\gq\setminus\gq^0$, $\epsilon\in\gq
^0\setminus\gq^1$ and $\xi\in\wh\gg^1\cap\gq^1$ such that $\ev_x(\xi)\neq 0$. (In particular
$\ev_x(\xi)$ and $\ev_x(\overline\xi)$ span the image of \eqref{eq:ev-1}.) Their equivalence classes $\llbracket\eta\rrbracket\in\wh\gg_{-1}$,
$\llbracket\epsilon\rrbracket\in\wh\gg_0$ and $\llbracket \xi\rrbracket\in\wh\gg_1$
are in $\gu$ but not in $\overline\gu$, by Corollary \ref{cor:10}. Moreover $$\wh\gg^2\subset\gq^2+\overline\gq^2=\gq\cap\overline\gq=\wh\stab\;,$$ 
so that $\wh\gg_p\subset\gu_p\cap\overline\gu_p$ for all $p\geq 2$.
\vskip0.2cm\par\noindent

If $E\in\gr(\gg)$, then one verifies directly that $\gr(\gg)$ is a $3$-nondegenerate $7$-dimensional homogeneous model  in the sense of Definition \ref{ukip} and \cite[Theorem 6.1]{San} says that there is only one such model, up to isomorphism. Since $E\in\gr(\gg)$, the Lie algebra $\gg$ is a trivial filtered deformation, so 
$\gg\cong\gr(\gg)$ is the graded Lie algebra described in \S\ref{sec:3.1}. If there is an identification $\gg\cong\gr(\gg)$ that is aligned, then
$\gq=\wh\gg\cap\gu$ by Corollary \ref{cor:stab!}, i.e., we obtain the CR algebra $(\gg,\gq)$ of the homogeneous model of \S\ref{sec:3.1}. In general, via appropriate normalizations as usual, we may write
\begin{align*}
\eta&=z+\beta E+\mu \overline M+\nu\overline N\;,\\
\epsilon&=M+\delta E+\tau \overline N\;,\\
\xi&=N\;,
\end{align*}
where $\beta,\mu,\nu,\delta,\tau\in\mathbb C$. The complexified stabilizer $\wh\stab$ is $1$-dimensional and generated by a non-zero element of the form $\lambda_1 E+\lambda_2 N+\lambda_3\overline N$, thanks to Proposition \ref{prop:crucial}. Clearly $\lambda_1\neq 0$, again by Proposition \ref{prop:crucial} and the fact that $\wh\gg_1\cap\gu_1\cap\overline\gu_1=0$. We normalize $\lambda_1=1$ and,
since the complexified stabilizer is stable by conjugation, we finally see that
$$\wh\stab=\langle E_o:=E+\lambda_2N+\overline\lambda_2\overline N\rangle\;,$$ 
where $\lambda_2\in\mathbb C$. In particular, we may substract appropriate multiples of $E_o$ and $\xi$ to $\eta$ and $\epsilon$ so to arrange for $\beta=\delta=0$.

We now compute 
\begin{align*}
[E_o,\epsilon]&=\tau\overline N+\lambda_2[N,M]
+\overline\lambda_2[\overline N,M]\\
&=\big(\lambda_2\tfrac i2+i\overline\lambda_2\big)N+\big(\tau+\tfrac32 i\overline\lambda_2\big)\overline N\;,
\end{align*}
which has to be an element of $\gq$. Hence $\tau+\tfrac32 i\overline\lambda_2=0$. If $\lambda_2=0$, then $\wh\stab=\langle E\rangle$. Since $\gq$ is stable under the adjoint action of the stabilizer, it is $\mathbb Z$-graded and equal to $\langle z,E,M,N\rangle$; i.e., $\mu=\nu=\tau=0$, and this is, again, our homogeneous model of \S\ref{sec:3.1}. If $\mu=\nu=0$, then $\gg$ is aligned in the sense of Definition \ref{def:aligned}, $\tau=0$ by Lemma \ref{lemma:semi-aligned} and then $\lambda_2=0$ again.
%Moreover $\gq\subset \wh\gg\cap\gu$ again by Lemma \ref{lemma:semi-aligned}, so $\wh\stab=\gq\cap\overline\gq\subset\wh\gg\cap\big(\gu\cap\overline\gu\big)=<E>$, which is actually an equality by dimensional reasons. We see that $\gq$ is $\mathbb Z$-graded and equal to $<z,E,M,N>$ -- this is, once more, our homogeneous model of \S\ref{sec:3.1}. 

We are then led to study the case where $\gq$ is generated by
\begin{align*}
\eta&=z+\mu \overline M+\nu\overline N\;,\\
\epsilon&=M-\tfrac32 i\overline\lambda_2 \overline N\;,\\
\xi&=N\;,\\
E_o&=E+\lambda_2 N+\overline\lambda_2\overline N\;,
\end{align*}
with $\lambda_2\neq 0$ and at least one of $\mu$ and $\nu$ non-zero as well. Using the Lie brackets \eqref{eq:brackets-model}, the fact that $\gq$ is a subalgebra turns out to be equivalent to the following system of equations:
 %\begin{align*}
%\mu=\tau&=-\tfrac 32 i\overline \lambda_2\;,\\
%2\nu-i\lambda_2\mu-\tfrac i2\mu\overline\lambda_2+\tfrac 32 i\operatorname{Im}(\lambda_2)\overline\lambda_2+\tfrac i2 \lambda_2\tau-2i\tau\overline\lambda_2&=0\;,\\
%i\mu-\tfrac 34\overline\lambda_2-\tfrac i2 \tau&=0\;,\\
%-\tfrac i2 \mu\tau+2i\nu+\tfrac 34 \tau\overline\lambda_2+2i\tau^2&=0\;.
 %\end{align*}
 %
 \begin{gather*}
\mu=-\tfrac 32 i\overline \lambda_2\;,\\
 % 2\nu+i\lambda_2\big(\tfrac 12 \tau-\mu\big)-i\overline\lambda_2\big(\tfrac 12\mu+2\tau\big)
 % +\tfrac 32 i\operatorname{Im}(\lambda_2)\overline\lambda_2&=0\;,\\
 % i\mu-\tfrac 34\overline\lambda_2-\tfrac i2 \tau&=0\;,\\
-\tfrac i2 \mu\tau+2i\nu+\tfrac 34 \tau\overline\lambda_2+2i\tau^2=0\;.
 \end{gather*}
We omit the somewhat long but straightforward check.  It solution is $\mu=\tau=-\tfrac 32 i\overline \lambda_2$, $\nu=-\tau^2$; in other words, we obtained the $1$-parameter family described in Example \ref{ex:1-parameter}. As explained there, this is nothing but the maximally symmetric homogeneous model in disguise.
\vskip0.2cm\par\noindent

If $E\notin\gr(\gg)$, then $\ga=\gr(\gg)\rtimes\mathbb RE$ is a $3$-nondegenerate $7$-dimensional homogeneous model in the sense of Definition \eqref{ukip} and again we may apply \cite[Theorem 6.1]{San}.
In this case $\dim\ga=8$ and $\dim(\gr\gg)=7$. Clearly $e_{-2}, z,\overline z\in\gr(\wh\gg)$ and $N,\overline N$ too, since $\llbracket\xi\rrbracket, \llbracket\overline\xi\rrbracket\in \wh\gg_1$ and $\dim\ga_1=2$. Since
\begin{equation}
\begin{aligned}
[N,z]&=-\frac{i}{2}M-\frac{3}{4}E\;,\quad [\overline N,\overline z]=\frac{i}{2}\overline M-\frac{3}{4}E\;,\\
[N,\overline z]&=-\frac{3}{2}iM-2i\overline M+\frac{3}{4}E\;,\quad
[\overline N,z]=\frac{3}{2}i\overline M+2i M+\frac{3}{4}E\;,
\end{aligned}
\end{equation}
we have $\wh\gg_0=\left\langle 2i M+3E, 2i\overline M-3E\right\rangle$ and $\gr(\gg)\cong \gsl_2(\bR)\ltimes S^3\bR^2$ as an abstract Lie algebra. 

We will now turn to show that the graded Lie algebra $\gr(\gg)$ is filtration rigid, i.e., $\gg\cong\gr(\gg)$ as filtered Lie algebra with filtrands $\gg^p=\bigoplus_{j\geq p}\gg_{j}$.
\vskip0.2cm\par\noindent
{\bf Second part of the proof.}
\vskip0.2cm\par\noindent
We now study the remaining simply-transitive case in more detail. The graded Lie algebra $\gr(\gg)\cong \gsl_2(\bR)\ltimes S^3\bR^2$
has components
$$
\gg_p=\begin{cases}
0\;\;\;\;\;\;\;\;\;\;\;\;\;\;\;\;\;\;\;\;\;\;\;\;\;\;\;\;\;\;\;\;\;\;\;\;\;\;\;\;\;\;\;\;\;\;\;\;\;\;\;&\text{for all}\;\;p>1\;,\\
\Re\left\langle N, \overline N\right\rangle\,\,\,\;\;\;\;\;\;&\text{for}\;\;p=1\;,\\
\Re\left\langle L,\overline L\right\rangle\,\;\;\;\;\;&\text{for}\;\;p=0\;,\\
\gc_{p}\,\;\;\;\;\;\;\;\;\;&\text{for}\;\;p=-2, -1\;,\\
%0\;\;\,\;\;\;\;\;\;\;\;&\text{for all}\;\;p>1\;,\\
\end{cases}
$$
where $L=2i M+3E$, and the following formulae (together with their conjugates) give the non-trivial structure relations of this Lie algebra:
\begin{equation}
\label{eq:brackets-smaller}
\begin{aligned}
[z,\overline z]&=-\frac{i}{2}e_{-2}\;,\quad [L,z]=-4z\;,\quad [L,\overline z]=2z-2\overline z\;,\quad
[L,e_{-2}]=-6e_{-2}\;,\\
[L,\overline L]&=-2L+2\overline L\;,\quad [N,e_{-2}]=-3i(z+\overline z)\;,\quad[N,z]=-\frac{1}{4}L\;,\\
[N,\overline z]&=-\frac{3}{4}L+\overline L\;,\quad [L,N]=4N\;,\quad [\overline L,N]=6N+2\overline N\;.
\end{aligned}
\end{equation}
Infinitesimal filtered deformations are governed by Spencer cohomology groups $H^{d,2}(\gg_-,\gr(\gg))$ in positive degrees $d$. The following two results come by direct computations, which we omit.
\begin{lemma}
\label{lem:cohomology}
The group $H^{d,2}(\gg_-,\gr(\gg))$ vanishes for $d=1$ and all $d>4$. On the other hand, we have:
\vskip0.3cm\par\noindent
\begin{itemize}
	\item[(i)]
$H^{2,2}(\gg_-,\gr(\gg))$ is $1$-dimensional and it is generated by the map $\psi^2:\gg_{-2}\otimes\gg_{-1}\to\gg_{-1}$ given by 
\begin{equation}
\psi^2(e_{-2},z)=i(z-\overline z)\;,\quad \psi^2(e_{-2},\overline z)=i(z-\overline z)\;.
\end{equation}
\item[(ii)] $H^{3,2}(\gg_-,\gr(\gg))$ is $2$-dimensional and it is generated by the maps $\psi^3_i:\gg_{-2}\otimes\gg_{-1}\to\gg_{0}$, $i=1,2$, given by 
\begin{equation}
\begin{aligned}
\psi^3_{1}(e_{-2},z)&=L\;,\quad\psi^3_{1}(e_{-2},\overline z)=-\overline L\;,\\
\psi^3_{2}(e_{-2},z)&=L+\overline L\;,\quad\psi^3_{2}(e_{-2},\overline z)=L+\overline L\;.
\end{aligned}
\end{equation}
\item[(iii)] $H^{4,2}(\gg_-,\gr(\gg))$ is $2$-dimensional and it is generated by the maps $\psi^4_i:\gg_{-2}\otimes\gg_{-1}\to\gg_{1}$, $i=1,2$, given by 
\begin{equation}
\begin{aligned}
\psi^4_{1}(e_{-2},z)&=N+7\overline N\;,\quad\psi^4_{1}(e_{-2},\overline z)=7 N+\overline N\;,\\
\psi^4_{2}(e_{-2},z)&=N+\overline N\;,\quad\psi^4_{2}(e_{-2},\overline z)=-(N+\overline N)\;.
\end{aligned}
\end{equation}
\end{itemize}
\end{lemma}
\begin{lemma}
\label{lem:spectrum}
The spectrum of the adjoint action of the element $\widetilde E=-\tfrac{1}{4}\big(L+\overline L)$ on $\gr(\gg)$ is as follows:
\begin{align}
 \begin{array}{|c|c|c|c|c|c|c|}\hline
 e_{-2}& z+\overline z  & z-\overline z    &  L+\overline L & L-\overline L & N+\overline N^{\phantom{c^c}}\!\! & N-\overline N \\ \hline\hline
  3 &  1 &  2& 0 & -1 & -3 & -2 \\ \hline
 \end{array}
 \end{align}
\end{lemma}
It then follows that all the cocycles displayed in Lemma \ref{lem:cohomology} are rescaled by the action of $\widetilde E$, and the rescaling is never trivial (the eigenvalues are $-2$, $-5$, $-4$, $-7$, and $-8$, respectively). 
Using Lemmas \ref{lem:cohomology}-\ref{lem:spectrum}, we now set to prove the following.
\vskip0.1cm\par\noindent
\begin{proposition}
The graded Lie algebra $\gr(\gg)\cong \gsl_2(\bR)\ltimes S^3\bR^2$ is filtration rigid.
\end{proposition}
By the first part of the proof, we need to consider filtered deformations of $\gr(\gg)\cong\gsl_2(\bR)\ltimes S^3\bR^2$. We first note that $\gr(\gg)$ is an almost full prolongation (of degree $1$) in the sense of \cite{CK}. In fact
$$
\Hom_{\mathfrak t}\big(H^{d,1}(\gg_-,\gr(\gg)),\gg_1)=0\;\;\text{for all}\;\;d\geq 1\;,
$$ 
where $\mathfrak t=\langle \widetilde E \rangle$ is the maximal reductive subalgebra of $\gg_0$. We omit the straightforward check, which uses the eigenvalues of $\widetilde E$ on Spencer cochains and the brackets \ref{eq:brackets-smaller}. 

Now, it is well-known that the restriction to $\gg_-$ of the first non-zero contribution of a filtered deformation is a cohomology class in positive degree which is $\gg_0$-invariant (see \cite[Prop. 2.2]{CK}) and that, in case of a coboundary, this can be absorbed via redefinition of the complementary subspaces in the chain of filtrands (see \cite[Prop. 2.3]{CK}). By Lemmas \ref{lem:cohomology}-\ref{lem:spectrum}, this is our case. Being an almost full deformation, the same is true for all the contributions: we may apply \cite[Cor. 2.3]{CK} and infer that $\gr(\gg)\cong\gsl_2(\bR)\ltimes S^3\bR^2$ has no non-trivial filtered deformations. The reader interested in more explicit details for these last steps may use the following argument.
\vskip0.1cm\par\noindent

Let $X,\widetilde E,Y$ be the standard basis of the Levi factor $\mathfrak{sl}_2(\mathbb R)$ of $\gr(\gg)$ as in \S\ref{sec:3.1} and $v_0,v_1,v_2,v_3$ the basis of the $4$-dimensional radical corresponding to the elements
$x^3,x^2y,xy^2,y^3$ in $S^3\R^2$. In this basis the contact filtration is the following:
 $$
\gg^{-2}=\gg,\ \gg^{-1}=\gsl_2(\bR)\ltimes \langle v_1,v_2,v_3\rangle,\
\gg^0=\langle \widetilde E,Y,v_2,v_3\rangle,\ \gg^1=\langle Y,v_3\rangle.
 $$
In a filtered deformation of $\gr(\gg)$, we are allowed to modify the brackets 
$[\xi,\eta]$ for $\xi\in\gg^i$, $\eta\in\gg^j$ by terms from $\gg^{i+j+1}$.
We can modify $\xi$ by $\gg^{i+1}$ and $\eta$ by $\gg^{j+1}$, and if this restores the graded brackets,
then the algebra is filtration rigid.

Since $\widetilde E$ has a simple spectrum, the filtered deformation can be assumed to be compatible with it via a redefinition of the complementary subspaces in the chain of filtrands. In order to preserve the Jacobi Identities, the deformations of the brackets should not only respect the filtration but also the grading w.r.t. the adjoint action of $\widetilde E$. With these restrictions, one can directly see that no deformation terms exist for the Lie brackets of 
$\gsl_2(\bR)$ with itself as well as with $S^3\bR^2$, i.e., both the Levi factor and its
representation are filtration rigid. On the other hand, the Lie brackets of $S^3\bR^2$ with itself do admit possible non-trivial deformations terms, which are indicated by the real parameters $\lambda_i$ in the last line:
 \begin{gather*}
[\widetilde E,X]=2X,\ [\widetilde E,Y]=-2Y,\ [X,Y]=\widetilde E,\\ 
[\widetilde E,v_0]=3v_0,\ [\widetilde E,v_1]=v_1,\ [\widetilde E,v_2]=-v_2,\ [\widetilde E,v_3]=-3v_3,\\
[X,v_1]=v_0,\ [X,v_2]=2v_1,\ [X,v_3]=3v_2,\ [Y,v_0]=3v_1,\ [Y,v_1]=2v_2,\ [Y,v_2]=v_3,\\ 
[v_0,v_2]=\lambda_0X,\ [v_0,v_3]=\lambda_1 \widetilde E,\ [v_1,v_2]=\lambda_2 \widetilde E,\ [v_1,v_3]=\lambda_3Y. 
 \end{gather*}
It turns out that the Jacobi Identities rule out all of them, that is, $\lambda_i=0$ for all $i=1,\ldots 3$.
For instance, the following Jacobi Identities $\operatorname{Jac}(Y,v_0,v_1)=-2\lambda_0 X$, $\operatorname{Jac}(X,v_2,v_3)=-2\lambda_3 Y$, $\operatorname{Jac}(v_0,v_2,v_3)=-(3\lambda_0+\lambda_1)v_2$, $\operatorname{Jac}(v_1,v_2,v_3)=(3\lambda_2+\lambda_3)v_3$
imply the desired result. Here $\operatorname{Jac}$ is the Jacobiator, which measures the failure of the brackets to satisfy the Jacobi Identity.

Again, we conclude that $\gr(\gg)$ is filtration rigid.
\vskip0.2cm\par\noindent
{\bf Third and last part of the proof.}
\vskip0.2cm\par\noindent
If there exists an identification $\gg\cong \gr(\gg)$ that is aligned in the sense of Definition \ref{def:aligned}, then the complex structure is completely determined by part (1) of Theorem \ref{thm:25} (see also Corollary \ref{cor:stab!}): we have a simply-transitive subalgebra of the full CR algebra of infinitesimal automorphisms of the maximally symmetric homogeneous model. 

Otherwise we consider $\eta\in\gq\setminus\gq^0$, $\epsilon\in\gq
^0\setminus\gq^1$ and $\xi\in\wh\gg^1\cap\gq^1$ such that $\ev_x(\xi)\neq 0$. Their equivalence classes 
$\llbracket \eta\rrbracket\in \wh\gg_{-1}$, $\llbracket\epsilon\rrbracket\in\wh\gg_0$ and $\llbracket \xi\rrbracket\in\wh\gg_1$
are in $\gu$ but not in $\overline\gu$, by Corollary \ref{cor:10}. Via appropriate normalizations, we may write
\begin{align*}
\eta&=z+\mu \overline L+\nu\overline N\;,\\
\epsilon&=L+\tau \overline N\;,\\
\xi&=N\;,
\end{align*}
for some $\mu,\nu,\tau\in\mathbb C$. Since $\wh\stab=0$ by dimensional reasons, these three vectors generate $\gq$. Using \eqref{eq:brackets-smaller}, we see that $\gq$ is a subalgebra if and only if
$\tau=8\mu$ and $2\mu\tau-10\nu+\tau^2=0$, i.e., $\tau=8\mu$ and $\nu=8\mu^2$. In other words, we obtained the $1$-parameter family of Example \ref{ex:1-parameter-bis} and 
it was shown in \S\ref{sec:section-added} that this is 
the maximally symmetric homogeneous model. 
\vskip0.2cm\par\noindent
{\bf End of the proof.}
\vskip0.2cm\par\noindent

\section{Global models}\label{sec:3.3}

We now complete the proof of the second part of Theorem \ref{T2} concerning the global behaviour of the model. 
We need the universal cover of the automorphism group, but since universal covers of disconnected groups 
are not well-known, we will first discuss them here.

\subsection{The universal cover of $GL_2(\R)$}\label{sec:3.3-prerequisite}

For any Lie group $G$, its connected component $G^o$ of the unity $E$ is a normal subgroup, so
$\pi_0(G)=G/G^o$ can be considered as a discrete group. Moreover, the fundamental group $\pi_1(G^o)$ is Abelian, so we will use additive notation for it. 
In what follows, we are going to relax the connectedness assumption for the universal cover. 

The universal cover of $SL_2(\R)$ is well-known, it is a non-algebraic simple Lie group with Lie algebra
$\mathfrak{sl}_2(\R)$. In the same vein, we define the universal cover of the connected Lie group
$GL^+_2(\R)=\{A\in\operatorname{End}(\R^2):\det A>0\}\cong SL_2(\R)\times\R_+$: it is the collection
$$
\widetilde{GL_2^+(\R)}=\left\{[\gamma]\;|\;\gamma:[0,1]\rightarrow GL_2^+(\R)\;\text{s.t.}\;\gamma(0)=E\right\}
$$
of continuous paths in $GL_2^+(\R)$ with starting point the unity, up to homotopy with both ends fixed. The group structure is given by the pointwise multiplication of paths 
$
[\gamma_1]\cdot [\gamma_2]=[\gamma_1\cdot\gamma_2]
$,
with $(\gamma_1\cdot\gamma_2)(t)=\gamma_1(t)\cdot\gamma_2(t)$ for all $t\in[0,1]$.

To extend this construction to $GL_2(\R)=GL^+_2(\R)\sqcup GL^-_2(\R)$, we define $E_1=E=\operatorname{diag}(1,1)$,
$E_{-1}=\operatorname{diag}(1,-1)$, and note that we have a group homomorphism 
$$\pi_0(GL_2(\R))\cong\bZ_2=\{\pm 1\}\to GL_2(\R)$$
given by $\epsilon\mapsto E_\epsilon$. We can then perform the universal cover 
$\widetilde{GL_2(\R)}=\widetilde{GL^+_2(\R)}\sqcup\widetilde{GL^-_2(\R)}$
for the two connected components and lift the multiplication operations. Explicitly

$$\widetilde{GL_2^-(\R)}=\left\{[\gamma]\;:\;\gamma:[0,1]\rightarrow GL_2^-(\R)\;\text{s.t.}\;\gamma(0)=E_{-1}\right\}$$ 
and the group structure for $\widetilde{GL_2(\R)}$ is again given by the pointwise multiplication of paths. 
(We note that it is well-defined since $E_{-1}$ squares to $E$). 
\begin{remark}\hfill
\begin{enumerate}
	\item 
	Obstructions controlling 
the existence of a group structure on the universal cover for disconnected groups $G$ rely on the fact that one has to specify multiplication of paths 
living on different connected components. The obstructions are due to the work \cite{Tay} by R.\,L.\ Taylor and belong to $H^k(\pi_0(G),\pi_1(G^o))$ for $k=3$. See also \cite{RVW}. In our case $G=GL_2(\R)$, we have that $\pi_1(G^o)\cong\bZ$ is the nontrivial
representation of $\pi_0(G)\cong\bZ_2$, and one can compute that the above cohomology groups vanish for $k$ even and are isomorphic to $\bZ_2$ for $k$ odd. In particular the group $H^3(\pi_0(G),\pi_1(G^o))$ is non-trivial, yet the obstruction class is trivial, since we made explicit its group structure.
On a more general level, whenever the natural exact sequence 
${\bf 1}\rightarrow G^o\rightarrow G\rightarrow\pi_0(G)\rightarrow{\bf 1}$ admits a splitting of groups, such an obstruction vanishes.
\item
According to \cite[(6.5)]{Tay} 
the cohomology group $H^2(\pi_0(G),\pi_1(G^o))$ acts simply transitively on the space of group coverings, therefore the covering is unique for $G=GL_2(\R)$. 
\end{enumerate}

\end{remark}
We have an epimorphism $p:\widetilde{GL_2(\R)}\to GL_2(\R)$, $[\gamma]\mapsto \gamma(1)$.
This is a covering map with deck transformation group 
\begin{equation}
\label{eq:deck}
p^{-1}(E)=\pi_1(GL_2^+(\R))\cong\bZ\;,
\end{equation}
a discrete normal subgroup of $\widetilde{GL_2(\R)}$. 
%Since the manifold  contracts to $SO(2)$, 
The generators of \eqref{eq:deck} are given by the fundamental group of $SO(2)=U(1)$, i.e.,
$p^{-1}(E)=\{[\gamma_k]\}$ with paths 
$\gamma_k(t)=e^{2\pi k i t}$, $t\in[0,1]$, for all $k\in\bZ$.
(Every time we write a complex $1\times1$ matrix we mean the corresponding real $2\times2$ matrix.) As explained later, the deck transformation group $p^{-1}(E)$ is not central in the universal cover group. 

In summary, the space $\widetilde{GL_2(\R)}$ has two connected components, both simply connected: 
\begin{equation}
\label{eq:pi01}
\pi_0(\widetilde{GL_2(\R)})\cong\bZ_2,\quad\pi_1(\widetilde{GL^+_2(\R)},E_1)=0,\quad\pi_1(\widetilde{GL^-_2(\R)},E_{-1})=0\;,
\end{equation} where we identified
$E_{\pm1}$ with the corresponding constant paths. We have the following exact sequence of group homomorphisms
 \begin{equation}\label{exseq1}
1\longrightarrow\bZ\longrightarrow\widetilde{GL_2(\R)}\longrightarrow GL_2(\R)\longrightarrow 1
 \end{equation}
and similarly for $GL^+_2(\R)$. The sequence does not split, 
by the above \eqref{eq:pi01} on $\pi_0(\widetilde{GL_2(\R)})$.
%on the connected components of $\widetilde{GL_2(\R)}$.
%(to which the homomorphism from the last term is given by $\det$).

Sequence \eqref{exseq1} has the following retract
 \begin{equation}\label{exseq3}
1\longrightarrow\bZ\longrightarrow\widetilde{O(2)}\longrightarrow O(2)\longrightarrow1.
 \end{equation}
Here $O(2)=S^1_+\sqcup S^1_-$
is the non-Abelian $1$-dimensional group with two connected components and group
operation $(e^{i\varphi_1},\epsilon_1)\cdot(e^{i\varphi_2},\epsilon_2)=
(e^{i(\varphi_1+\epsilon_1\varphi_2)},\epsilon_1\epsilon_2)$. 
(The sign in $S^1_\epsilon$ is, of course, nothing but the determinant.) 
Its universal cover is $\widetilde{O(2)}=\R^1_+\sqcup\R^1_-$ with the group 
operation $(\varphi_1,\epsilon_1)\cdot(\varphi_2,\epsilon_2)=
(\varphi_1+\epsilon_1\varphi_2,\epsilon_1\epsilon_2)$ and
we again have $\pi_0(\widetilde{O(2)})=\pi_0(O(2))\cong\bZ_2$.

The center of $GL_2(\R)$ is $\mathcal{Z}=\{\operatorname{diag}(a,a):a\in\bR_\times\}
\subset GL^+_2(\R)$, and its preimage via the covering map is the subgroup
 $
\widetilde{\mathcal{Z}}=
\{[\gamma_k^a]:a\neq0,k\in\bZ\}
 $ of $\widetilde{GL_2^+(\R)}$, where the paths are
\begin{equation*}
\gamma_k^a(t)=\left\{\begin{array}{ll}
e^{2\pi k i t+t\ln(a)} & \text{ for }a>0,\\
e^{2\pi (k+\tfrac12) i t+t\ln(-a)} & \text{ for }a<0.
\end{array}\right. 
 \end{equation*}
Consequently, $\widetilde{\mathcal{Z}}$ is central in $\widetilde{GL^+_2(\R)}$ but due to 
 \begin{align*}
E_{-1}\gamma_k^aE_{-1}&=\gamma^a_{-k} \quad\text{ for }a>0\;,\\
E_{-1}\gamma_k^aE_{-1}&=\gamma^a_{-k-1} \,\text{ for }a<0\;,
 \end{align*}
it is only normal in $\widetilde{GL_2(\R)}$. In particular, 
the action of $\pi_0(GL_2(\R))\cong \mathbb Z_2$ on $\pi_1(GL_2^+(\R))\cong\mathbb Z$ is non-trivial and the deck transformation group is not central.
(In fact, it is actually easy to see that the center of the universal cover is formed by the paths $[\gamma_k^a]$ with $a>0$ and $k=0$.)

We conclude this section with the following crucial observation on $1$-dimensional subgroups.
The subgroup $H=\{\operatorname{diag}(a^2,1/a):a\in\R_\times\}=\{\operatorname{diag}(e^{2\lambda},\epsilon e^{-\lambda}):\lambda\in\mathbb R, \epsilon=\pm 1\}\cong\mathbb R_+\times\mathbb Z_2$ determines the following subsequence of \eqref{exseq1}:
 \begin{equation}\label{exseq2}
1\longrightarrow\bZ\longrightarrow\widetilde{H}\longrightarrow H\longrightarrow1\;,
 \end{equation} 
where 
$\widetilde{H}=p^{-1}(H)=\{
[\gamma_{\lambda,\epsilon,k}(t):=\gamma_k(t)\cdot\operatorname{diag}(e^{2\lambda t},\epsilon e^{-\lambda t})]:\lambda\in\R,\epsilon=\pm1,k\in\bZ\}\subset \widetilde{GL_2(\R)}$. 
%$\Lambda_\lambda^\epsilon=\operatorname{diag}(e^{2\lambda},\epsilon e^{-\lambda})$
This subsequence splits, using the above formula for the paths $\gamma_{\lambda,\epsilon,k}(t)$ with $k=0$,
thus 
$$\widetilde{H}\cong H\times\bZ\cong \bR_+\times\mathbb Z_2\times \bZ$$
as a Lie group.
Since $\pi_0(H)\cong\bZ_2$, we get $\pi_0(\widetilde{H})\cong\bZ_2\times\bZ$, with the natural component group homomorphism $\widetilde H\to \pi_o(\widetilde H)$ simply given by
$[\gamma_{\lambda,\epsilon,k}(t)]\mapsto(\epsilon,k)$.

 \subsection{Proof of the main results: global models}

Let us now integrate the Lie algebra of infinitesimal CR automorphisms $\gg=\ggl_2(\R)\ltimes\R^4$, $\mathbb R^4\cong S^3\mathbb R^2$, to the connected Lie group
$G^o=GL_2^+(\R)\ltimes\R^4$, 
and the stabilizer subalgebra
$\mathfrak h=\stab=\langle\operatorname{diag}(-2/3,1/3)\rangle\subset\frak{gl}_2(\mathbb R)$ 
to the closed connected subgroup
$H^o=\{\operatorname{diag}(a^2,1/a):a\in\mathbb R_+\}$. Let us first note that the center of $\gg$ is trivial and that the group of inner automorphisms of $\gg$ is precisely $G^o=GL_2^+(\R)\ltimes\R^4$. In particular any other connected Lie group with Lie algebra $\gg$ covers $G^o$ and it is a quotient of the universal cover $\widetilde {G^o}$ of $G^o$ by a discrete central subgroup. 
The Zariski closure yields
$G=GL_2(\R)\ltimes\R^4$ and  $H=\{\operatorname{diag}(a^2,1/a):a\in\mathbb R_\times\}$ -- the quotient is the same manifold
$\cM^7=G/H=G^o/H^o$, with the action of the full $G$ that is still effective (since $GL_2(\R)$ acts effectively on $\bR^4\cong S^3\R^2$).
Passing to $G/H^o$ yields a disconnected manifold, which we do not allow in our analytic setup.

Since the model is locally unique, all other global models are obtained by coverings and by quotients.
The quotient of $G$ by a discrete normal subgroup is in fact not possible because such a subgroup would project
along the nilradical to a discrete normal subgroup in $GL_2(\R)$, which is central, and so is 
$\bZ_2=\{\pm E\}$. It is then easy to see that the unique discrete normal subgroup of $G$ is the unity subgroup. On the other hand, $H$ is a maximal subgroup in $G$ with Lie algebra $\mathfrak h\subset \gg$. In summary $\cM^7\cong G/H$ cannot be quotiented.

Now we shall pass to the automorphism group $G=GL_2(\R)\ltimes\R^4$. It also has two connected components,
$\pi_0(G)=\bZ_2$, so $G=G^+\sqcup G^-$ according to the determinant of the reductive part.
Each component is not simply connected, namely, we have the same retracts as in \S\ref{sec:3.3-prerequisite}:
$\pi_1(G^+,E_1)\cong\bZ$ and $\pi_1(G^-,E_{-1})\cong\bZ$.
The passage to the universal cover is similar to \eqref {exseq1}: we get the short exact sequence
 \begin{equation}\label{exseq4}
1\longrightarrow\bZ\longrightarrow\widetilde{G}\longrightarrow G\longrightarrow1\;,
 \end{equation}
where $\widetilde G=\widetilde{GL_2(\R)}\ltimes\bR^4$. The group structure of
$\widetilde{GL_2(\R)}$ is as in \S\ref{sec:3.3-prerequisite} and $\widetilde{GL_2(\R)}$ acts on $\mathbb R^4\cong S^3\mathbb R^2$ by factorization through $GL_{2}(\mathbb R)$. We note that
$$\widetilde G=\widetilde{G^+}\sqcup\widetilde{G^-}\;,$$ with simply-connected components $\widetilde{G^\pm}=\widetilde{GL_2^\pm(\R)}\ltimes\R^4$. 
%i.e., $\pi_1(\widetilde{G^+},E_1)=\pi_1(\widetilde{G^-},E_{-1})=0$. 

The retract \eqref {exseq3} and the subsequence \eqref {exseq2} hold for the sequence \eqref{exseq4} as well. In this case the center $\mathcal{Z}$ of $G$ is trivial, yet 
its preimage  $
\widetilde{\mathcal{Z}}=\{[\gamma_k]:k\in\bZ\}
 $ via the covering map 
is central only in $\widetilde{G^+}$, and it is normal in 
$\widetilde{G}$ (by the same reasons as in \S\ref{sec:3.3-prerequisite}). The group $\widetilde{\mathcal{Z}}$ is the maximal discrete normal subgroup in $\widetilde{G}$.
 
Using that the sequence \eqref{exseq2} splits, so that
$\widetilde H\cong H\times \widetilde{\mathcal Z}\cong H\times\bZ$ as groups, we may obtain all the homogeneous models as quotients
$$\widetilde{\cM}_m=\big(\widetilde G/m\bZ\big)/H\cong \widetilde{G}/\big(H\times m\bZ\big)$$ 
by a subgroup $m\bZ$ of $\widetilde{\mathcal{Z}}\cong\bZ$, for $m\ge0$. Clearly
$\widetilde{\cM}_m$ is an $\mathbb Z_m$-covering of $\cM=G/H\cong\widetilde G/\widetilde H$.
(If $m=1$ we get back our initial manifold $\cM^7$, if $m=0$ we get the universal cover of $\cM^7$.)
Each of the models $\widetilde{\cM}_m$ has the structure of a $3$-nondegenerate 
CR manifold of hypersurface type and its automorphism group is $\widetilde{G}/m\bZ$.

\section{The maximal symmetric model: CR realization and beyond}\label{sec:5}

Here we will provide a realization $\cM^7\subset \bC^4$ as a real hypersurface in $\bC^4$ and give a local coordinate expression for the homogeneous model $G/H$ of \S\ref{sec:3.1}, where 
$G=GL_2(\mathbb R)\ltimes S^3\R^2$
and $H=\{\operatorname{diag}(a,1/a^2):a\neq 0\}$. (For simplicity of exposition in this section, the stabilizer  is conjugated to that of \S\ref{sec:3.1}, namely $\mathfrak{h}=\operatorname{Lie}(H)$ is generated by $\operatorname{diag}(1/3,-2/3)$ here.)

\subsection{Tube realization}\label{sec:5.1}

To realize it we note that the Abelian part $\R^4\cong S^3\R^2$ acts by
translations on itself, which we denote $\R^4_y$, and the
reductive part $GL_2(\R)$ acts on $\R^4_x\cong S^3\mathbb R^2$ with the minimal orbit
in the projectivization given by the degree $3$ rational normal curve
$\{[1:\lambda:\lambda^2:\lambda^3]\mid\lambda\in\mathbb R\mathbb P^1\}\subset\R\mathbb{P}_x^3$.
Let $R=\{(r^3,r^2s,rs^2,s^3)\mid r,s\in\mathbb R\}\subset\R^4_x$ be the cone over it, which we refer to as the {\it rational normal cone} in $\mathbb R^4_x$. It is given by the relations
 $
d_1:=x_0x_2-x_1^2=0$, $d_2:=x_0x_3-x_1x_2=0$, $d_3:=x_1x_3-x_2^2=0
 $, where $(x_0,x_1,x_2,x_3)$ are the coordinates of $\mathbb R^4_x$.
Note that the syzgies between the generators of the ideal $\langle d_1,d_2,d_3\rangle$ defining the rational normal cone are
$x_3d_1-x_2d_2+x_1d_3=0$, $x_2d_1-x_1d_2+x_0d_3=0$.

We let $TR\subset\mathbb R^4_x$ be the tangent variety to the rational normal cone, which is locally parametrized 
as 
$x_0=r^3$, $x_1=r^2(s+t)$, $x_2=rs(s+2t)$, $x_3=s^2(s+3t)$, where $r,s,t\in\mathbb R$.
Note that not only $t=0$ but also $r=0$ is in the singular locus, so we may either require an additional chart like
$x_0=r^2(r-3t)$, $x_1=rs(r-2t)$, $x_2=s^2(r-t)$, $x_3=s^3$, or have a surjective map as in \eqref{7D} in the Introduction. With any approach, eliminating the parameters, one gets the following global defining equations
$$
TR=\big\{x\in\mathbb R^4_x\mid x_0^2x_3^2-6x_0x_1x_2x_3+4x_0x_2^3+4x_1^3x_3-3x_1^2x_2^2=0\big\}\;,
$$
which are in agreement with and give a geometric interpretation to the equations in \cite[\S 5.1]{FK2}.

We note that a matrix $\begin{pmatrix} a & b \\ c &d \end{pmatrix}$ in $GL_2(\R)$ acts on $\R^4_x$ via the matrix representation
 $$
\begin{pmatrix} 
a^3 & 3a^2b & 3ab^2 & b^3 \\ a^2c & a^2d+2abc & 2abd+b^2c & b^2d \\
ac^2 & bc^2+2acd & 2bcd+ad^2 & bd^2 \\ c^3 & 3c^2d & 3cd^2 & d^3
\end{pmatrix}
 $$
and that the $GL_2(\R)$-orbit through the point $p_1=(1,0,0,0)$ is $R\setminus 0$. On the other hand, the orbit through the point $p_2=(0,1,0,0)\in TR\setminus R$ 
%as well as the point $p'_2=(1,1,0,0)\in\Sigma\setminus R$, 
is the full complement $TR\setminus R$.
The stabilizers of those points in $GL_2(\R)$ are the $2$-dimensional solvable group
 \begin{equation}\label{HH1}
\operatorname{Stab}_{p_1}=\left\{\begin{pmatrix} 1 & b \\ 0 & d \end{pmatrix}: d\neq 0\right\}\cong\operatorname{Sol}(2)\;,
 \end{equation}
and the subgroups
 \begin{equation}\label{HH2}
H=\operatorname{Stab}_{p_2}=\left\{\begin{pmatrix} a &  0 \\ 0 & a^{-2} \end{pmatrix}: a\neq 0\right\}\;.
%\;\;\text{and}\;\;
 % H'=
%\operatorname{Stab}_{p'_2}=\left\{\begin{pmatrix} a & & \dfrac{1-a^3}{3a^2} \\ & & \\ 0 & & \dfrac{1}{a^2} \end{pmatrix}\right\}\;,
 \end{equation}
%which are, of course, conjugated, since $p_2$ and $p_2'$ are two different points in the same orbit. 
%NB: One can easily show $p'_2\in\Sigma\setminus R$ but may miss the point $p_2$ in some parametrizations
%One may have difficulty in showing that $p_2$ is in $\Sigma$, as the vector $\vec{p_1p_2}$
%is not a derivative of the above parametrization by neither $r$ nor $s$ at any point.
%It is however a combination of those derivatives, which is a correct way to do in affine (non-projective) version.
We note that the stabilizer $\operatorname{Stab}_{p_2}$ is the subgroup $H$ of $GL_2(\R)$ from \S\ref{sec:3.3}. We set $\Sigma:=TR\setminus R$.

 \begin{proposition}
\label{prop:diff-type}
The orbit $\Sigma$ is diffeomorphic to $S^1\times\R^2$.
 \end{proposition}
 
 \begin{proof}
Consider the map $\Psi:(\R^1\!\!\!\mod2\pi\bZ)\times\R\times\R\longrightarrow\R^4_x$ given by
 \begin{eqnarray*}
(\phi,b,t)\mapsto \!\!&(b\cos^3\phi-3t\cos^2\phi\sin\phi,
b\cos^2\phi\sin\phi+t(\cos^3\phi-2\cos\phi\sin^2\phi),\\
&\,b\cos\phi\sin^2\phi-t(\sin^3\phi-2\cos^2\phi\sin\phi),b\sin^3\phi+3t\cos\phi\sin^2\phi).
 \end{eqnarray*}
We claim that the restriction of $\Psi$ to $(\R^1\!\!\!\mod2\pi\bZ)\times\R\times\R_+$ is injective, with the image $\Sigma$. Indeed, 
the above map has the form $\Psi(\phi,b,t)=b\gamma(\phi)+t\gamma'(\phi)$ and
one easily checks that the $4\times2$ matrix $[\gamma(\phi),\gamma'(\phi)]$ has rank 2 for any fixed
parameter $\phi$, so any half-plane parametrized by $(b,t)\in \R\times\R_{+}$ is embedded. We also note that the boundary $(\R^1\!\!\!\mod2\pi\bZ)\times\R\times\{t=0\}$ corresponds to the rational normal cone $R$.

Consequently the image $\Psi\big((\R^1\!\!\!\mod2\pi\bZ)\times\R\times\R_+\big)$ is fibered by half-planes parametrized by $(b,t)\in \R\times\R_+$ 
over the rational normal curve $\Psi\big((\R^1\!\!\!\mod2\pi\bZ)\times\{b=1\}\times\{t=0\}\big)\subset S^3$;  note that the parametrization of the rational normal curve is injective. This proves injectivity of our map.

The change $\phi\mapsto\phi+\pi$ results in the reflection $(b,t)\mapsto (-b,-t)$ and therefore interchanges the half-planes $\{b\in\R, t>0\}$
and $\{b\in\R, t<0\}$. 
This implies the claim about the image.
 \end{proof}
 
\begin{remark}
The projective version $\mathbb{P}\Sigma$ gives the M\"obius band, as it is a line bundle 
over $\mathbb{P}R$ with connected complement to the (central) rational normal curve $\R\mathbb{P}^1$.
\end{remark}

We define the tube $TR\times\R^4\subset\mathbb C^4_z=\R^4_x\times\R^4_y$,
where the coordinates $z_k=x_k+iy_k$ specify the standard complex
structure $J$ of $\mathbb C^4_z$, namely, $J\partial_{x_k}=\partial_{y_k}$ for all $0\leq k\leq3$. Now the group $GL_2(\R)$ acts diagonally on $\mathbb C^4_z$
and it preserves $J$, since it is in fact a subgroup of $GL_2(\mathbb C)$. Thus $G=GL_2(\mathbb R)\ltimes S^3\R^2$ is a group of complex affine transformations of the CR manifold $(TR\times\R^4, \cD, \cJ)$, with $\cD$ the maximal $J$-complex subbundle of the tangent bundle of $TR\times\R^4$
and
%$T((TR)\times\R^4)\cap J\big(T((TR)\times\R^4)\big)$ and 
$\cJ$ the restriction of $J$ to $\cD$.

Note that the $7$-dimensional manifold $TR\times\mathbb R^4$ is not homogeneous for the action of $G$, as there are three orbits: two orbits in $R\times\R^4_y$ (determined by the punctured rational normal cone and, respectively, its vertex),
and the open orbit $\Sigma\times\R^4_y$ that is complementary to $R\times\R^4_y$.
The $3$-nondegenerate $7$-dimensional CR homogeneous model is the latter orbit, which we denoted $\mathcal R^7$ in the Introduction. 
\begin{theorem}
\label{thm:29}
The above $7$-dimensional CR manifold $\cM^7=\Sigma\times\R^4_y$ is 3-nondegenerate and 
it is diffeomorphic to $S^1\times\mathbb R^6$. Its automorphism group is $G$, which acts transitively, and $\cM^7\cong G/H$, where the stabilizer subgroup is as in \eqref{HH2}.
 \end{theorem}
We present here a straightforward coordinate computation. Another argument, adaptable to higher dimensions,
will be given in \S\ref{sec:5.4}.
\begin{proof}
The claim on the diffeomorphism type of $\cM^7$ follows immediately by Proposition \ref{prop:diff-type}. 
 The CR-distribution $\cD$ is generated by 
the vector fields
 \begin{align*}
X_1&=x_0\p_{x_0}+x_1\p_{x_1}+x_2\p_{x_2}+x_3\p_{x_3},\\
X_2&=4d_1^2\p_{x_1}+4d_1d_2\p_{x_2}+3d_2^2\p_{x_3},\\
X_3&=2d_1\p_{x_1}+d_2\p_{x_2},\;
\end{align*}
and $Y_1=JX_1$, $Y_2=JX_2$, $Y_3=JX_3$.
% \begin{align*}
% Y_1&=JX_1=x_0\p_{y_0}+x_1\p_{y_1}+x_2\p_{y_2}+x_3\p_{y_3},\\
% Y_2&=JX_2=4d_1^2\p_{y_1}+4d_1d_2\p_{y_2}+3d_2^2\p_{y_3},\\
% Y_3&=JX_3=2d_1\p_{y_1}+d_2\p_{y_2}.
% \end{align*}
Using the local parametrization $x_0=r^3$, $x_1=r^2(s+t)$, $x_2=rs(s+2t)$, $x_3=s^2(s+3t)$ of $TR$,  we get more symmetric formulae
$$-\frac{1}{2r^3t^2}X_3=r\partial_{x_1}+s\partial_{x_2}$$ and 
 $$
\langle X_1,X_2\rangle = \langle 3r^2\partial_{x_0}+2rs\partial_{x_1}+s^2\partial_{x_2},
r^2\partial_{x_1}+2rs\partial_{x_1}+3s^2\partial_{x_3}\rangle.
 $$
Hence $\cD_{10}$ is generated by
 \begin{align*}
Z_1&=x_0\p_{z_0}+x_1\p_{z_1}+x_2\p_{z_2}+x_3\p_{z_3},\\
Z_2&=4d_1^2\p_{z_1}+4d_1d_2\p_{z_2}+3d_2^2\p_{z_3},\\
Z_3&=2d_1\p_{z_1}+d_2\p_{z_2},
 \end{align*}
where $\p_{z_k}=\frac12\big(\p_{x_k}-i\p_{y_k}\big)$ as usual,
and $\cD_{01}$ is generated by $\bar{Z}_1,\bar{Z}_2,\bar{Z}_3$.

The following commutation relations hold modulo $\cD_{01}$ on $\cM^7=\Sigma\times\mathbb R^4_y$:
\begin{align*}
[\bar{Z}_1,Z_1]& =\frac12 Z_1,\  [\bar{Z}_1,Z_2] =2 Z_2,\ [\bar{Z}_1,Z_3] =Z_3,\\
[\bar{Z}_2,Z_1]&= \frac12 Z_2,\  [\bar{Z}_2,Z_2] =\frac{8d_1(x_0d_3-x_2d_1)}{d_2}Z_2,\ 
[\bar{Z}_2,Z_3]= 2(x_0d_2-2x_1d_1)Z_3,\\
[\bar{Z}_3,Z_1]& =\frac12 Z_3,\  
[\bar{Z}_3,Z_2] = -\frac{2x_2d_1+x_1d_2}{d_2}Z_2+\frac{2d_1(x_0d_3-x_2d_1)}{d_2}Z_3,\\
[\bar{Z}_3,Z_3]&= -\frac{2x_2d_1+x_1d_2}{2d_2}Z_3+\frac{d_1(x_0d_3-x_2d_1)}{d_2}\p_{z_1}\not\in D.
 \end{align*}
Therefore we have 
$\cF^0_{10}=\langle Z_1,Z_2\rangle$
and the next term of the Freeman sequence is 
$\cF^1_{10}=\langle Z\rangle$,
for the vector field $4x_0Z=4d_1(x_0d_3-x_2d_1)Z_1-d_2Z_2$. We can simplify this vector field as
 $$
Z=d_1(x_0d_3-x_2d_1)\p_{z_0}-d_1(x_1d_3-x_3d_1)\p_{z_1}
-d_3(x_0d_3-x_2d_1)\p_{z_2}-d_3(x_1d_3-x_3d_1)\p_{z_3}\;,
 $$
and in the local parametrization, we get the simpler expression
 $$
-\frac1{2r^6t^5s}Z=r^3\partial_{z_0}+r^2s\partial_{z_1}+rs^2\partial_{z_2}+s^3\partial_{z_3}\;.
 $$
Finally $\cF^2_{10}=0$, and this finishes the proof of 3-nondegeneracy.

By our classification of locally homogeneous $3$-nondegenerate $7$-dimensional CR manifolds, the dimension of $G$
is the upper bound for the dimension of the automorphism group of such a structure.  
The affine automorphism group is exactly $G$, since $GL_2(\mathbb R)$ is known to be the affine automorphism group of the rational normal cone and its tangent variety. 
The fact that the entire automorphism group is $G$ follows since any discrete extension of $G$ acts linearly on the radical $\R^4_y$, therefore it normalizes the action of $GL_2(\mathbb R)$, it preserves the minimal orbit (the rational normal curve), and it factors through the action of $GL_2(\mathbb R)$. In \S \ref{sec:3.3}, we proved that any non-trivial cover
of $G$ acts non-effectively on $\cM^7\cong G/H$, whence the claim. 
\end{proof}

\subsection{Tube over rational normal cone}\label{sec:5.2}

The $G$-orbit $\cN^6=(R\setminus0)\times\R^4_y$ can be interpreted in its own right
as a CR manifold of CR-dimension $2$ and CR-codimension $2$, and it also
satisfies the H\"ormander condition, i.e., the corresponding CR-distribution $\cD$ is bracket generating.
We remark that the Freeman filtration and the notion of $k$-nondegeneracy equally apply to higher CR-codimensions.  

 \begin{theorem}
The above $6$-dimensional CR manifold $\cN^6=(R\setminus0)\times\R^4_y$ is 2-nondegenerate 
and it is diffeomorphic to $S^1\times\mathbb R^5$. Its automorphism group is $G$, which acts transitively, and 
$\cN^6\cong G/\operatorname{Sol}(2)$, where the stabilizer subgroup is as in \eqref{HH1}.
 \end{theorem}

 \begin{proof}
The claim on the diffeomorphism type of $\cM^7$ follows by the proof of Proposition \ref{prop:diff-type}. 
In the standard parametrization of $R\setminus 0$ given by $x_0=r^3,x_1=r^2s,x_2=rs^2,x_3=s^3$, $(r,s)\neq(0,0)$,
the CR-distribution $\cD$ is generated by the vector fields
 \begin{align*}
X_1&= 3r^2\p_{x_0}+2rs\p_{x_1}+s^2\p_{x_2},\\
X_2&= r^2\p_{x_1}+2rs\p_{x_2}+3s^2\p_{x_3},
\end{align*}
together with $Y_1=JX_1$, $Y_2=JX_2$.

The H\"ormander condition then follows from
 $$
Y_3=[X_1,Y_2]=[X_2,Y_1]=r\partial_{y_1}+s\partial_{y_2},\
[X_1,Y_3]=\partial_{y_1},\ [X_2,Y_3]=\partial_{y_2}\;,
 $$
while the Cauchy characteristic space is easily seen to be generated by
 \begin{align*}
X_0&= \tfrac13(rX_1+sX_2)= r^3\p_{x_0}+r^2s\p_{x_1}+rs^2\p_{x_2}+s^3\p_{x_3},\\
Y_0&= \tfrac13(rY_1+sY_2)= r^3\p_{y_0}+r^2s\p_{y_1}+rs^2\p_{y_2}+s^3\p_{y_3}.
\end{align*}

The distribution $\cD_{10}$ is generated by the vector fields $Z_1= \frac12(X_1-iY_1)$, $Z_2= \frac12(X_2-iY_2)$.
We then have $\cF^0_{10}=\langle Z_0\rangle$, where $Z_0= \tfrac13(rZ_1+sZ_2)$, while 
the next term $\cF^1_{10}$ of the Freeman sequence is trivial. This proves $2$-nondegeneracy.

The symmetry algebra is obtained by a somewhat tedious but straightforward computation, which we omit
and make available in Maple supplement to the arXiv version of this paper.
The claim on the automorphism group follows by applying Lie theoretic arguments as in the end of the proof of Theorem \ref{thm:29}.
\end{proof}

\subsection{Relation to other geometries}\label{sec:5.3}

Here we describe the maximal symmetric CR model $\cM^7$ and some related geometries in
the spirit of Klein's Erlangen program. The automorphism group for all of them is $G=GL_2(\R)\ltimes\R^4$,
the models are all homogeneous and the stabilizer subgroup is the subgroup of $GL_2(\R)$ indicated on the edges of the diagram below
(for edges not emanating from $G$ the meaning of the label is a fiber,
a subgroup that has to be added to the one above it to generate the desired stabilizer).
All maps in the diagram are $G$-equivariant, except for the dashed horizontal arrow that represents a natural fibration, which is not a group quotient (note that $\operatorname{Sol}(2)$ does not include $H$ as a subgroup).

This contributes to the Segre correspondence between CR manifolds and related finite type differential equations, well developed for the Levi-nondegenerate case, in our degenerate situation.

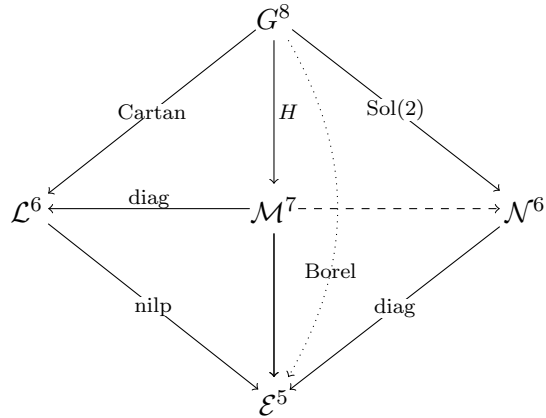
\[
\begin{tikzpicture}[descr/.style={fill=white}]
\matrix(m)[matrix of math nodes, row sep=5em, column sep=7em,
text height=3.0ex, text depth=0.9ex]
 % {& $G$ \\ $L$ & $M$ & $N$ \\ $\mathcal{E}$\\};
{& G^8 \\ \cL^6\, & \cM^7\mathstrut & \,\cN^6 \\ & \mathcal{E}^5 \\};
\path[->,font=\scriptsize]
(m-1-2) edge node[descr] {Cartan} (m-2-1)
(m-1-2) edge node[right] {\,$H$} (m-2-2)
(m-1-2) edge node[descr] {$\operatorname{Sol}(2)$} (m-2-3)
(m-2-2) edge node[above] {diag} (m-2-1)
(m-2-1) edge node[descr] {nilp} (m-3-2)
(m-2-3) edge node[descr] {diag} (m-3-2);
\path[dashed,->]
(m-2-2) edge (m-2-3);
\path[dotted,->]
(m-1-2) edge [bend left=30] (m-3-2);
\path[->]
(m-2-2) edge (m-3-2);
\path[->,font=\scriptsize]
(m-2-2) edge node[above right] {{\vbox{\hbox{ \,\,\,\, Borel}\hbox{}}}} (m-3-2);
\end{tikzpicture}
\]

We have already discussed the $\cM^7$ and $\cN^6$ nodes in \S\ref{sec:5.1} and \S\ref{sec:5.2}. The bottom node $\mathcal{E}^5$ is the 
fourth order trivial ODE $y^{\rm{iv}}(x)=0$ considered as a submanifold in jets
 $$
\mathcal{E}^5=\{y_4=0\}\subset J^4(\R,\R)\cong\R^6(x,y_0,y_1,y_2,y_3,y_4)\;. 
 $$
Its symmetry algebra is known to be $\gg=\operatorname{Lie}(G)$, but in order to
make the automorphism group precisely $G$ one has to assume the independent variable $x\in S^1=\R\,\operatorname{mod}\pi\bZ$ to be periodic,
so that actually $\mathcal{E}^5\cong S^1\times\R^4$ (instead of $\mathcal{E}^5\cong\R^5$).
In this case the stabilizer group is the Borel subgroup $B$ of $GL_2(\R)$ and the periodicity is due to
the fact that $GL_2(\R)/B=\R\mathbb{P}^1\cong S^1$.

This quotient $\mathcal{E}^5=G/B$ can be conveniently represented by the root diagram of $\gg$ below,
compare to that of \S\ref{sec:3.1}. The stabilizer subalgebra corresponds to $\gg_0\oplus\gg_1$ -- 
for the moment, ignore the circle around the first component as well as the integral sign above the second. 
The grading corresponds to the Tanaka prolongation of the negatively graded part $\gg_{-4}\oplus\cdots\oplus\gg_{-1}$, which is the symbol
of the Cartan distribution of the equation $\mathcal{E}$.
We also indicate generators for the weak derived flag of vector distributions corresponding to the graded components of negative degree, where 
$D_x=\partial_x+y_1\partial_{y_0}+y_2\partial_{y_1}+y_3\partial_{y_2}$ is the truncated total derivative.

 
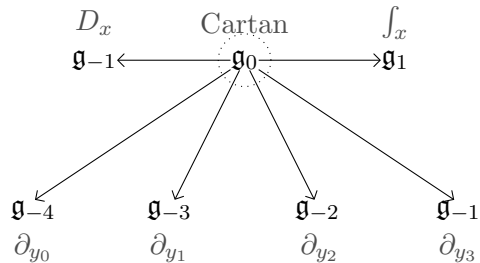
\begin{center}
\begin{tikzpicture}
\node (A) at (-2,2) {$\gg_{-1}$};
\node (B) at (0,2) {$\gg_0$};
\node (C) at (2,2) {$\gg_1$}; 
\node (a) at (-3,0) { \quad $\gg_{-4}$};
\node (b) at (-1,0) {$\gg_{-3}$};
\node (c) at (1,0) {$\gg_{-2}\,$};
\node (d) at (3,0) {\!\!\!$\gg_{-1}$ \quad };
\draw[dotted] (0,2) circle (0.35);
\path[->,font=\scriptsize,>=angle 90]
(B) edge (A)
(B) edge (C)
(B) edge (a)
(B) edge (b)
(B) edge (c)
(B) edge (d);
\node[black!70] (A1) at (-2,2.5) {$D_x$};
\node[black!70] (B1) at (0,2.5) {Cartan};
\node[black!70] (C1) at (2,2.5) {$\int_x$};
\node[black!70] (a1) at (-3,-0.5) { \quad $\partial_{y_0}$};
\node[black!70] (b1) at (-1,-0.5) {$\partial_{y_1}$};
\node[black!70] (c1) at (1,-0.5) {$\partial_{y_2}$};
\node[black!70] (d1) at (3,-0.5) {\!\!\!$\partial_{y_3}$ \quad };
\end{tikzpicture}
 \end{center}
 
The final node of the diagram is the quotient $\cL^6=G/G_0$, with $G_0$ the completely non-compact Cartan subgroup of $GL_2(\R)$ given by the invertible diagonal matrices. 
This node is intermediate between $\cM^7$ and $\mathcal{E}^5$, and it can be described as follows. 

First of all note that $GL_2(\R)/G_0=SL_2(\R)/\big(G_0\cap SL_2(\R)\big)$ is an adjoint orbit, in particular it has a naturally associated symplectic form up to homothety (we will not make use of it, however). There are three types of non-zero orbits on $\mathfrak{sl}_2(\R)$, ours is diffeomorphic to $T^*S^1$ and it is the orbit that admits 
a Lorentzian metric of constant curvature w.r.t. the Killing form. Thus $\cL^6\cong T^*S^1\times\R^4\cong S^1\times\R^5$.

Next, the quotient $\cL^6=G/G_0$ can be again represented by the root diagram of $\gg$, as
$\cL^6$ is a line bundle over $\mathcal{E}^5$. In fact, the stabilizer reduces from the Borel subgroup $B$ to the Cartan subgroup $G_0$ and we indicate the
changes on the above diagram: the stabilizer subalgebra $\gg_0$ is circled and
the fiber of the line bundle is $\gg_1$. The latter is generated by the symbol $\int_x$ above the top right node, which is an ``algebraic'' integration (a differential operator inverse to $D_x$ in the sense that it acts in the 
opposite direction for the corresponding vector distributions).
We stress that the Cartan subalgebra $\gg_0$ acts as a bigrading on the root spaces of $\gg$, thus
$\cL^6$ is line-parallelizable, in the sense that each root space gives rise to a $1$-dimensional subbundle of $T\cL^6$. Consequently we have many canonical vector distributions.

More explicitly $\int_x$ and $D_x$ form a pair of raising and lowering operators
\smallskip
 $$
[-D_x,\cdot]:\langle\partial_{y_k}\rangle\mapsto\langle\partial_{y_{k-1}}\rangle\,\operatorname{mod}\langle D_x,
%\partial_{y_3},\dots,
\partial_{y_k}\rangle\;,
\qquad
[\textstyle{\int_x},\cdot]:\langle\partial_{y_k}\rangle\mapsto\langle\partial_{y_{k+1}}\rangle\,
\operatorname{mod}\langle \textstyle{\int_x},
%\partial_{y_0},\dots,
\partial_{y_k}\rangle\;,
 $$
where the formulae have to be understood with the agreement that $\langle\partial_{y_{-1}}\rangle=\langle\partial_{y_{4}}\rangle=0$.
Denoting $X=\langle D_x\rangle$, $I=\langle \int_x\rangle$, and $Y_k=\langle \partial_{y_k}\rangle$, we
get the following integrable vector distributions 
 \begin{gather*}
XI,\; XY_0,\; XY_0Y_1,\; XY_0Y_1Y_2,\; XY_0Y_1Y_2Y_3,\\ 
Y_0Y_1Y_2Y_3,\;
IY_3,\; IY_3Y_2,\; IY_3Y_2Y_1,\; IY_3Y_2Y_1Y_0\;,
 \end{gather*}
where we omitted the direct sum symbol.
Thus the invariant geometric structure on $\cL$ consists of a line-parallelization satisfying the above integrability constraints.

Finally, we can write this structure in local coordinates. We use the coordinates $x,y_0,\dots,y_3$ on $\mathcal{E}^5$
lifted to $\cL^6$, keep the same expressions for the generators $D_x,\partial_{y_k}$ and add a coordinate $t$ and
the generator
$$\int_x=\partial_t+x^2\partial_x+3xy_0\partial_{y_0}
+(3y_0+xy_1)\partial_{y_1}+(4y_1-xy_2)\partial_{y_2}+3(y_2-xy_3)\partial_{y_3}\;.$$
A straightforward computation shows that the symmetry algebra of the above line-parallelism is
precisely $\gg=\mathfrak{gl}_2(\R)\ltimes\R^4$ (this is also independently verified in Maple).

\subsection{Higher dimensional generalizations}\label{sec:5.4}

For CR-codimension $1$ and CR-dimension $k>2$ we have the following constructions
in $\mathbb C^{k+1}$. Consider the rational normal cone $R$ in $\R^{k+1}_x$ 
as the cone over the degree $k$ rational normal curve 
$\{[1:\lambda:\dots:\lambda^k]\}\subset\R\mathbb{P}^k_x$.
Its subsequent tangent varieties $T^qR$ for $q=1,\dots,k-2$ 
are obtained by uniting the osculating spaces at any fixed point (for $q=k-1$ we simply have $T^{k-1}R=\R^{k+1}$).

We consider $T^{k-2}R$ and the tube
 $
T^{k-2}R\times\R^{k+1}_y\subset\mathbb C^{k+1}_z=\R^{k+1}_x\times\R^{k+1}_y\;,
 $ 
which inherits a natural CR-distribution $\cD$ with complex structure $\cJ$. 
It turns out that the non-singular part 
 $$
\cM^{2k+1}=\Sigma\times\R^{k+1}_y
 $$ 
is holomorphically nondegenerate. Here and in the following $\Sigma:=T^{k-2}R\setminus T^{k-3}R$.

 \begin{proposition}
The Freeman sequence of $(\cM^{2k+1},\cD,\cJ)$ decreases by one dimension at each step, so this CR structure is 
$k$-nondegenerate. 
 \end{proposition}

 \begin{proof}
We use local coordinates where $\g:\l\mapsto(1,\lambda,\dots,\lambda^k)$ is a curve in $\R^{k+1}_x$ so that the rational normal cone $R$ 
is parametrized as $(\l,t_0)\mapsto t_0\g(\l)$. Then the tangent variety $T^qR$ is parametrized as
 $$
\psi:\tau=(\l,t_0,t_1,\dots,t_q)\mapsto t_0\g(\l)+t_1\g'(\l)+\dots+t_q\g^{(q)}(\l)\;,
 $$
which we will consider for $q=k-2$. This parametrization covers only a proper open dense subset of 
$T^{k-2}R$, but this is sufficient due to $GL_2(\R)$-equivariancy.

Note that at nonsingular points $\Sigma=T^{k-2}R\setminus T^{k-3}R$ we have
 $$
T_{\psi(\tau)}\Sigma=\langle\gamma,\gamma',\dots,\gamma^{(k-1)}\rangle.
 $$
The first $(k-1)$ terms correspond to $\psi_*\partial_{t_0},\dots,\psi_*\partial_{t_{k-2}}$, whereas
the last term is a combination of them and $\psi_*\partial_\l$. 
Note that $\gamma$ is a radial vector field for $R$ parallelly translated along $T^{k-2}R$.
The other generators are also $t_s$-independent, i.e., they are constant along the foliation
of $\Sigma$ by $(k-1)$-planes.
In particular, for the standard affine connection $\nabla$ on $\R^{k+1}$ we have
$$\nabla_{\psi_*\partial_{t_s}}\g^{(r)}=0\;\;\text{while}\;\;
\nabla_{\psi_*\partial_\l}\g^{(r)}=\g^{(r+1)}\;,$$
for all $0\leq s\leq k-2$. 
The CR-distribution $\cD$ is generated by the vectors $X\in T\Sigma\subset\R^{k+1}_x$ and their counterparts $JX\in\R^{k+1}_y$, which respectively have the form 
$$X=\sum a_i^j(x)\partial_{x_j}\;\;\text{and}\;\;JX=\sum a_i^j(x)\partial_{y_j}\;.$$ 
Therefore the commutator of such vector fields $X$ and $Y$ is 
$\nabla_XY\in\R^{k+1}_y$ for the above trivial affine connection $\nabla$.
This allows to easily compute the terms of the Freeman filtration. 

Let us denote 
$$Z_s=\frac12(\g^{(s)}_x-i\g^{(s)}_y)$$
for all $0\leq s\leq k-1$, where the subscripts $x, y$ indicate to which of the two components in $\mathbb C^{k+1}_z=\R^{k+1}_x\times\R^{k+1}_y$ the vector belongs. Then
 $$
\cF^s_{10}=\langle Z_0,Z_1,\dots,Z_{k-s-2}\rangle\;,
 $$
so that $\cF^{k-2}_{10}\neq0$ and $\cF^{k-1}_{10}=0$.
 \end{proof}
 
Furthermore the affine automorphism group of $\cM^{2k+1}$ is clearly 
$G=GL_2(\mathbb R)\ltimes S^k\mathbb R^2$
and we expect that this is equal to the entire automorphism group, namely
 \begin{equation}\label{conj}
\operatorname{Aut}\big(\cM^{2k+1},\cD,\cJ\big)=GL_2(\R)\ltimes S^k\mathbb R^2.
 \end{equation}
It is important here that $k>2$. In fact, for $k=2$, the rational normal cone is a quadric, 
the null cone for a Lorentzian $3$-dimensional metric, and this results in a bigger
symmetry algebra $\mathfrak{so}(2,3)$ acting on the tube over the future light cone, see \cite{IZ,MS}.
Note that the automorphism group of $\big(\cM^5,\cD,\cJ\big)$ is $GL_2(\R)\ltimes S^2\mathbb R^2$,
but $\cM^5$ is densely embedded in the homogeneous model in a complex projective quadric
with automorphism group $SO(2,3)^o$.

In fact, for $k=2$, the rational normal cone is a quadric, 
the null cone for a Lorentzian $3$-dimensional metric, and this results in a bigger
automorphism group, the conformal group $SO(2,3)$ acting on the tube over the future light cone, see \cite{IZ,MS}. 
If the above conjecture \eqref{conj} is true then  the model will be almost simply transitive 
for $k=4$ and inhomogeneous for $k\geq 5$. 
(We remark that for $k=4$ our model here is locally equivalent to that of Example \ref{ex:4nondeg}. 
Indeed, identifying points in $\R^5=S^4\R^2$ with the coefficients of a quartic, one can show that 
the second tangent of the rational normal curve passes through the point $(1,1,1,0,0)$, 
which lies on the same $GL_2(\R)$-orbit as $(0,0,1,1,1)$.) 
For a construction of homogeneous $k$-nondegenerate 
CR manifolds in dimension $2k+3$, we refer to \cite{La,MMN}.

Relations to other geometries, like higher codimension CR tubes and higher order ODEs also generalize.
In particular, $\gg=\ggl_2(\R)\ltimes S^k\mathbb R^2$ is the
symmetry algebra of the trivial ODE $y^{(k+1)}(x)=0$.
Again the case $k=2$ is special: the symmetry algebra is $\mathfrak{sp}_4(\R)=\mathfrak{so}(2,3)$.
Thus we have an affine bundle $\cM^{2k+1}\to\mathcal{E}^{k+2}$ of rank $(k-1)$ over the equation
manifold of the trivial ODE for every $k\ge2$, and the action of $\gg$ is projectable.
\bigskip\par

\end{document}